\documentclass[11pt,reqno]{amsart}
\usepackage{amsmath,amsthm,amssymb,mathrsfs,stmaryrd,color}
\usepackage[all]{xy}
\usepackage{url}

\usepackage{geometry}
\usepackage[utf8]{inputenc}
\usepackage[T1]{fontenc}

\usepackage{enumitem}
\usepackage[colorlinks=true,hyperindex, linkcolor=magenta, pagebackref=false, citecolor=cyan,pdfpagelabels]{hyperref}
\usepackage[capitalize]{cleveref}
\setlist[enumerate]{itemsep=2pt,parsep=2pt,before={\parskip=2pt}}

\usepackage{relsize}
\usepackage[bbgreekl]{mathbbol}
\usepackage{amsfonts}
\DeclareSymbolFontAlphabet{\mathbb}{AMSb}
\DeclareSymbolFontAlphabet{\mathbbl}{bbold}
\newcommand{\Prism}{\mathbbl{\Delta}}

\newcommand{\cosimp}[3]{\xymatrix@1{#1 \ar@<.4ex>[r] \ar@<-.4ex>[r] & {\ }#2 \ar@<0.8ex>[r] \ar[r] \ar@<-.8ex>[r] & {\ } #3 \ar@<1.2ex>[r] \ar@<.4ex>[r] \ar@<-.4ex>[r] \ar@<-1.2ex>[r] & \cdots }}

\newcommand{\colim}{\mathop{\mathrm{colim}}}

\newcommand{\adjunction}[4]{\xymatrix@1{#1{\ } \ar@<0.3ex>[r]^{ {\scriptstyle #2}} & {\ } #3 \ar@<0.3ex>[l]^{ {\scriptstyle #4}}}}

\newtheorem{theorem}{Theorem}[section]
\newtheorem{proposition}[theorem]{Proposition}
\newtheorem{lemma}[theorem]{Lemma}
\newtheorem{corollary}[theorem]{Corollary}

\theoremstyle{definition}
\newtheorem{definition}[theorem]{Definition}
\newtheorem{question}[theorem]{Question}
\newtheorem{remark}[theorem]{Remark}

\newtheorem{example}[theorem]{Example}

\newtheorem{notation}[theorem]{Notation}

\newtheorem{construction}[theorem]{Construction}

\newtheorem{claim}[theorem]{Claim}

\newtheorem*{convention*}{Conventions}

\newcommand{\RH}{\mathrm{RH}}

\newcommand{\pfd}{\mathrm{perfd}}

\newcommand{\perf}{\mathrm{perf}}

\begin{document}
\title{Cohen--Macaulayness of absolute integral closures}
\author{Bhargav Bhatt}
\address{University of Michigan, Ann Arbor}
\date{}
\maketitle

\begin{abstract}
We prove that, modulo any power of a prime $p$, the absolute integral closure of an excellent noetherian domain is Cohen--Macaulay. A graded analog is also established, yielding variants of Kodaira vanishing ``up to finite covers'' in mixed characteristic. Our main tools are (log) prismatic cohomology (which yields a Frobenius action in mixed characteristic) and the $p$-adic Riemann-Hilbert functor for constructible \'etale $\mathbf{F}_p$-sheaves on varieties over a $p$-adic field (which almost controls perfectified prismatic cohomology).
\end{abstract}

\setcounter{tocdepth}{1}
\tableofcontents

\newpage

\section{Introduction}
\label{sec:Intro}

Fix a prime number $p$.  Given an integral domain $R$, we use the term {\em absolute integral closure} to describe the integral closure of $R$ in an algebraic closure of its fraction field (following \cite{ArtinJoins}); such a closure is unique up to (non-unique) isomorphism, commutes with localization, and is typically denoted $R^+$. There is an evident generalization to integral schemes.

\subsection{Cohen--Macaulayness of $R^+$}

Recall that Cohen--Macaulay (CM) modules $M$ over a noetherian ring $R$ are exceptionally well-behaved $R$-modules; for instance, they are flat over any noetherian normalization of $\mathrm{Spec}(R)$ and have extremely simple local cohomology. The main commutative algebra result of this paper is following, stating roughly that absolute integral closures of excellent noetherian domains are CM in mixed characteristic.

\begin{theorem}[see Theorem~\ref{AICVanishLocalCoh}, Corollary~\ref{CMRegSeqAIC}]
\label{RPlusCMIntro}
Let $R$ be an excellent domain with an absolute integral closure $R^+$. Then the $R/p^nR$-module $R^+/p^nR^+$ is Cohen--Macaulay\footnote{In fact, if $R$ is a complete noetherian local domain, we show that the $p$-adic completion $\widehat{R^+}$ is Cohen--Macaulay over $R$ in see Corollary~\ref{CMPlusComp}. The hypothesis on $R$ in the previous sentence turns out to not be necessary: the recent paper \cite{BMPSTWW} explains how to deduce directly from Theorem~\ref{RPlusCMIntro} that the $p$-adic completion $\widehat{R^+}$ is Cohen--Macaulay for any local excellent domain $R$. We also further note that $\widehat{R^+}$ was recently shown to be a domain by Heitmann \cite{HeitmannRPlusDomain} under mild extra assumptions on $R$ (such as analytic irreducibility).}
 for all $n \geq 1$.
\end{theorem}

Despite involving the large ring $R^+$, Theorem~\ref{RPlusCMIntro} is essentially a finitistic statement. In fact, in the key special case where $R$ is local of residue characteristic $p$ and $n=1$, Theorem~\ref{RPlusCMIntro} amounts to the following concrete assertion: for any system of parameters $x_1,...,x_d$ in $R/p$ and any relation $\sum_i r_i x_i = 0$ in $R/pR$, there exists a finite injective map $R \hookrightarrow S$ of domains such that the relation $\sum_i r_i x_i = 0$ becomes trivial in $S/pS$, i.e., is an $S/pS$-linear combination of trivial Koszul relations. Note that it is necessary to work modulo a power of $p$ (or after $p$-adic completion) in Theorem~\ref{RPlusCMIntro}: if $R$ is an excellent noetherian domain containing $\mathbf{Q}$ and $\dim(R) \geq 3$, then $R^+$ is {\em never} CM over $R$ due to trace obstructions.

Theorem~\ref{RPlusCMIntro} has some standard consequences, e.g., it implies splinters are CM in mixed characteristic (Corollary~\ref{splintersCM}) and that $R^+/\sqrt{pR^+}$ is CM over $R/p$ (Corollary~\ref{LyubeznikConj}), answering a question highlighted by Lyubeznik \cite{LyubeznikRPlus}. Moreover, it yields (Remark~\ref{rmk:WeakFunNonLoc}) a  simple and explicit construction of ``weakly functorial  CM algebras'' (previously shown to exist by Andr\'e \cite{AndreWeakFun} and Gabber \cite{GabberMSRINotesBCM} via rather indirect constructions), which is known to imply a many of the ``homological conjectures'' in commutative algebra; see \cite{HochsterHomologicalMSRI} for a recent survey on this question.

\begin{remark}[What was previously known?]
If $\dim(R) \leq 2$, Theorem~\ref{RPlusCMIntro} follows from the Auslander-Buchsbaum formula (and then the result holds true even without reducing modulo $p^n$). Next, if $R$ is an $\mathbf{F}_p$-algebra, then Theorem~\ref{RPlusCMIntro} is the main result of Hochster-Huneke's \cite{HHCM} (see also \cite{HHCMSurvey,HHCMSurvey2,HunekeSurvey} for expositions focusing on applications); this result underlies and conceptualizes many basic results in $F$-singularity theory, was inspired by the theory of tight closure \cite{HHTightJAMS}, and also formed the initial motivation for this paper. Finally, if $\dim(R) = 3$ and $R$ has mixed characteristic, the image of Theorem~\ref{RPlusCMIntro} in ``almost mathematics over $R^+$" (in the classical sense of Faltings, see \cite{FaltingsJAMS,FaltingsAlmostEtale,GabberRamero}) was the main result of Heitmann's \cite{HeitmannDSC3}; see also \cite{HeitmannExt} for a stronger almost vanishing theorem in dimension $3$ also covered by Theorem~\ref{RPlusCMIntro}.
\end{remark}

\subsection{Kodaira vanishing up to finite covers}

There is a well-known analogy between projective geometry and local algebra, e.g., the global cohomological properties of a projective variety $X \subset \mathbf{P}^n$ are faithfully reflected in the local cohomological properties of its affine cone $Y \subset \mathbf{A}^{n+1}$ near the vertex $0 \in Y$. This analogy suggests that Theorem~\ref{RPlusCMIntro} ought to have a global variant. This is indeed the case; in fact, the global analog can be regarded as an ``up to finite covers'' variant of Kodaira vanishing  in mixed characteristic, and states the following (giving a mixed characteristic analog of the graded main theorem in \cite{HHCM} when translated to algebra):

\begin{theorem}[see Theorem~\ref{KodairaFiniteCover}, Remark~\ref{rmk:ProperDVR}]
\label{XPlusCMIntro}
Let $V$ be an excellent $p$-henselian and $p$-torsionfree DVR. Let $X/V$ be a proper flat integral scheme of relative dimension $d$, and let $L \in \mathrm{Pic}(X)$ be semiample and big (e.g., $L$ could be ample). For any $n \geq 1$,  there is a finite surjective map $\pi:Y \to X$ such that the pullbacks
\begin{enumerate}
\item $H^*(X, L^{b})_{tors} \to H^*(Y, \pi^* L^{b})_{tors}$ for fixed $b < 0$.
\item $H^{< d}(X_{p^n=0}, L^b) \to H^{< d}(Y_{p^n=0}, \pi^* L^b)$ for all $b < 0$.
\item $H^{> 0}(X_{p^n=0}, L^a) \to H^{> 0}(Y_{p^n=0}, \pi^* L^a)$ for all $a \geq 0$.
\end{enumerate}
are the $0$ map. (Here $X_{p^n=0} := X \otimes_V V/p^n$ denotes the mod $p^n$ fibre of $X/V$, and similarly for $Y$.)
\end{theorem}

Parts (1) and (2) of Theorem~\ref{XPlusCMIntro} imply that spurious cohomology classes --- ones that should not be there if Kodaira vanishing were true --- are annihilated by pullback to a finite cover, explaining why we view Theorem~\ref{XPlusCMIntro} as a variant of Kodaira vanishing in mixed characteristic. This theorem in fact admits a generalization to the relative case where the base $V$ is replaced by any finitely presented algebra $T$ over $V$ (Theorem~\ref{VanishingFiniteCov}); the latter is quite useful in applications (such as those mentioned in Remark~\ref{rmk:MMP}) and was recently extended to all $p$-henselian excellent domains $T$ in \cite[\S 3]{BMPSTWW} by an approximation argument.

\begin{remark}[What was previously known?]
As far as we know, Theorem~\ref{XPlusCMIntro} is new even when $d=1$. The equal characteristic $p$ analog of Theorem~\ref{XPlusCMIntro} was for $L$ ample corresponds to the graded main theorem from \cite{HHCM}, and the general characteristic $p$ case was the subject of  \cite{ddscposchar}. In mixed characteristic, the special case of Theorem~\ref{XPlusCMIntro} (3) when $a=0$ was previously known, at the expense of relaxing finite covers to alterations, by \cite{Bhattpdiv} and formed a key input to the $p$-adic Poincar\'e lemma\footnote{Curiously, while the Poincar\'e lemma of \cite{Beilinsonpadic} was used in {\em loc.\ cit.} to give a simple proof of Fontaine's $C_{dR}$ conjecture in $p$-adic Hodge theory, our proof of Theorem~\ref{XPlusCMIntro} relies on modern advances in $p$-adic Hodge theory.} of  \cite{Beilinsonpadic}.   
\end{remark}

\begin{question}
Theorem~\ref{RPlusCMIntro} (resp., Theorem~\ref{XPlusCMIntro}) imply that local (resp., coherent) cohomology classes can be annihilated by finite covers in some situations. Our method of proof, which relies largely on studying structural features of the absolute integral closure directly and is sketched below, does not directly shed light on how one might construct such finite covers in examples. It would thus be interesting to give such explicit constructions in some specific instances, e.g., in the special case of Theorem~\ref{XPlusCMIntro} (3) when $d=1$. 
\end{question}

\subsection{Some further remarks on the main theorems}

\begin{remark}[A comment on excellence]
The excellence of a noetherian local ring ensures a tight connection between the ring and its completion. Without this hypothesis, pathologies can occur,  e.g., by \cite{Heitmann-completion-UFD}, there exist noetherian local UFDs with non-equidimensional completions\footnote{See \cite{LechBad,OgomaBad,Nagata} for earlier examples as well as \cite[Tag 02JE]{StacksProject} for an exposition of a key example by Nagata.}, so the excellence assumption cannot be  dropped entirely from Theorem~\ref{RPlusCMIntro} (see Example~\ref{ExcNecc}). However, by analogy with the improvement to \cite{HHCM} in \cite{QuyRPlus}, one might wonder if Theorem~\ref{RPlusCMIntro} holds true for any noetherian local domain which is the quotient of a CM ring (this property generalizes excellence by \cite[Corollary 1.2]{KawasakiCM}); see also Remark~\ref{MellowDualizingDrop}.
\end{remark}

\begin{remark}[Historical comment]
As mentioned above, certain consequences of Theorem~\ref{RPlusCMIntro}, such as Lyubeznik's question from \cite{LyubeznikRPlus} (see also \cite{LyubeznikMSRITalk}) discussed in Corollary~\ref{LyubeznikConj}, have attracted some attention in the literature. In fact, as Hochster kindly informed us, the entirety of Theorem~\ref{RPlusCMIntro} has often been discussed by experts  since \cite{HHCM} appeared. However, as best as we could determine, this result has never been explicitly conjectured in print, most likely due to the lack of supporting evidence, e.g., \cite[\S 10]{RobertsHomConjSurvey} raises the question addressed by Theorem~\ref{RPlusCMIntro} only in dimension $3$, and \cite[paragraph 2, page 11]{HunekeSurvey} discusses the general question with a slight suspicion of a negative answer; even the very recent  papers \cite[\S 3.6]{HochsterOpenProblemsLyubeznik} and \cite[Problem 2]{ShimomotoIntegerPerfAndre} only ask about variants of the almost Cohen--Macaulayness (as in Theorem~\ref{ACMIntro} below).
\end{remark}

\begin{remark}[Recent applications of Theorems~\ref{RPlusCMIntro} and \ref{XPlusCMIntro}]
\label{rmk:MMP}
Theorem~\ref{XPlusCMIntro} (as well as a relative variant, see \S \ref{rmk:RelativeGradedBCM}) was used in recently in \cite{BMPSTWW, TakamatsuYoshikawaMMP} to establish the minimal model program in dimension $\leq 3$ in mixed characteristic; this theory was then used in \cite{HaconLamarcheSchwede} to establish global generation results for mixed characteristic analogs of multiplier ideals. 
\end{remark}

\subsection{Some key ingredients}

Let us briefly outline our proof of Theorem~\ref{RPlusCMIntro} and introduce some key players in this paper along the way; a more detailed outline can be found in \S \ref{ss:OutlineProof}. By standard arguments, we reduce to checking Theorem~\ref{RPlusCMIntro} for essentially finitely presented algebras over a $p$-adic DVR. We then attempt to imitate the cohomological approach to the main result of \cite{HHCM} given by Huneke-Lyubeznik \cite{HunekeLyubeznik}. To explain this, fix an excellent noetherian local domain $(R,\mathfrak{m})$ over $\mathbf{F}_p$ admitting a dualizing complex. The CM property for $R^+$ is essentially the assertion that each class in $H^i_{\mathfrak{m}}(R)$ with $i < \dim(R)$ can be annihilated by passing to finite covers of $R$. In \cite{HunekeLyubeznik}, such annihilation of cohomology classes takes place in two steps: Frobenius actions are used to kill certain finite length submodules in $H^{i}_{\mathfrak{m}}(R)$ by finite covers, and an induction on dimension based on Grothendieck duality is used  to reduce to the finite length situation by passage to a finite cover.

In imitating the above argument in  mixed characteristic, the first obstacle is  the lack of a Frobenius on the structure sheaf $\mathcal{O}$. We resolve this problem by replacing $\mathcal{O}$ with the prismatic structure sheaf $\Prism$ from \cite{BMS1,BMS2, BhattScholzePrism} as the latter was designed to carry a Frobenius. However, this maneuver comes at a price: the sheaf $\Prism$ has an extra ``arithmetic parameter'' (necessary to accommodate a Frobenius action).  Consequently, a new obstacle is encountered in the induction step to handle this extra parameter. In fact, jumping this hurdle turns out to be fairly close to proving the target theorem in almost mathematics over $R^+$. While approaching this problem with the current strategy seems hopelessly circular, this is exactly the problem solved recently in \cite{BhattLuriepadicRH1} by entirely different methods. A simplified variant of what we  need is:

\begin{theorem}[\cite{BhattLuriepadicRH1}, see Theorem~\ref{ACMaic}]
\label{ACMIntro}
Let $R$ be an excellent noetherian domain with an absolute integral closure $R^+$. Then the $R/p^nR$-module $R^+/p^nR^+$ is almost Cohen--Macaulay i.e., the $R/p^nR$-module $\sqrt{pR^+}/p^n \sqrt{pR^+}$ is Cohen--Macaulay for all $n \geq 1$. 
\end{theorem}

\begin{remark}[Almost vs honest Cohen--Macaulayness]
Theorem~\ref{ACMIntro} generalizes \cite{HeitmannDSC3} to arbitrary dimension. Despite ``merely'' being a statement in almost mathematics, it suffices for some applications, e.g., it yields a new proof of Hochster's direct summand conjecture, which was recently established  by Andr\'e \cite{AndreDSC,AndrePerfectoidAbhyankar}; unlike Andr\'e's or other previously known proofs such as   \cite{BhattDSC,HeitmannMaTor},  this proof does not rely on  the  Riemann extension theorem (or closely related techniques). For other applications, however, the gap between Theorem~\ref{ACMIntro} and Theorem~\ref{RPlusCMIntro} is substantial, e.g., one cannot prove splinters are CM this way, and the almost variant of Theorem~\ref{XPlusCMIntro} is not as useful. More conceptually, from the perspective of the Riemann-Hilbert functor \cite{BhattLuriepadicRH1} (see also \S \ref{sec:RHreview}), the almost vanishing result in Theorem~\ref{ACMIntro} is essentially a reflection in algebra of a statement concerning the topology of the generic fibre $\mathrm{Spec}(R[1/p])$, while Theorem~\ref{RPlusCMIntro} wrestles with the entirety of $\mathrm{Spec}(R)$.
\end{remark}

In implementing the strategy sketched above for proving Theorem~\ref{RPlusCMIntro}, we also need to ensure that the intervening prismatic cohomology groups are related to their perfections in a controlled fashion. Such control is provided by the isogeny theorem for prismatic cohomology. As the isogeny theorem only holds true for (log) smooth schemes, we work with semistable schemes and log prismatic cohomology\footnote{The name ``logarithmic prismatic cohomology'' was not used in \cite{CesnavicusKoshikawa}: the latter gave a logarithmic variant of \cite{BMS1}, predating the prismatic theory as eventually developed in \cite{BhattScholzePrism}. It is quite likely that our arguments can also be run using the log prismatic theory of the very recent preprint \cite{KoshikawaLogPrisms}.} as developed by \v{C}esnavi\v{c}ius-Koshikawa \cite{CesnavicusKoshikawa} to implement the above outline. Alteration theorems of de Jong \cite{deJongAlterations} ensure an abundant supply of such schemes. At the end, to ensure that we are still computing the object of interest (namely, the local cohomology of $R^+$ rather than that of an alteration), we need the following result from \cite{BhattLuriepadicRH1}, ensuring that higher $\mathcal{O}$-cohomology of proper maps can be annihilated $p$-adically by finite covers:

\begin{theorem}[\cite{BhattLuriepadicRH1}, see Theorem~\ref{KillCohMixed}]
\label{CohCohKill}
Let $f:X \to \mathrm{Spec}(R)$ be a proper morphism of $p$-torsionfree noetherian schemes. Then there exists a finite cover $\pi:Y \to X$ such that $\pi^*:R\Gamma(X, \mathcal{O}_X/p) \to R\Gamma(Y, \mathcal{O}_Y/p)$ factors over $H^0(Y, \mathcal{O}_Y)/p \hookrightarrow H^0(Y, \mathcal{O}_Y/p) \to R\Gamma(Y,\mathcal{O}_Y/p)$ compatibly with the maps from $R/p$.
\end{theorem}

\begin{remark}
Theorem~\ref{CohCohKill} extends some results in \cite{Bhattpdiv,Beilinsonpadic} by improving alterations to finite covers, and allowing torsion classes as well as higher dimensional bases. Notably, it implies the equivalence of splinters and derived splinters for noetherian rings  in mixed characteristic, extending \cite{ddscposchar} to this setting.
\end{remark}

The proofs of Theorem~\ref{ACMIntro} and Theorem~\ref{CohCohKill} in \cite{BhattLuriepadicRH1} rely on the Riemann-Hilbert functor, whose construction was the primary objective of \cite{BhattLuriepadicRH1}.  This functor attaches coherent objects in mixed characteristic to \'etale $\mathbf{F}_p$-sheaves on algebraic varieties in characteristic $0$. This construction respects various geometric operations, can be regarded as an integral and relative variant of Fontaine's $D_{HT}(-)$ functor from $p$-adic Hodge theory (see \cite[\S 1.5]{Fontainepadicrepr}), and agrees with the perfectoidization functor from \cite{BhattScholzePrism} for constant sheaves.	Granting basic properties of this functor, both Theorem~\ref{ACMIntro} and Theorem~\ref{CohCohKill}  reduce to relatively easy statements about the behaviour of \'etale $\mathbf{F}_p$-sheaves (especially the perverse $t$-structure) under finite covers; after recounting the basic properties of the Riemann-Hilbert functor, we sketch these proofs  in \S \ref{sec:RHreview}. 

The proof of Theorem~\ref{XPlusCMIntro}  broadly follows the same outline as that of Theorem~\ref{RPlusCMIntro} once one has an almost Cohen--Macaulayness result replacing Theorem~\ref{ACMIntro}; the key to proving this analog is to relate the graded absolute integral closure of the section ring $\oplus_{n \geq 0} H^0(X,L^n)$ to a perfectoid formal scheme over $X^+$ obtained by taking an inverse limit of the total spaces of a compatible system $\{L^{-1/n}\}_{n \geq 1}$ of roots of $L^{-1}$ on $X^+$.

\subsection{Notation and conventions}
\label{GlobalNotation}

Let us spell out some conventions that we adopt in this paper.

{\em Regular sequences:} Given a commutative ring $R$, a sequence $\underline{x}$ of elements $x_1,...,x_r \in R$ and an $R$-module $M$, we say that $\underline{x}$ is regular on $M$ if $x_1$ is a nonzerodivisor on $M$ and $x_k$ is a nonzerodivisor on $M/(x_1,...,x_{k-1})M$ for $2 \leq k \leq r$. In particular, we do not demand that $M/(x_1,...,x_r)M = 0$. 

{\em Cohen--Macaulay modules:} A (not necessarily finitely generated) module $M$ over a noetherian local ring $(S,\mathfrak{m})$ is called Cohen--Macaulay if $M/\mathfrak{m}M \neq 0$ and every system of parameters on $S$ is a regular sequence on $M$. A module $M$ over a general noetherian ring $S$ is called Cohen--Macaulay if it is Cohen--Macaulay after localization at every prime of $S$ in the previous sense.

{\em Vanishing loci and quotients:} For a scheme $X$ and an element $f \in H^0(X, \mathcal{O}_X)$, we write $X_{f=0} \subset X$ for the closed subscheme defined by $f=0$; we will only use this when $f$ is a nonzerodivisor. Given $K \in D_{qc}(X)$, we shall write $K/f$ for the cone of multiplication by $f$ on $K$. When $K$ is an honest quasi-coherent sheaf, we shall write $K/fK$ for the usual quotient in quasi-coherent sheaves, so $K/fK = H^0(K/f)$.

{\em Derived categories and completions:} Given a commutative ring $R$, we shall use the terms ``connective'', ``coconnective'' and ``discrete'' to describe objects in $D(R)$ that lie in the full subcategories $D^{\leq 0}(R)$, $D^{\geq 0}(R)$, and $D^{\leq 0}(R) \cap D^{\geq 0}(R) = \mathrm{Mod}_R$ respectively. Given a finitely generated ideal $I \subset R$, we shall use the notions of derived $I$-completions and $I$-complete flatness as in \cite{BhattScholzePrism}. In particular, all completions are interpreted in the derived sense unless explicitly otherwise specified.

{\em Complete objects over a perfectoid ring:} 
We shall use the basic theory of perfectoid rings as in \cite[\S 3]{BMS1} as well as their connection with perfect prisms as in \cite[Theorem 3.10]{BhattScholzePrism} (see also \cite[\S 3]{CesnavicusScholzePurity} and \cite[Lecture IV]{BhattEilenberg}). Fix a perfectoid ring $R$. For objects in $D(A_{\inf}(R))$, unless otherwise specified, completeness is always interpreted to be derived $(\ker(\theta),p)$-completeness. Thus, given an $A_{\inf}(R)$-algebra $B$, the notation $D_{comp}(B) \subset D(B)$ describes to the full subcategory of all derived $(\ker(\theta),p)$-complete objects.   In particular, $D_{comp}(R) \subset D(R)$ is the full subcategory of derived $p$-complete objects.

{\em Almost mathematics over a perfectoid ring and its prism:}  Say $R$ is a perfectoid ring. All occurrences of almost mathematics over $R$ are always in the ``standard'' context that we now recall. The ideal $I := \sqrt{pR} \subset R$ satisfies $I \otimes_R^L I \simeq I \widehat{\otimes}_R^L I \simeq I$. Consequently, for any $R$-algebra $S$, restriction of scalars gives a fully faithful embedding $D(S \otimes_R^L R/I) \to D(S)$, and the quotient $D(S) \to D(S)^a := D(S)/D(S \otimes_R^L R/I)$ is called the {\em almost derived category} of $S$; for $M \in D(S)$, write $M^a \in D(S)^a$ for its image.  The localization functor $D(S) \to D(S)^a$ has both left and right adjoints given by $M^a \mapsto (M^a)_! := I \otimes_R^L M$ and $M^a \mapsto (M^a)_* := \mathrm{RHom}_R(I,M)$ respectively.   One employs similar notation for the full subcategory $D_{comp}(S) \subset D(S)$ of derived $p$-complete objects.  Finally, we define $D_{comp}(A_{\inf}(R))^a := D_{comp}(A_{\inf}(R)) / D_{comp}(A_{\inf}(R/I))$; this makes sense as $D_{comp}(A_{\inf}(R/I)) \to D_{comp}(A_{\inf}(R))$ is again fully faithful, and the localization functor $D_{comp}(A_{\inf}(R)) \to D_{comp}(A_{\inf}(R))^a$ again has both left and right adjoints.  

{\em Commutative algebra over a finitely presented algebra over a valuation ring:} 
To avoid subtle  complications stemming from the arithmetic of $p$-adic fields, we often find it convenient to work over an algebraic closure of a $p$-adic field. This inevitably leads to non-noetherian schemes of the following sort: $V$ is an excellent $p$-henselian and $p$-torsionfree DVR with absolute integral closure $\overline{V}$ and $X/\overline{V}$ is a finite type scheme.  Let us recall some standard techniques for working with such $X$'s. First, if $X$ is flat over $\overline{V}$ (i.e., $\mathcal{O}_X$ is torsionfree over $V$), then $X$ is automatically finitely presented by Nagata's theorem (\cite[Tag 053E]{StacksProject}). Assume for the rest of this paragraph that we are in this situation. Then $X$ descends to a finitely presented $W$-scheme $Y$ for some finite extension $V \subset W$ of DVRs contained in $\overline{V}$. In particular, we have $X \simeq \lim_{W'} Y_{W'}$, where the limit runs over all finite extensions $W'/W$ of DVRs contained in $\overline{V}$. Standard limit arguments (e.g., as in \cite[Tag 01YT]{StacksProject}) help understand the non-noetherian scheme $X$ via the noetherian schemes $Y_{W'}$. For example, any such $X$ is a coherent scheme since the transition maps in the tower $\{Y_{W'}\}$ are flat, and similarly for $X_{p=0} = \lim_{W'} Y_{p=0,W'/p}$. We shall use the relative dualizing complex $\omega^\bullet_{X/\overline{V}}$ for $X/\overline{V}$ defined in \cite[Tag 0E2S]{StacksProject}; note that this arises via base change from the usual relative dualizing complex for $Y/W$ by \cite[Tag 0E2Y]{StacksProject}, and agrees with $f^! \mathcal{O}_{\mathrm{Spec}(\overline{V})}$ if the structure map $f:X \to \mathrm{Spec}(\overline{V})$ is proper by \cite[Tag 0E2Z]{StacksProject}.

{\em Ind-objects and pro-objects:} The language of ind-objects and pro-objects in a category helps efficiently package many of our arguments (e.g., those involving the collection of all finite or proper covers of a ring or a scheme).  When applied to diagrams in the derived category, these are always interpreted in the $\infty$-categorical sense (to avoid mistakes, e.g., we often use that taking cohomology commutes with filtered colimits).

{\em \'Etale sheaves:} We shall use the classical theory of \'etale $\mathbf{Z}/p^n$-sheaves on a scheme $X$ as well as the constructible derived category $D^b_{cons}(X, \mathbf{Z}/p^n)$ of such sheaves as well as full unbounded derived category $D(X, \mathbf{Z}/p^n)$ of all \'etale $\mathbf{Z}/p^n$-sheaves. To avoid excess notation, given a scheme $X$, an object $F \in D(X, \mathbf{Z}/p^n)$, and a (pro-)open subscheme $U \subset X$, we say that $F$ is $*$-extended from $U$ if $F \simeq Rj_* (F|_U)$ via the natural map, where $j:U \to X$ is the inclusion; similarly for the property of being $!$-extended from $U$.  Finally, in the context of varieties over an algebraically closed field of characteristic $0$, we shall use the perverse $t$-structure on $D^b_{cons}(X,\mathbf{Z}/p^n)$, in the sense of  middle perversity \cite[\S 4]{BBDG}.

\subsection{Acknowledgements} 
It is hopefully obvious that the methods and results of this paper rely crucially on my joint work \cite{BhattLuriepadicRH1} with Jacob Lurie. The argument outlined above for proving Theorem~\ref{RPlusCMIntro} emerged in various discussions with Peter Scholze over a long period of time; indeed, the prospect of approaching Theorem~\ref{RPlusCMIntro} by tilting (as in Scholze's thesis \cite{ScholzePerfectoidSpaces}) to \cite{HHCM} was one of the first questions we ever discussed, back in 2011.  Both Lurie and Scholze have declined to be co-authors on this article. I am extremely grateful to both of them for discussions, encouragement, and their friendship over the years. 

I am deeply indebted to  Johan de Jong and Mel Hochster: besides countless conversations and their encouragement over years, the question addressed here first came up in my thesis work supervised by de Jong and concern areas of mathematics where Hochster (with collaborators, notably Huneke) has had a transformative impact.  I would like thank my colleagues on the ``BuBbLeSZ" team (Manuel Blickle, Gennady Lyubeznik, Anurag Singh, Wenliang Zhang) for our multiple collaborations and  conversations. In particular, the goal of one of these projects, \cite{BBLSZRHC}, is to give  soft proofs of various results in equal characteristic commutative algebra using the classical Riemann-Hilbert correspondence (over $\mathbf{C}$ and $\mathbf{F}_p$ respectively); working on \cite{BBLSZRHC} provided a very useful guide for contemplating the shape of a mixed characteristic analog. I would also like to thank Bogdan Zavyalov: his proof of Poincar\'e duality for rigid spaces (inspired by some ideas of Gabber), which he kindly explained during his stay at Michigan in Fall 2019, was an important source of inspiration for the duality compatibility in \cite{BhattLuriepadicRH1} that plays an essential role in this paper. 

Finally, many thanks are due to Dan Abramovich, K\k{e}stutis \v{C}esnavi\v{c}ius, Rankeya Datta, Ofer Gabber, Ray Heitmann, Wei Ho, Craig Huneke, Linquan Ma, Zsolt Patakfalvi, Karl Schwede, Kazuma Shimomoto, Paul Roberts, Kevin Tucker, Joe Waldron, Jakub Witaszek  and Wenliang Zhang for a number of helpful exchanges during the preparation of this paper, and to Yves Andr\'e, Sasha Beilinson, David Hansen,  Luc Illusie, Srikanth Iyengar and Akhil Mathew for inspiring discussions related to the mathematics in this paper. I was supported by NSF DMS (\#1801689, \#1952399, \#1840234), a Packard fellowship, and the Simons Foundation (\#622511).

\newpage
\section{(Ind-)Cohen--Macaulay complexes}
\label{sec:CM}

The purpose of this section is twofold. First, we formulate a notion of Cohen--Macaulayness for an object of the derived category in terms of the behaviour of local cohomology (Definition~\ref{CMDerived}) and relate this to classical Cohen--Macaulayness (in the case of modules) as defined in terms of regular sequences (Lemma~\ref{CMLocCohCrit}). Secondly, we formulate a notion of Cohen--Macaulayness for certain ind-objects (Definition~\ref{IndCMDef}) and prove a crucial lemma, essentially borrowed from \cite{HunekeLyubeznik}, stating roughly that having this property in dimension $< n$ implies a finiteness statement in dimension $n$ (Lemmas~\ref{IndCMPunctured} and \ref{IndCMPuncturedNN}).

\subsection{CM modules via local cohomology}

The condition of being Cohen--Macaulay for a module over a noetherian local ring involves a condition on systems of parameters. For our applications, it will be convenient to use a local cohomological variant of this condition instead, as that passes easily to derived category; the following definition captures what we need, and additionally makes sense in non-noetherian contexts as well.

\begin{definition}[CM complexes]
\label{CMDerived}
Given a finite dimensional scheme $X$, an object $M \in D_{qc}(X)$ is called {\em cohomologically CM} if $R\Gamma_x(M_x) \in D^{\geq \dim(\mathcal{O}_{X,x})}$ for all $x \in X$. 
\end{definition}

Even though the above definition is formulated for all finite dimensional schemes, we shall only use it when $X$ is topologically noetherian. In the rest of this section, we make some elementary remarks on this notion and show why it captures the condition on parameters in good cases (Corollary~\ref{CatCohCM}).

\begin{remark}
Definition~\ref{CMDerived} is obviously not equivalent to the usual definition of Cohen--Macaulayness when specialized to modules, e.g., the $0$ module is always cohomologically CM but is not Cohen--Macaulay (as it violates the support property). Hence, it might be better to dub the notion in Definition~\ref{CMDerived} as ``weakly cohomologically CM" but we do not do so. We also remark that a closely related notion of Cohen--Macaulay complexes was introduced by Roberts in \cite[Definition 2]{RobertsCMGR}.
\end{remark}

\begin{remark}
\label{BiequiCM}
Say $X$ is a biequidimensional noetherian scheme of dimension $n$, i.e., $X$ has finite Krull dimension and any maximal chain of specializations in $X$ has length $n=\dim(X)$. Then $\dim(\mathcal{O}_{X,x}) + \dim(\overline{\{x\}}) = n$ for all $x \in X$ as the left side is the length of some maximal chain of specializations in $X$. Definition~\ref{CMDerived} can thus be also reformulated as follows for such $X$: an object $M \in D_{qc}(X)$ is cohomologically CM exactly when $R\Gamma_x(K_x) \in D^{\geq n-\dim(\overline{\{x\}})}$ for all $x \in X$.
\end{remark}

\begin{remark}
\label{CCMConn}
Say $R$ is a ring with $\mathrm{Spec}(R)$ topologically noetherian. Fix $M \in D^{\leq 0}(R)$. We claim that $M$ is cohomologically CM exactly when $R\Gamma_x(M_x)$ is concentrated in degree $\dim(R_x)$ for all $x \in \mathrm{Spec}(R)$. To see this, it is enough to show that $R\Gamma_x(M_x) \in D^{\leq \dim(R_x)}$ for all $M \in D^{\leq 0}(R)$ and all $x \in \mathrm{Spec}(R)$. As $R$ is topologically noetherian, each $\{x\} \subset \mathrm{Spec}(R_x)$ is constructible for $x \in \mathrm{Spec}(R)$. This implies $R\Gamma_x(M_x)  = M \otimes_R R\Gamma_x(R_x)$ for all $x \in \mathrm{Spec}(R)$, and further (by Schreiderer's theorem \cite[Tag 0A3G]{StacksProject}) that $R\Gamma_x(R_x) \in D^{\leq \dim(R_x)}$ for all $x \in \mathrm{Spec}(R)$, whence the same holds true for $R\Gamma_x(M_x)$. 
\end{remark}

\begin{example}[Cohomologically CM objects via duality]
\label{CMDualConn}
Let $R$ be a biequidimensional noetherian ring admitting a normalized dualizing $\omega^\bullet_R$ (Definition~\ref{DefNormDual}) with associated duality functor $\mathbf{D} := \mathrm{RHom}_R(-,\omega^\bullet_R)$ on $D^b_{coh}(R)$. For any $M \in D^b_{coh}(R)$, the dual $\mathbf{D}(M)$ is cohomologically CM exactly when $M \in D^{\leq -\dim(R)}$: this follows from the compatibility of $\mathbf{D}$ with Matlis duality under local cohomology, see proof of Lemma~\ref{IndCMPunctured}. In other words, cohomologically CM objects in $D^b_{coh}(R)$ are exactly the coconnective part of a shift of the Grothendieck dual of the standard $t$-structure on $D^b_{coh}(R)$; we refer to \cite{Kashiwaratstructures,Gabbertstructures,ArinkinBezrukavnikovPervCoh} for a much more thorough study of this $t$-structure. 
\end{example}

Our goal is to prove that cohomologically CM modules over catenary noetherian domains satisfy the condition on regular sequences in the usual definition of Cohen--Macaulayness. For this, we need the following elementary (and presumably standard) observation:

\begin{lemma}[Heights of parameter ideals in catenary equidimensional rings]
\label{CatHeight}
Let $R$ be a noetherian local ring which is catenary and equidimensional. If  $x_1,...,x_d$ is a system of parameters in $R$, then $\mathrm{ht}(x_1,...,x_i) = i$.
\end{lemma}

We shall often apply this lemma when $R = S/f$, with $S$ a catenary noetherian local domain and $f \in S$ a nonzerodivisor. Note that such an $R$ is indeed catenary (by lifting saturated chains in $R$ to $S$) and equidimensional (by lifting saturated chains again, and using that any minimal prime of $R$ containing $f$ has height $1$ by Krull's theorem).

\begin{proof}
First, we observe that for any prime $\mathfrak{p}$ in such an $R$, we must have 
\begin{equation}
\label{CatDimIneq}
\mathrm{ht}(\mathfrak{p}) + \dim(R/\mathfrak{p}) = d.
\end{equation}
Indeed, splicing together a saturated chain of length $\mathrm{ht}(\mathfrak{p})$ connecting a minimal prime to $\mathfrak{p}$ with a saturated chain of length $\dim(R/\mathfrak{p})$ connecting $\mathfrak{p}$ to $\mathfrak{m}$ gives a saturated chain that connects a minimal prime of $R$ to $\mathfrak{m}$ and has length $\mathrm{ht}(\mathfrak{p}) + \dim(R/\mathfrak{p})$. As $R$ is catenary and equidimensional, all such chains have length $d$, which gives the formula. Now $\mathrm{ht}(x_1,....,x_i) = \mathrm{ht}(\mathfrak{p})$ for some minimal prime $\mathfrak{p}$ containing $x_1,....,x_i$. By Krull's theorem, we have $\mathrm{ht}(\mathfrak{p}) \leq i$. To show equality, using \eqref{CatDimIneq}, it is enough to show $\dim(R/\mathfrak{p}) \leq d-i$. But $\dim(R/\mathfrak{p}) \leq \dim(R/(x_1,...,x_i)) = d-i$, where the last equality follows as  $x_1,...,x_d$ is an SOP.
\end{proof}

The promised relation between cohomologically CM modules and regular sequences rests on the following:

\begin{lemma}[Regularity of sequences via local cohomology]
\label{CMLocCohCrit}
Let $R$ be a noetherian local ring. Let $x_1,...,x_d$ be a system of parameters in $R$ with $\mathrm{ht}(x_1,...,x_i) = i$.   Let $M$ be an $R$-module which is cohomologically CM as an object of $D(R)$\footnote{As $M$ is an honest $R$-module, for any prime $\mathfrak{p}$ of $R$, we have always have $R\Gamma_{\mathfrak{p}R_{\mathfrak{p}}}(M_{\mathfrak{p}}) = M_{\mathfrak{p}} \otimes_{R_{\mathfrak{p}}}^L R\Gamma_{\mathfrak{p}R_{\mathfrak{p}}}(R_{\mathfrak{p}}) \in D^{\leq \dim(R_{\mathfrak{p}})}$ (as the same holds true when $M=R$), so our assumption is equivalent to the a priori stronger constraint that $R\Gamma_{\mathfrak{p}R_{\mathfrak{p}}}(M_{\mathfrak{p}})$ is concentrated in degree $\dim(R_{\mathfrak{p}})$. The formulation in the lemma is nevertheless more convenient for later use.}. Then $x_1,...,x_d$ is a regular sequence on $M$. 
\end{lemma}
\begin{proof}
We shall prove by induction on $i$ that $x_1,...,x_i$ is regular on $M$. The base case $i=0$ is vacuous. Assume now that $x_1,...,x_{i-1}$ is regular on $M$. Consider the discrete $R$-module $N = M/(x_1,...,x_{i-1})M$. We want to show $x_i$ is a nonzerodivisor on $N$. Assume this is not the case, so $N[x_i]$ is nonzero. Let $\mathfrak{p}$ be some associated prime of a finitely generated submodule of $N[x_i]$, so we have an injection $R/\mathfrak{p} \subset N[x_i] \subset N$; note that $x_1,...,x_i \in \mathfrak{p}$ as $(x_1,...,x_i)$ annihilates $N[x_i]$.  Localizing the inclusion $R/\mathfrak{p} \subset N$ at $\mathfrak{p}$ then shows that 

\begin{equation}
\label{RegEq1}
H^0_{\mathfrak{p}}(N_{\mathfrak{p}}) \neq 0.
\end{equation}
But the regularity of $x_1,...,x_{i-1}$ on $M_{\mathfrak{p}}$ and the formula $N_{\mathfrak{p}} = M_{\mathfrak{p}}/(x_1,...,x_{i-1})$ shows that 
\begin{equation}
\label{RegEq2}
R\Gamma_{\mathfrak{p}}(N_{\mathfrak{p}}) \in D^{\geq \dim(R_{\mathfrak{p}}) - (i-1)}
\end{equation}
by our assumption that $R\Gamma_{\mathfrak{p}R_{\mathfrak{p}}}(M_{\mathfrak{p}}) \in D^{\geq \dim(R_{\mathfrak{p}})}$. Comparing \eqref{RegEq1} and \eqref{RegEq2}  shows that $0 \geq \dim(R_{\mathfrak{p}}) - (i-1)$ whence $\dim(R_{\mathfrak{p}}) \leq i-1$. On the other hand, $\mathfrak{p}$ contains $x_1,...,x_i$, so $\dim(R_{\mathfrak{p}}) \geq i$ by assumption on our system of parameters. Thus, we get a contradiction. 
\end{proof}

\begin{corollary}[From cohomological CMness to regular sequences]
\label{CatCohCM}
\label{RegSeqCrit}
If $S$ is a catenary and equidimensional noetherian local ring and $M$ is an $S$-module which is cohomologically CM as an object of $D(S)$, then every system of parameters on $S$ is a regular sequence on $M$.
\end{corollary}
\begin{proof}
Combine Lemma~\ref{CMLocCohCrit} and Lemma~\ref{CatHeight}.
\end{proof}

We end this subsection by recording a derived variant of the ``miracle flatness'' lemma.

\begin{lemma}[Cohomologically CM modules over regular rings are flat]
\label{FlatCritComp}
Let $R$ be a regular ring and fix $M \in D^{\leq 0}(R)$. The following are equivalent:
\begin{enumerate}
\item $M$ is $R$-flat
\item $M$ is cohomologically CM. 
\end{enumerate}
Similarly, if $f \in R$ is a nonzerodivisor, then the following are equivalent:
\begin{enumerate}
\item[$(1)_f$] $M$ is $f$-completely flat over $R$  (i.e., $M/f$ is $R/f$-flat).
\item[$(2)_f$]  For each prime $\mathfrak{p}$ of $R$ containing $f$, the complex $R\Gamma_{\mathfrak{p}R_{\mathfrak{p}}}(M_{\mathfrak{p}})$ is concentrated in degree $\dim(R_{\mathfrak{p}})$.
\end{enumerate}
\end{lemma}

We shall freely use the observation in Remark~\ref{CCMConn}.

\begin{proof}
$(1) \Rightarrow (2)$: As $R$ is regular, the ring $R_{\mathfrak{p}}$ is Cohen--Macaulay, so $R\Gamma_{\mathfrak{p}R_{\mathfrak{p}}}(R_{\mathfrak{p}})$ is concentrated in degree $\dim(R_{\mathfrak{p}})$. For any $M \in D(R)$, we have $M \otimes_R^L R\Gamma_{\mathfrak{p}R_{\mathfrak{p}}}(R_{\mathfrak{p}}) \simeq R\Gamma_{\mathfrak{p}R_{\mathfrak{p}}}(M_{\mathfrak{p}})$. Thus, if $M$ is flat, then $R\Gamma_{\mathfrak{p} R_{\mathfrak{p}}}(M_{\mathfrak{p}})$ must also be concentrated in degree $\dim(R_{\mathfrak{p}})$ by exactness of tensoring with $M$. 

$(2) \Rightarrow (1)$: It is enough to prove that $M_{\mathfrak{m}}$ is flat over $R_{\mathfrak{m}}$ for every maximal ideal $\mathfrak{m}$ of $R$. We prove this by induction on $\dim(R_{\mathfrak{m}})$. Relabelling, we may assume $R = R_{\mathfrak{m}}$.  If $\dim(R) = 0$, then $R$ is a field and $\mathfrak{m} = 0$, whence $M \simeq R\Gamma_{\mathfrak{m}}(M)$, which is concentrated in degree $0$ (and thus flat) by assumption.  Assume then that $\dim(R) > 0$, and that we know that $M_{\mathfrak{p}}$ is $R_{\mathfrak{p}}$-flat for every prime $\mathfrak{p} \subsetneq \mathfrak{m}$. Thus, $M$ is already flat on the punctured spectrum of $R$. To check that $M$ is $R$-flat, we must show that $M \otimes_R^L N$ is discrete whenever $N$ is a finitely presented $R$-module; in fact, as $M \in D^{\leq 0}$, it is enough to show the tensor product is in $D^{\geq 0}$. Using the local cohomology exact triangle with supports at $V(\mathfrak{m})$ and the flatness of $M$ on the punctured spectrum, we may assume $N$ is $\mathfrak{m}$-power torsion, whence $M \otimes_R^L N \simeq R\Gamma_{\mathfrak{m}}(M) \otimes_R^L N$. Now $R\Gamma_{\mathfrak{m}}(M) \in D^{\geq \dim(R)}$ by assumption on $M$, while $N$ has projective dimension $\leq \dim(R)$ by regularity, so the tensor product $R\Gamma_{\mathfrak{m}}(M) \otimes_R^L N$ is in $D^{\geq 0}$.

Now fix $f \in R$ a nonzero divisor. 

$(1)_f \Rightarrow (2)_f$: for any prime $\mathfrak{p} \subset R$ containing $f$, we have a base change isomorphism
\[M/f \otimes_{R/f}^L R\Gamma_{\mathfrak{p}R_{\mathfrak{p}}}(R_{\mathfrak{p}}/f) \simeq R\Gamma_{\mathfrak{p}R_{\mathfrak{p}}}(M_{\mathfrak{p}}/f).\]
If $M/f$ is $R/f$-flat, then the left side is concentrated in degree $\dim(R_{\mathfrak{p}}/f)$, whence the same holds true for the right side. By the Bockstein sequence, since $f \in \mathfrak{p}$, we obtain $(2)_f$.

$(2)_f \Rightarrow (1)_f$: We will show $M/f$ is $R/f$-flat by imitating the proof of $(2) \Rightarrow (1)$. Namely, as above, we reduce to the case where $R = R_{\mathfrak{m}}$ is regular local with $f \in \mathfrak{m}$, and that $M/f$ is already known to be flat on the punctured spectrum of $R/f$. As above, we reduce to checking the following: for any finitely presented discrete $R/f$-module $N$ supported at $\mathfrak{m}$, the complex $M/f \otimes_{R/f}^L N$ is coconnective. But this tensor product is also $M \otimes_R^L N$, so we can follow the same argument used above to establish containment in $D^{\geq 0}$. 
\end{proof}

\subsection{Ind-CM objects}

Mimicking Definition~\ref{CMDerived}, we define what it means for an ind-object of the derived category to be CM. This notion will be useful in applied eventually to various infinite constructions (such as absolute integral closures of a domain) by approximating this construction through suitably finite ones.

\begin{definition}[Ind-CM objects]
\label{IndCMDef}
Fix a finite dimensional scheme $X$. An ind-object $\{M_k\}$ in $D_{qc}(X)$ is called {\em ind-CM} if the ind-object $\{H^i_x(M_{k,x})\}$ is  $0$ for $i < \dim(\mathcal{O}_{X,x})$ for all $x \in X$. 
\end{definition}

\begin{remark}
Given an ind-CM ind-object $\{M_k\}$ on a noetherian scheme $X$,  the colimit $\colim_k M_k$ is cohomologically CM as filtered colimits are exact.
\end{remark}

\begin{remark}
Say $R$ is a noetherian $\mathbf{F}_p$-domain which is the homomorphic image of a Gorenstein ring. Choose an absolute integral closure $R \to R^+$ and let $\{S\}_{S \subset R^+}$ be the ind-object of all $R$-finite subalgebras of $R^+$. The main theorem of \cite{HunekeLyubeznik} is essentially the  statement that $\{S\}_{S \subset R^+}$ is ind-CM over $R$.
\end{remark}

\begin{definition}[Normalized dualizing complex]
\label{DefNormDual}
A normalized dualizing complex on a noetherian scheme $X$ is a dualizing complex $\omega_X^\bullet$ such that for any closed point $x \in X$, the complex $\mathrm{RHom}_X(\kappa(x), \omega_X^\bullet)$ is concentrated in degree $0$.
\end{definition}

\begin{example}[Relative dualizing complexes]
\label{ex:RelDualizing}
Say $V$ is a DVR and $f:X \to \mathrm{Spec}(V)$ is a finitely presented flat map.  For any nonzero $\pi \in V$, the $!$-pullback $(f^! V)/\pi \in D_{qc}(X_{\pi=0})$ is a normalized dualizing complex. By Noether normalization and \cite[Tag 0AX1]{StacksProject}, this assertion reduces to the case $X = \mathbf{A}^n_V$ whence $f^! V = \Omega^n_{X/V}[n] \simeq \mathcal{O}_X[n]$. In this case, it is easy to see the claim by a direct Koszul calculation relying on the fact that $V/\pi$ is Gorenstein.
\end{example}

\begin{remark}[The dimension function via normalized dualizing complexes]
\label{DualizingDim}
Let $X$ be a noetherian scheme admitting a normalized dualizing complex $\omega_X^\bullet$. Given $x \in X$, there is a unique integer $\delta(x)$  such that $\omega_X^\bullet[-\delta(x)]_x$ is a normalized dualizing complex on $\mathcal{O}_{X,x}$, see \cite[Tag 0A7W]{StacksProject}. We claim that for any $x \in X$, the number $\delta(x)$ equals the length of any maximal chain of specializations starting at $x$. (In particular, all such chains have the same length $\dim(\overline{\{x\}})$ and $\delta(x) = \dim(\overline{\{x\}})$ is independent of $\omega_X^\bullet$.) Indeed, if $x \rightsquigarrow y$ is an immediate specialization, then $\delta(y) = \delta(x) -1 $ by \cite[Tag 0A7Z]{StacksProject}. Consequently, since $\delta(y) = 0$ for $y$ closed by assumption, any maximal chain of specializations starting at $x$ has length $\delta(x)$. 
\end{remark}

This following lemma roughly proves that ind-CMness in dimension $< n$ implies some finiteness in dimension $n$ for ind-objects. Such reasoning will be crucial in our eventual application as it enables us to reduce to a sufficiently finitistic statement via induction (see proof of Theorem~\ref{AICCMMain}). The proof is a variant of an argument from \cite{HunekeLyubeznik} relying on the (shifted) compatibility of the formation of dualizing complexes with complexes that goes back to \cite[Theorem VIII.2.1]{SGA2}.

\begin{lemma}[CM in dimension $< n$ implies finiteness in dimension $n$]
\label{IndCMPunctured}
Let $X$ be a biequidimensional noetherian scheme admitting a normalized dualizing complex $\omega_X^\bullet$. Fix an ind-object $\{M_k\}$ in $D^b_{coh}(X)$ which is ind-CM after restriction to any non-closed point of $X$. Then the following hold true:
\begin{enumerate}
\item For each closed point $x \in X$ and each $i < \dim(\mathcal{O}_{X,x})$, the ind-object $\{H^i_x(M_{k,x})\}$ is isomorphic to an ind-(finite length coherent sheaf).
\item The ind-object in (1) is $0$ except for finitely many choices of the closed point $x$.
\end{enumerate}
Write $\mathbf{D}(-)$ for the autoduality of $D^b_{coh}(X)$ induced by $\omega_X^\bullet$.
\begin{enumerate}[resume]
\item For $j > -\dim(X)$, the pro-object $\{H^j \mathbf{D}(M_k)\}$ is isomorphic to a pro-(finite length coherent sheaf).
\item For each $M_k$, there is a map $M_k \to M_{k'}$ in the ind-system $\{M_k\}$ such that the induced map $H^j \mathbf{D}(M_{k'}) \to H^j \mathbf{D}(M_k)$ has finite length image for $j > -\dim(X)$.
\end{enumerate}
 \end{lemma}

The statement in (4) essentially implies all other statements; we have spelled out the rest of the statements for ease of use.

\begin{proof}
 For any $x \in X$, the complex $\omega_X^\bullet[-\delta(x)]$ is normalized dualizing after localizing $x$, with $\delta(x) = \dim(\overline{ \{x\} })$ as in Remark~\ref{DualizingDim}. By the compatibility of Grothendieck and Matlis duality, for any $M \in D^b_{coh}(X)$ and any $x \in X$, we can identify the Matlis dual of $H^i_x(M_x)$ with the completion at $x$ of $\mathrm{Ext}^{-i}_X(M, \omega_X^\bullet[-\delta(x)]) = \mathrm{Ext}^{-\delta(x)-i}_X(M, \omega_X^\bullet) = H^{-\delta(x)-i} \mathbf{D}(M)$. By the faithful flatness of completion for noetherian local rings, the assumption of the lemma is thus equivalent to the following statement: for each non-closed point $x \in X$ and each $i < \dim(\mathcal{O}_{X,x})$, the pro-object $\{H^{-\delta(x)-i} \mathbf{D}(M_k)\}$ vanishes after localization at $x$. Note that $i < \dim(\mathcal{O}_{X,x})$ exactly when $i+\delta(x) < \dim(\mathcal{O}_{X,x}) + \delta(x) = \dim(X)$, where the last equality uses the biequidimensionality of $X$ as in Remark~\ref{BiequiCM}. Thus, our assumption translates to the following: 

\begin{itemize}
\item[$(\ast)$] For each non-closed point $x \in X$ and each $j > -\dim(X)$, the pro-object $\{H^j \mathbf{D}(M_k)\}$ vanishes after localization at $x$.
\end{itemize}

Now fix some $M_k$ in the ind-system and some $j > -\dim(X)$. As  $H^j \mathbf{D}(M_k)$ is a coherent sheaf, it has finitely many associated points. Let $x_1,...,x_r$ be all the non-closed associated points of this sheaf. Using $(\ast)$, we can find a map $M_k \to M_{k'}$ in the ind-system such that $H^j \mathbf{D}(M_{k'}) \to H^j \mathbf{D}(M_k)$ is the $0$ map after localizing at each $x_i$.  As the $x_i$'s give all the non-closed associated points of the target, the image of $H^j \mathbf{D}(M_{k'}) \to H^j \mathbf{D}(M_k)$ is supported entirely at closed points. As this image is also coherent, it must then be supported at finitely many closed points and thus have finite length. The same then trivially also holds if we replace $k'$ by a bigger index in the ind-system. This gives (3) and (4), and dualizing back to local cohomology gives (1) and (2).
\end{proof}

We obtain the following relation between Definition~\ref{CMDerived} and Definition~\ref{IndCMDef}, showing that ind-(coherent sheaves) are ind-CM exactly when their colimits are ind-CM in certain settings.

\begin{lemma}[Cohomologically CM objects come from ind-CM coherent objects]
\label{IndCMcCM}
Let $X$ be a  biequidimensional noetherian scheme admitting a normalized dualizing complex. Let $\{G_k\}$ be an ind-object in $D^b_{coh}(X)$ with colimit $F \in D_{qc}(X)$. Then $F$ is cohomologically CM exactly when $\{G_k\}$ is ind-CM in $D_{qc}(X)$. 
\end{lemma}
\begin{proof}
If $\{G_k\}$ is ind-CM, it is clear that $F$ is cohomologically CM. Conversely, assume $F$ is cohomologically CM. We prove by induction on $\dim(X)$ that $\{G_k\}$ is ind-CM. 

If $\dim(X) = 0$, then $X$ is a finite disjoint union of $0$-dimensional schemes. The local cohomology functors are identified with restriction to a connected component. Unwinding definitions, the claim now reduces to the following statement: if $A$ is an artinian local ring and $\{M_i\}$ is an ind-object of $D^b_{coh}(A)$ whose colimit lies in $D^{\geq 0}$, then for each $M_i$ and each $k < 0$, there is a map $M_i \to M_j$ in the ind-system such that $H^{k}(M_i) \to H^{k}(M_j)$ is $0$. But this follows easily from the finite generation of $H^k(M_i)$ and the vanishing of $\colim_k H^k(M_j)$ (and thus holds true over all noetherian rings).

By induction, we may then assume that the claim holds true after localization at every non-closed point of $X$. We must show that for every closed point $x \in X$, the ind-object $\{H^i_x(G_{k,x})\}$ is $0$ for each $i < \dim(X)$. By Lemma~\ref{IndCMPunctured} (1), this ind-object is an ind-(finite length $\mathcal{O}_{X,x}$-module). Moreover, the colimit vanishes by assumption. We can then argue as in the previous paragraph.
\end{proof}

The following non-noetherian variant of Lemma~\ref{IndCMPunctured} will be more useful to us later.

\begin{lemma}[CM in dimension $< n$ implies finiteness in dimension $n$: a non-noetherian case]
\label{IndCMPuncturedNN}
Say $V$ be an excellent $p$-henselian and $p$-torsionfree DVR with absolute integral closure $\overline{V}$, and let $X/\overline{V}$ be a finitely presented flat integral scheme. Say $\{M_k\} \in D^b_{coh}(X_{p=0})$ is an ind-object that is ind-CM when localized at any non-closed point. Then the following hold true:
\begin{enumerate}
\item For each closed point $x \in X_{p=0}$ and each $i < \dim(\mathcal{O}_{X_{p=0},x})$, the ind-object $\{H^i_x(M_{k,x})\}$ is isomorphic to an ind-(finitely presented $\overline{V}/p$-module). 
\item The ind-object in (1) is $0$ except for finitely many choices of the closed point $x$. 
\end{enumerate}
Let $\mathbf{D}(-)$ be the autoduality of $D^b_{coh}(X_{p=0})$ induced by the relative dualizing complex for $X/\overline{V}$.
\begin{enumerate}[resume]
\item For $j > -\dim(X_{p=0})$, the pro-object $\{H^j \mathbf{D}(M_k)\}$ is isomorphic to a pro-(coherent sheaf on $X_{p=0}$ with finite support).
\item For each $M_k$ and $j > -\dim(X_{p=0})$, there is a map $M_k \to M_{k'}$ in the ind-system $\{M_k\}$ such that $\mathrm{im}(H^j \mathbf{D}(M_{k'}) \to H^j \mathbf{D}(M_k))$ is supported at finitely many closed points of $X_{p=0}$, and each of these finitely many stalks is finitely presented over $\overline{V}/p$. 
\end{enumerate}
\end{lemma}
\begin{proof}
As there does not seem a reference discussing the relation of Matlis duality with Grothendieck duality on the non-noetherian scheme $X_{p=0}$, we shall deduce the lemma from the argument in Lemma~\ref{IndCMPunctured} by approximation. Without loss of generality, we may assume $V$ is strictly henselian. By replacing $V$ with a finite extension of its strict henselisation and find a finitely presented flat integral $V$-scheme $Y$ such that $X = Y_{\overline{V}}$. As $V$ is strictly henselian, the map $X_{p=0} \to Y_{p=0}$ is a flat integral universal homeomorphism. Moreover, as $W$-ranges over all finite extensions $V \subset W$ of DVRs contained in $\overline{V}$, we obtain $X \simeq \lim_W Y_W$ and hence $X_{p=0} = \lim_W Y_{p=0,W/p}$.  Note that the topological space is constant in the tower $\{Y_{p=0,W/p}\}$ and identified with $Y_{p=0}$. Write $\mathbf{D}_{Y_{p=0,W/p}}(-)$ for the autoduality of $Y_{p=0,W/p}$ induced by the relative dualizing complex $\omega^\bullet_{Y_W/W}/p$; as the formation of relative dualizing complexes commutes with flat base change, the autodualities $\mathbf{D}_{Y_{p=0,W/p}}(-)$ of $D^b_{coh}(Y_{p=0})$ and $\mathbf{D}(-)$ on $D^b_{coh}(X_{p=0})$ are intertwined by pullback along the flat map $X_{p=0} \to Y_{p=0,W/p}$. For future reference, let us also remark that $Y_{p=0}$ is a biequidimensional scheme: indeed, any Cartier divisor in an integral biequidimensional scheme is biequidimensional, so the claim reduces to the biequidimensionality of $Y$, which follows as $Y$ is an integral finitely presented flat $V$-scheme\footnote{Indeed, the flatness of $Y/V$ and the integrality of $Y$ ensure that $Y_{p=0,red}$ is an equidimensional finite type scheme over an algebraically closed field of characteristic $p$. Any such scheme is biequidimensional by \cite[Lemma 2.6]{HeinrichBiequidim}.}.

Fix some index $k$ for the ind-system $\{M_k\}$. We can descend $M_k$ to some $N_k \in D^b_{coh}(Y_{p=0,W/p})$ for $W$ sufficiently large. For finitely many non-closed points $x_1,...,x_r \in Y_{p=0,W/p}$, the ind-CM property of $M_k$ at the $x_i$'s also descends, i.e., at the expense of enlarging $W$, there exists a map $N_k \to N_{k'}$ in $D^b_{coh}(Y_{p=0,W/p})$ descending a transition map $M_k \to M_{k'}$ in $\{M_k\}$ such that $H^i_{x_j}(N_{k,x_j}) \to H^i_{x_j}(N_{k',x_j})$ is the $0$ map for $i < \dim(\mathcal{O}_{Y_{p=0},x_i})$: indeed, if $M_k \to M_{k'}$ is chosen to witness the ind-CM property at all the $x_i$'s, then any descent of this map does the job by faithful flatness of $\overline{V}$ over $W$. Arguing as in Lemma~\ref{IndCMPunctured} via duality on $Y_{p=0,W/p}$, we then conclude that there exists a map $N_k \to N_{k'}$ descending a sufficiently large transition map $M_k \to M_{k'}$ in $\{M_k\}$ such that $H^i \mathbf{D}_{Y_{p=0,W/p}}(N_{k'}) \to H^i \mathbf{D}_{Y_{p=0,W/p}}(N_k)$ has finite length image for $i > -\dim(Y_{p=0})$. Pulling back to $X_{p=0}$ gives (3) and (4), while dualizing back to local cohomology and then pulling back o $X_{p=0}$ gives (1) and (2). 
\end{proof}

For completeness, we record the following non-noetherian variant of Lemma~\ref{IndCMcCM}.

\begin{lemma}[Cohomologically CM objects come from ind-CM coherent objects: a non-noetherian case]
\label{IndCMcCMNN}
Fix $X$ as in Lemma~\ref{IndCMPuncturedNN} and an integer $N$. Let $\{G_k\}$ be an ind-object in $D^{b,\geq -N}_{coh}(X_{p=0})$ with colimit $F \in D^{\geq -N}_{qc}(X_{p=0})$. Then $F$ is cohomologically CM exactly when $\{G_k\}$ is ind-CM as an object of $D_{qc}(X_{p=0})$. 
\end{lemma}
\begin{proof}
After possibly enlarging $V$, choose a descent $Y/V$ of $X/\overline{V}$ as in the first paragraph of the proof of Lemma~\ref{IndCMPuncturedNN}, so $X_{p=0} \to Y_{p=0}$ is a flat universal homeomorphism and $X_{p=0} = \lim_W Y_{p=0,W/p}$. We shall regard quasi-coherent sheaves on $X_{p=0}$ as quasi-coherent sheaves on $Y_{p=0}$ via pushforward. Since $X_{p=0} \to Y_{p=0}$ is a flat universal homeomorphism, the condition of being cohomologically CM for an object of $D_{qc}(X_{p=0})$ is equivalent to its image in $D_{qc}(Y_{p=0})$ being cohomologically CM, and similarly for the ind-CM property of ind-objects. In particular, $F$ is cohomologically CM over $Y_{p=0}$. Fix some $G_k$, a point $x \in X_{p=0}$, and an integer $i < \dim(\mathcal{O}_{X,x})$. We must show that there is some $G_k \to G_{k'}$ such that $H^i_x(G_{k,x}) \to H^i_x(G_{k',x})$ is $0$.  As a first step, we descend $G_k$ to $V$, i.e., after possibly enlarging $V$, we  find an $H \in D^{b,\geq -N}_{coh}(Y_{p=0})$ and an isomorphism $H \otimes_{V/p}^L \overline{V}/p \simeq G_k$ on $X_{p=0} = Y_{p=0} \otimes_{V/p} \overline{V}/p$. Consider the ind-object $\{H_j\}$ of $D^{b,\geq -N}_{coh}(Y_{p=0})$ formed by all diagrams of the form $H \to H_j \to F$ in $D^{\geq -N}_{qc}(Y_{p=0})$ with $H_j \in D^{b,\geq -N}_{coh}(Y_{p=0})$. Note that $\colim_j H_j = F$ as $D^{\geq -N}_{qc}(Y_{p=0})$ is compactly generated by $D^{b,\geq -N}_{coh}(Y_{p=0})$. Applying Lemma~\ref{IndCMcCM} to $\{H_j\}$, we learn that there exists some $H' = H_j \in D^{b, \geq -N}_{coh}(Y_{p=0})$ and a factorization $H \to H' \to F$ in $D^{\geq -N}_{qc}(Y_{p=0})$ with $H^i_x(H_x) \to H^i(H'_x)$ being the $0$ map. Since $\colim_{k' \geq k} G_{k'} = F$, the compactness of $H'$ in $D^{\geq -N}_{qc}(Y_{p=0})$ shows that there is a sufficiently large index $k'$ and a map $H' \to G_{k'}$ factoring $H' \to F$. This gives us a diagram
\[ \xymatrix{ H \ar[r] \ar[d] & H' \ar[d] \ar[rd] & \\
		G_k \ar[r] & G_{k'} \ar[r] & F}\] 
in $D^{\geq -N}_{qc}(Y_{p=0})$ where the triangle on the right as well as the outer quadrilateral commute. By compactness of $H$ in $D^{\geq -N}_{qc}(Y_{p=0})$, it follows that after enlarging $G_{k'}$ if necessary, the square on the left also commutes.   Applying $H^i_x( (-)_x)$ to the left square, we obtain a commutative diagram
\[ \xymatrix{ H^i_x(H_x) \ar[r]^a \ar[d]^d & H^i_x(H'_x)  \ar[d]^b \\
		H^i_x(G_{k,x}) \ar[r]^c & H^i_x(G_{k',x}).}\]
Now $a=0$ by choice of $H'$, so $c$ vanishes on $\mathrm{im}(d)$ by the diagram. But $c$ is a map of $\overline{V}/p$-modules and $\mathrm{im}(d)$ generates $H^i_x(G_{k,x})$ as a $\overline{V}/p$-module since $H \otimes_{V/p} \overline{V}/p \simeq G$. Thus, $c=0$, as wanted.
\end{proof}

\newpage \section{Review of prismatic cohomology and the $p$-adic Riemann-Hilbert functor}
\label{sec:RHreview}

In this section, for the convenience of the reader as well as ease of reference, we review some results from \cite{BhattLuriepadicRH1} that will be used in the sequel. 
As the theory in \cite{BhattLuriepadicRH1} as well as later results in this paper rely on prismatic cohomology \cite{BhattScholzePrism}, we give a quick summary of the necessary results in \S \ref{PrismaticReview} first. We then summarize the relevant structural results on the Riemann-Hilbert functor in \S \ref{ss:RH}, and explain some consequences for questions in commutative algebra (especially Cohen--Macaulayness considerations) in \S \ref{ss:RHCM}.

\subsection{Review of prismatic cohomology}
\label{PrismaticReview}

We review the key  properties of the prismatic complexes from \cite{BhattScholzePrism,BMS1} and then discuss an explicit example.

\begin{construction}[Prismatic cohomology over a perfect prism]
\label{ConsPrism}

Let $\mathcal{O}$ be a perfectoid ring corresponding to the perfect prism $(A,(d)) = (A_{\inf}(\mathcal{O}), \ker(\theta))$. Let $R$ be a $p$-complete $\mathcal{O}$-algebra. In \cite[\S 7.2 \& Remark 4.6]{BhattScholzePrism}, one finds a functorially defined commutative algebra object $\Prism_R \simeq \Prism_{R/A} \in D_{comp}(A)$. This objects comes equipped with a Frobenius endomorphism $\phi_R:\Prism_R \to \Prism_R$ that is linear over the Frobenius on $A$ and is a lift of the Frobenius on $\Prism_R/p$ in a suitable sense. This construction has the following features:
\begin{enumerate}
\item Hodge-Tate comparison \cite[Construction 7.6]{BhattScholzePrism}: The object $\Prism_R/d \in D_{comp}(\mathcal{O})$ admits an increasing exhaustive $\mathbf{N}$-indexed multiplicative ``Hodge-Tate'' filtration with graded pieces given by $\wedge^i L_{R/\mathcal{O}}[-i]$. In particular, the map $R = \mathrm{gr}_0^{HT}(\Prism_R/d) \to \Prism_R/d$ is a map of commutative algebras, so we can regard $\Prism_R/d$ as a commutative algebra in $D_{comp}(R)$. For $R$ formally  smooth, the description of the higher graded pieces shows that $\Prism_R/d$ is a perfect complex over $R$.

\item de Rham comparison \cite[Corollary 15.4]{BhattScholzePrism}: There is a natural isomorphism between $(\phi_{A_{\inf}(\mathcal{O})}^* \Prism_R)/d$ and the $p$-completed derived de Rham complex $L\Omega_{R/\mathcal{O}}$ of $R/\mathcal{O}$. 

\item \'Etale comparison \cite[Theorem 9.1]{BhattScholzePrism}: There is a natural isomorphism between $\left(\Prism_R/p[\frac{1}{d}]\right)^{\phi=1}$ and $R\Gamma(\mathrm{Spec}(R[1/p]), \mathbf{F}_p)$. 

\item Perfections in mixed characteristic \cite[\S 8]{BhattScholzePrism}: Let $\Prism_{R,\perf} := \colim_\phi \Prism_R \in D_{comp}(A)$ and $R_\pfd := \Prism_{R,\perf}/d \in D_{comp}(R)$. If $R_\pfd$ is concentrated in degree $0$, then $R_\pfd$ is a perfectoid ring and the map $R \to R_\pfd$ is the universal map from $R$ to a perfectoid ring. More generally, $R_\pfd$ identifies with $R\lim S$, where the inverse limit runs over all perfectoid rings $S$ equipped with a map $R \to S$. 

\item Isogeny theorem \cite[Theorem 15.3]{BhattScholzePrism}: If $R/\mathcal{O}$ is formally smooth, then the $A$-linearization $\phi^* \Prism_R \to \Prism_R$ of $\phi_R$ is a $d$-isogeny in $D_{comp}(A)$, i.e., it admits left and right inverses up to multiplication by $d^n$ for some $n$; in fact, $n = \dim(R/\mathcal{O})$ works by \cite[Corollary 15.5]{BhattScholzePrism} and (1). Reducing mod $p$ and taking the colimit shows that the map $\Prism_R/p \to \Prism_{R,\perf}/p$ has cone annihilated by $d^c$ for some integer $c \geq 1$ depending only on $\dim(R/\mathcal{O})$; in fact, one can show that $c = \frac{n}{p-1}$ works.

\item Colimit compatibility  \cite[Construction 7.6]{BhattScholzePrism}: The functor carrying the $p$-complete $\mathcal{O}$-algebra $R$ to $\Prism_R \in D_{comp}(A)$ commutes with filtered colimits. (In fact, when regarded as a functor to commutative algebra objects in $D_{comp}(A)$, this functor commutes with all colimits.) In particular, if $\{R_i\}$ is a filtered diagram of $p$-complete $\mathcal{O}$-algebras whose colimit $R = \colim_i R_i \in D_{comp}(\mathcal{O})$ is perfectoid, then we have
\[ (\colim_i \Prism_{R_i})/d \simeq (\colim_i \Prism_{R_i,\perf})/d \simeq R\] 
in $D_{comp}(\mathcal{O})$ by (4).

\item The \'etale sheaf property \cite[Construction 7.6]{BhattScholzePrism}: The construction $R \mapsto \Prism_R$ is a sheaf of complexes (in the $\infty$-categorical sense) for the Zariski (and even \'etale) topology on $\mathrm{Spf}(R)$ by (1). Consequently, for any $p$-adic formal $\mathcal{O}$-scheme $X$, one obtains a commutative algebra object $\Prism_X \in D(X,A)$ with a Frobenius endomorphism $\phi_X:\Prism_X \to \Prism_X$ satisfying analogs of (1), (2), and (5) above. 
\end{enumerate}
In the sequel, if $\mathcal{O}'$ is a ring with $p$-completion $\mathcal{O}$ and $R$ is any $\mathcal{O}'$-algebra, we set $\Prism_R := \Prism_{\widehat{R}}$, where $\widehat{R}$ is the $p$-completion of $R$. Similarly, if $X$ is an arbitrary $\mathcal{O}'$-scheme, we write $\Prism_X \in D(X, A)$ for the image of $\Prism_{\widehat{X}} \in D(\widehat{X}, A)$, where $\widehat{X}$ is the $p$-adic completion of $X$; thus, $\Prism_X/(p,d) \in D(X A/(p,d))$ naturally lifts to $D_{qc}(X_{p=0})$ by (1).
\end{construction}

\begin{remark}[How much of Construction~\ref{ConsPrism}  do we use in this paper?]
The \'etale comparison theorem in (3) does not play a direct role in this paper. We shall use the remaining results when $\mathcal{O}$ is the $p$-adic completion of the absolute integral closure $\overline{V}$ of a $p$-henselian and $p$-torsionfree DVR $V$, and $R/\mathcal{O}$ is obtained as the $p$-completion of a finitely presented $\overline{V}$-algebra (and similarly for schemes).  In this case, all results above also follow from the construction of the complexes $A\Omega_R$ in \cite{BMS1}, which preceeded \cite{BhattScholzePrism}: we have $\Prism_R = \phi_* A\Omega_R$ by \cite[Theorem 17.2]{BhattScholzePrism}.  Moreover, we shall crucially use the logarithmic variant of (1) and (4) for semistable $\mathcal{O}$-schemes; this theory is due to \v{C}esnavi\v{c}ius-Koshikawa \cite{CesnavicusKoshikawa} (which generalizes the construction in \cite{BMS1} to the logarithmic context), and is reviewed in \S \ref{ss:ProvePn}. Finally, in one isolated spot (Lemma~\ref{IncreaseFinite}), we shall also use the more general prismatic theory developed in \cite{BhattScholzePrism} in the context of the so-called ``Breuil-Kisin prism'' (which is a particular type of imperfect prism). 
\end{remark}

To give the reader a feel for the theory, we give the fundamental example of a calculation of $\Prism_R$. 

\begin{example}[The case of a torus]
Say $\mathcal{O} = \mathbf{Z}_p^{\text{cycl}} = \mathbf{Z}_p[\mu_{p^\infty}]^{\wedge}$ is the ring of integers of the perfectoid field $K = \mathbf{Q}_p(\mu_{p^\infty})^{\wedge}$. Let $R = \mathcal{O}[T^{\pm 1}]^{\wedge}$, where the completion is $p$-adic. Our goal is to describe $\Prism_R$ explicitly as a complex over $A = A_{\inf}(\mathcal{O})$.  Fix a non-trivial compatible system $(\epsilon_0,\epsilon_1,...)$ of $p$-power roots of $1$ in $\mathcal{O}$ (so $\epsilon_n$ is a primitive $p^n$-th root of $1$). This system determines an element $[\underline{\epsilon}] \in \mathcal{O}^\flat := \lim_{\phi} \mathcal{O}/p$ and hence an element $q = [\underline{\epsilon}] \in A = W(\mathcal{O}^\flat)$. In fact, one shows that $A = \mathbf{Z}_p[q^{1/p^\infty}]^{\wedge}$, where the completion is $(p,q-1)$-adic; the map $\theta:A \to \mathcal{O}$ is determined by $q^{1/p^n} \mapsto \epsilon_n$ for all $n$. A generator $d \in \ker(\theta:A \to \mathcal{O})$ is given by $d = \frac{q-1}{q^{1/p}-1} =: [p]_{q^{1/p}}$, the $q^{1/p}$-analog of $p$; note that $(p,q-1)$ and $(p,d)$ agree up to radicals. Let $I \subset A$ be the ideal of almost mathematics, so $I$ is the $(p,q-1)$-completion of the ideal generated by $\{q^{1/p^n}-1\}_{n \geq 1}$. The complex $\Prism_R \in D_{comp}(A)$ turns out to be given by the following direct sum of $2$-term complexes:
\[ \Prism_R \simeq \widehat{\bigoplus_{i \in \mathbf{Z}}} \left( A \cdot T^i \xrightarrow{ [i]_{q^{1/p} } } I \cdot T^i d\log_{q^{1/p}} T \right), \]
where the completion is $(p,q-1)$-adic, $[i]_{q^{1/p}} := \frac{q^{i/p}-1}{q^{1/p}-1}$ is the $q^{1/p}$-analog of $i$, and $T^i$ as well as  $T^i d\log_{q^{1/p}} T$ are  formal symbols. The $\phi_A$-semilinear Frobenius endomorphism of $\Prism_R$ is determined in the above presentation by requiring that it carry $T^i$ to $T^{ip}$ in the first term on the right. A more conceptual description of $\Prism_R$, and one that explains the $q$-analogs as well as Frobenius twists appearing above, is given by observing that $\phi_A^* \Prism_R$ is naturally identified with the $q$-de Rham complex of the $A$-algebra $A[T^{\pm 1}]$ with invertible \'etale co-ordinate given by $T$; see \cite{ScholzeqdR} and \cite[\S 16]{BhattScholzePrism} for more on this perspective.
\end{example}

\subsection{The Riemann-Hilbert functor}
\label{ss:RH}

In \cite{BhattLuriepadicRH1}, we construct a Riemann-Hilbert functor, attaching coherent objects to constructible \'etale sheaves on $p$-adic schemes. This construction can be regarded as a relative and integral variant of Fontaine's $D_{HT}(-)$ functor from \cite[\S 1.5]{Fontainepadicrepr}. We summarize the main results about this construction next:

\begin{theorem}[The Riemann-Hilbert functor for torsion coefficients]
\label{thm:RH}
Let $\mathcal{O}$ be a perfectoid ring, and let $X/\mathcal{O}$ be a scheme. For each $n \geq 1$, there is an exact colimit preserving functor
\begin{equation}
\tag{RH}
\label{eq:RHGeneral}
\RH_{\overline{\Prism}}:D(X, \mathbf{Z}/p^n) \to D_{qc}(X_{p^n=0}) 
\end{equation}
with the following features:
\begin{enumerate}
\item {\em The constant sheaf:} We have $\RH_{\overline{\Prism}}(\mathbf{Z}/p^n) = \Prism_{X,\perf}/(d,p^n) = \mathcal{O}_{X,\pfd}/p^n$. In particular, if the $p$-completion $\widehat{X}$ is perfectoid, then we have $\RH_{\overline{\Prism}}(\mathbf{Z}/p^n) = \mathcal{O}_X/p^n$ (where the quotient is derived).

\item {\em Proper pushforward:} If $f:Y \to X$ is a proper map of $\mathcal{O}$-schemes, then there is a natural isomorphism 
\[ \RH_{\overline{\Prism}} \circ Rf_{*} \simeq Rf_* \circ \RH_{\overline{\Prism}}\] 
of functors $D(Y, \mathbf{Z}/p^n) \to D_{qc}(X_{p^n=0})$.

\item {\em Sheaves on the generic fibre:} If $j:X[1/p] \hookrightarrow X$ denotes the inclusion, then $\RH_{\overline{\Prism}}$ carries the full subcategory 
\[ j_!:D(X[1/p], \mathbf{Z}/p^n) \hookrightarrow D(X, \mathbf{Z}/p^n)\] 
into the full subcategory 
\[ (-)_!:D_{qc}(X_{p^n=0})^a \hookrightarrow D_{qc}(X_{p^n=0}),\]  
where the almost category $D_{qc}(X_{p^n=0})^a$ is in the usual context (see Notation~\ref{GlobalNotation}). Moreover, for any $F \in D(X, \mathbf{Z}/p^n)$, the map $\RH_{\overline{\Prism}}(j_! (F|_{X[1/p]})) \to \RH_{\overline{\Prism}}(F)$ coming from functoriality identifies with the counit $(M^a)_! \to M$ available for any $M \in D_{qc}(X_{p^n=0})$; in particular, this map is an almost isomorphism.
\end{enumerate}

For the rest of the theorem, assume $\mathcal{O} = \mathcal{O}_C$ for a nonarchimedean perfectoid extension $C/\mathbf{Q}_p$, and  $X/\mathcal{O}_C$ is a finitely presented flat $\mathcal{O}_C$-scheme. We may regard $\RH_{\overline{\Prism}}$ as giving a functor 
\begin{equation}
\tag{RH$^a$}
\label{eq:RHOC}
\RH_{\overline{\Prism}}:D^b_{cons}(X[1/p], \mathbf{Z}/p^n) \to D_{qc}(X_{p^n=0})^a
\end{equation}
via (3). This functor has the following features:

\begin{enumerate}[resume]

\item {\em Almost coherence:} The functor in \eqref{eq:RHOC} takes values in the full subcategory $D^b_{acoh}(X_{p^n=0})^a$ of almost coherent complexes; here $N \in D^b_{qc}(X_{p^n=0})^a$ is called almost coherent if it is bounded above and for each $\epsilon \in \sqrt{p\mathcal{O}_C}$, there is an object $M_\epsilon \in D^b_{coh}(X_{p^n=0})$ and a map $M_\epsilon \to N$ in $D^b_{qc}(X_{p^n=0})^a$ whose cone has cohomology sheaves annihilated by $\epsilon$.

\item {\em Duality compatibility:} Writing $\mathbf{D}_V = \underline{\mathrm{RHom}}_{X[1/p]}(-, \omega_{X[1/p],et})$ and $\mathbf{D}_G = \underline{\mathrm{RHom}}_{X_{p^n=0}}(-, \omega_{X/\mathcal{O}_C,coh}^\bullet/p^n)$ for the Verdier and Grothendieck duality functors on $D^b_{cons}(X[1/p], \mathbf{Z}/p^n)$ and $D_{acoh}^b(X_{p^n=0})^a$ respectively, we have a natural isomorphism 
\[ \RH_{\overline{\Prism}} \circ \mathbf{D}_V \simeq \mathbf{D}_G \circ \RH_{\overline{\Prism}} \]
of contravariant functors $D^b_{cons}(X[1/p], \mathbf{Z}/p^n) \to D_{acoh}^b(X_{p^n=0})^a$

\item {\em Perverse $t$-structures:} Endow $D^b_{cons}(X[1/p], \mathbf{Z}/p^n)$ with the perverse $t$-structure for middle perversity, as in \cite[\S 4]{BBDG}. Then the functor in \eqref{eq:RHOC} satisfies
\[ \RH_{\overline{\Prism}} \left( {}^p D^{\leq 0}_{cons} \right) \subset D^{\leq 0}_{qc}(X_{p^n=0})^a \]
and
\[ \RH_{\overline{\Prism}} \left( {}^p D^{\geq 0}_{cons} \right) \subset \{K \in D^+_{qc}(X_{p^n=0})^a \mid R\Gamma_x(K_x) \in D^{\geq -\dim(\overline{\{x\}})} \ \ \forall x \in X_{p^n=0} \}. \]
\end{enumerate}
\end{theorem}

The subscript in $\RH_{\overline{\Prism}}$ is meant to be suggestive: roughly, if one works over a perfectoid ring $\mathcal{O}$, then one has a family of such functors parametrized by $\mathrm{Spec}(\Prism_{\mathcal{O}})$, and the functor $\RH_{\overline{\Prism}}$ discussed above is the restriction of this family to the Hodge-Tate divisor 
$\mathrm{Spec}(\mathcal{O}) = \mathrm{Spec}(\overline{\Prism}_{\mathcal{O}}) \subset \mathrm{Spec}(\Prism_{\mathcal{O}})$ (see also Remark~\ref{rmk:AinfLift}).

\begin{remark}[Refining $\RH$ to account for Frobenius]
\label{rmk:AinfLift}
Fix notation as in Theorem~\ref{thm:RH} and a distinguished element $d \in \ker(\theta) \subset A_{\inf}(\mathcal{O})$. Recall from Construction~\ref{ConsPrism} that $\mathcal{O}_{X,\pfd} = \Prism_{X,\perf}/d$, and that $\Prism_{X,\perf} \in D_{comp}(X,A_{\inf}(\mathcal{O}))$ carries a Frobenius automorphism. In \cite{BhattLuriepadicRH1}, we refine Theorem~\ref{thm:RH} to take values in Frobenius modules over $\Prism_{X,\perf}$. More precisely, write $D_{comp,qc}(X,\Prism_{X,\perf}/p^n)$ denotes the full subcategory of $D(X, \Prism_{X,\perf}/p^n)$ spanned by $d$-complete objects that are quasi-coherent modulo $d$ after restriction of scalars along $\mathcal{O}_X/p^n \to \Prism_X/(d,p^n) \to \Prism_{X,\perf}/(d,p^n)$. Then we show that $\RH_{\overline{\Prism}}$ factors as 
\[ D^b_{cons}(X, \mathbf{Z}/p^n) \xrightarrow{\RH_{\Prism}} D_{comp,qc}(X, \Prism_{X,\perf}/p^n)^{\phi=1} \xrightarrow{- \otimes_{\Prism_{X,\perf}}^L \mathcal{O}_{X,\pfd}} D_{qc}(X_{p^n=0})\] 
with the following features:
\begin{itemize}
\item The functor $\RH_{\Prism}$ symmetric monoidal and satisfies the natural variants of Theorem~\ref{thm:RH} (1) - (3).
\item In the context of the second half of Theorem~\ref{thm:RH}, if $X/\mathcal{O}_C$ is proper, then $\RH_\Prism$ is fully faithful. 
\item $\RH_{\Prism}$ identifies with the functor in \cite[Theorem 12.1.5]{BhattLurieModpRH} when $X$ is an $\mathbf{F}_p$-scheme.
\end{itemize}
In the sequel, we shall use the existence of this lift $\RH_\Prism$ of $\RH_{\overline{\Prism}}$.
\end{remark}

\begin{remark}[Refining $\RH$ to $\mathcal{O}_{\pfd}$-coefficients]
Fix notation as in Theorem~\ref{thm:RH}. It follows from Remark~\ref{rmk:AinfLift} that the functor $\RH_{\overline{\Prism}}$ from \eqref{eq:RHGeneral} can be regarded as a symmetric monoidal functor 
\[ \RH_{\overline{\Prism}}:D^b_{cons}(X, \mathbf{Z}/p^n) \to D_{qc}(X, \mathcal{O}_{X,\pfd}/p^n).\]
In \cite{BhattLuriepadicRH1}, we show that this variant of $\RH_{\overline{\Prism}}$ commutes with pullbacks along maps $\mathcal{O}$-schemes. 

Moreover, if $X/\mathcal{O}_C$ is finitely presented and flat, then (using \cite{Gabbertstructures}), \cite{BhattLuriepadicRH1} constructs both``constructible'' and ``perverse'' $t$-structures on $D_{qc}(X, \mathcal{O}_{X,\pfd}/p^n)^a$ making the functor 
\[ \RH_{\overline{\Prism}}:D^b_{cons}(X[1/p], \mathbf{Z}/p^n) \to D_{qc}(X, \mathcal{O}_{X,\pfd}/p^n)^a\] 
coming from  \eqref{eq:RHOC}  $t$-exact for the corresponding $t$-structures on the left hand side.
\end{remark}

\begin{remark}[Applying $\RH$ over non-complete ground fields]
\label{NonCompBase}
In some applications of the second half of Theorem~\ref{thm:RH}, to make noetherian approximation arguments feasible, we often find ourselves in the following situation: the $\mathcal{O}_C$-scheme $X$ is the base change of a $W$-scheme $Y$, where $W$ is a rank $1$ absolutely integrally closed valuation ring with $p$-completion $\mathcal{O}_C$.  In this situation, we shall often write $\RH$ for the composition of the functor in \eqref{eq:RHOC} with the scalar extension functor $D^b_{cons}(Y[1/p], \mathbf{Z}/p^n) \to D^b_{cons}(X[1/p], \mathbf{Z}/p^n)$ along the extension $C/W[1/p]$ of algebraically closed fields
\end{remark}

We close this subsection with a key sample calculation of $\RH_{\overline{\Prism}}$.

\begin{example}[The Riemann-Hilbert functor for standard sheaves]
Let $R$ be a perfectoid ring. Theorem~\ref{thm:RH} gives a functor $\RH_{\overline{\Prism}}:D(\mathrm{Spec}(R), \mathbf{F}_p) \to D_{qc}(R/p)$. Let $i:\mathrm{Spec}(R/I) \subset \mathrm{Spec}(R)$ be a  closed subset defined by an ideal $I \subset R$, and let $j:U \subset \mathrm{Spec}(R)$ be the complementary open. We have a short exact sequence
\[ 0 \to j_! \mathbf{F}_p \to \mathbf{F}_p \to i_* \mathbf{F}_p \to 0\]
of $\mathbf{F}_p$-sheaves on $\mathrm{Spec}(R)$. The functor $\RH_{\overline{\Prism}}$ carries this sequence to the short exact sequence
\[ 0 \to I_\pfd \to R \to (R/I)_\pfd \to 0\]
by Theorem~\ref{thm:RH} (2) and \cite[Theorem 7.4]{BhattScholzePrism}. This calculation illustrates another feature of the functor $\RH_{\overline{\Prism}}$: it is $t$-exact with respect to the standard $t$-structures on $D(\mathrm{Spec}(R), \mathbf{F}_p)$ and $D(R/p)$ when $R$ is perfectoid. 
\end{example}

\subsection{Applications of the Riemann-Hilbert functor in commutative algebra}
\label{ss:RHCM}

In this subsection, we recall from \cite{BhattLuriepadicRH1} some consequences of Theorem~\ref{thm:RH} for commutative algebra. Many of these rely on the following statement contained in Theorem~\ref{thm:RH} (6).

\begin{corollary}[$\RH$ carries perverse sheaves to shifted almost Cohen--Macaulay complexes]
\label{cor:PervACM}
With notation as in the second half of Theorem~\ref{thm:RH}, if $F \in {}^p D^{\geq 0}_{cons}(X[1/p], \mathbf{Z}/p^n)$, then $R\Gamma_x(\RH_{\overline{\Prism}}(F)_x) \in D^{\geq -\dim(\overline{\{x\}}),a}$ for all  $x \in X_{p^n=0}$.
\end{corollary}

In applications of Corollary~\ref{cor:PervACM}, we need a large supply of perverse sheaves. Besides standard examples in algebraic geometry (e.g., those provided by Artin vanishing), the main new source of examples comes from the following simple but critical observation (see proof of Theorem~\ref{ACMaic}):

\begin{proposition}[\'Etale acyclicity of absolutely integrally closed integral schemes]
\label{EtaleCohAIC}
Let $X$ be an integral normal scheme with algebraically closed function field. Then $\mathbf{F}_p \simeq R\Gamma(X, \mathbf{F}_p)$. If $f:Y \to X$ is a dominant  map between two such schemes (e.g., a non-empty open immersion), then $\mathbf{F}_p \simeq Rf_* \mathbf{F}_p$.
\end{proposition}
\begin{proof}[Sketch of proof]
For the first part, one first shows that the Zariski and \'etale topologies of $X$ coincide: any affine \'etale $X$-scheme itself a finite product integral normal schemes with algebraically closed function fields, and these schemes are open subsets of $X$ by generic point considerations (see \cite[Lemma 3]{GabberAffineAnalog}). Having translated from \'etale to Zariski cohomology, we conclude by Grothendieck's theorem \cite[Tag 02UW]{StacksProject} that constant sheaves have vanishing higher cohomology on irreducible topological spaces.

For the second part, as the Zariski and \'etale topologies of $X$ coincide, it is sufficient to prove that $R\Gamma(U, \mathbf{F}_p) \simeq R\Gamma(f^{-1}U, \mathbf{F}_p)$ for all non-empty opens $U \subset X$. Since $f$ is dominant, $f^{-1}U$ is a non-empty open subscheme of $Y$. But then both $U$ and $f^{-1}U$ are integral normal schemes with algebraically closed function fields, so the claim follows from the previous paragraph as both sides identify with $\mathbf{F}_p$.
\end{proof}

The main reason we cared about Corollary~\ref{cor:PervACM} was part (2) of the following result, which was previously known by work of Heitmann \cite{HeitmannDSC3} in the case of relative dimension $2$.

\begin{theorem}[Almost Cohen--Macaulayness of absolute integral closures]
\label{ACMaic}
Let $W$ be an absolutely integrally closed rank $1$ valuation ring with pseudouniformizer $p$. Let $X/W$ be a finitely presented flat integral scheme of relative dimension $n$. Let $\pi:X^+ \to X$ be an absolute integral closure. Write $(-)_\eta$ for the functor of inverting $p$. 
\begin{enumerate}
\item The object $\pi_{\eta,*} \mathbf{F}_p[n] \in D^b(X_\eta, \mathbf{F}_p)$ is a filtered colimit of perverse sheaves.
\item For any $x \in X_{p=0}$, the complex $R\Gamma_x(\pi_* \mathcal{O}_{X^+,x}/p)$ is almost concentrated in degree $n-\dim(\overline{\{x\}})$. (In other words, $\pi_* \mathcal{O}_{X^+}/p \in D_{qc}(X_{p=0})$ satisfies the almost analog of Definition~\ref{CMDerived}.)
\end{enumerate}
\end{theorem}

Note that part (1) of this result would be completely false with characteristic $0$ coefficients.

\begin{proof}[Sketch of proof]
(1): Consider the category (in fact, poset) $\mathcal{P}$ of all factorizations $\{X^+ \xrightarrow{g_Y} Y \xrightarrow{f_Y} X\}$ with $Y \to X$ finite surjective with $Y$ integral normal and $X^+ \to Y$ dominant. As $\pi_{\eta,*} \mathbf{F}_p$ is the colimit of the ind-object $\{f_{Y,\eta,*} \mathbf{F}_p\}_{\mathcal{P}}$ in $D^b_{cons}(X_\eta, \mathbf{F}_p)$, it is enough to show that the ind-object $\{f_{Y,\eta,*} \mathbf{F}_p[n]\}_{\mathcal{P}}$ is ind-perverse. Fix some $Y \in \mathcal{P}$. As $Y \to X$ is a finite map of reduced varieties in characteristic $0$, we may choose an affine open subset $j:U \subset X_\eta$ such that $Y_U \to U$ is finite \'etale, so $Rj_* (f_{Y,\eta,*} \mathbf{F}_p[n]|_{U})$ is perverse by Artin vanishing as in \cite[Corollary 4.1.3]{BBDG}. It is then enough to prove that there is a map $Z \to Y$ in $\mathcal{P}$ such that the pullback map $f_{Y,\eta,*} \mathbf{F}_p[n] \to f_{Z,\eta,*} \mathbf{F}_p[n]$ factors over $f_{Y,\eta,*} \mathbf{F}_p[n] \to Rj_* (f_{Y,\eta,*} \mathbf{F}_p[n]|_U)$. But this can be checked after taking a colimit over all such $Z$'s by compactness of constructibile complexes, and then it follows from the observation that the constant sheaf $\mathbf{F}_p$ on $X^+$ is $*$-extended from $X^+_U$ by Proposition~\ref{EtaleCohAIC}.

(2): Consider the functor $\RH:D(X_\eta, \mathbf{F}_p) \to D_{qc}(X_{p=0})^a$ coming from Theorem~\ref{thm:RH} applied to $p$-completion of $X/W$ as in Remark~\ref{NonCompBase}. Theorem~\ref{thm:RH} (1) and (2) show that $\RH(\pi_{\eta,*} \mathbf{F}_p) \simeq \pi_* \mathcal{O}_{X^+}/p$ since the $p$-completion of $X^+$ is a perfectoid formal scheme (see Lemma~\ref{AICPerfectoid} below); the claim now follows from (1) as well as Corollary~\ref{cor:PervACM}.
 \end{proof}

Another consequence of Theorem~\ref{thm:RH} that we shall need is the following, giving a $p$-adic lift of the main result of \cite{ddscposchar} as well as a generalization of \cite{Bhattpdiv} (as well as parts of \cite{Beilinsonpadic}) that accommodates torsion classes and works over a higher dimensional base. 

\begin{theorem}[Annihilating cohomology of proper maps]
\label{KillCohMixed}
Let $f:X \to S$ be a proper morphism of noetherian $p$-torsionfree schemes. Then there exists a finite cover $\pi:Y \to X$ such that, with $g = f \circ \pi$, the map $Rf_* \mathcal{O}_X/p \to Rg_* \mathcal{O}_Y/p$ factors over $(g_* \mathcal{O}_Y)/p \hookrightarrow H^0(Rg_* \mathcal{O}_Y/p) \to Rg_* \mathcal{O}_Y/p$ over $\mathcal{O}_S/p$. 
\end{theorem}
\begin{proof}[Sketch of proof]
By standard arguments using the coherence of higher direct images, one reduces to the following case: $X$ and $S$ are integral, $f$ is surjective, and $S$ is finitely presented and flat over a $p$-adic DVR. Choose an absolute integral closure $\pi:X^+ \to X$ and let $h:X^+ \to S^+ = \underline{\mathrm{Spec}}_S(h_* \mathcal{O}_{X^+})$ be the Stein factorization of $X^+ \to X \to S$. By a limit argument, it is enough to show that $(h_* \mathcal{O}_{X^+})/p \simeq Rh_* \mathcal{O}_{X^+}/p$. Since $\mathcal{O}_{S^+} = h_* \mathcal{O}_{X^+}$ by construction, we  have  $\mathcal{O}_{S^+}/p = (h_* \mathcal{O}_{X^+})/p$, so we must show $\mathcal{O}_{S^+}/p \simeq Rh_* \mathcal{O}_{X^+}/p$. For this, observe that $S^+ \to S$ is an absolute integral closure: this is a surjective integral map from an integral scheme $S^+$, and $S^+$ is absolutely integrally closed as $X^+$ is so\footnote{It is enough to show that any finite cover $T \to S^+$ has a section. The base change of $T \times_{S^+} X^+ \to X^+$ has a section as $X^+$ is integral normal with algebraically closed fraction field (e.g., by taking the closure of a section over the generic point). Fixing such a section  yields an $S^+$-map map $X^+ \to T$. By the universal property of the affinization map $X^+ \to S^+$, this factors over $X^+ \to S^+$ to yield an $S^+$-map $S^+ \to T$, thus giving a section of $T \to S^+$.}. Thus, both $S^+$ and $X^+$ are perfectoid after $p$-completion (see Lemma~\ref{AICPerfectoid}). The isomorphy of $\mathcal{O}_{S^+}/p \to Rh_* \mathcal{O}_{X^+}/p$ will now follow from Theorem~\ref{thm:RH} (1) and (2)  if we can show that $\mathbf{F}_p \simeq Rh_* \mathbf{F}_p$ as sheaves on $S^+$. But this is clear from Proposition~\ref{EtaleCohAIC}.
\end{proof}

\newpage \section{The geometric result}
\label{CMgeomcase}

This section is the technical heart of the paper: we prove that absolute integral closures of essentially finitely presented normal local domains $R$ over $p$-adic DVRs are Cohen--Macaulay (Theorem~\ref{AICCMMain}). We begin in \S \ref{MellowPn} by formulating two stronger versions of this statement that are equivalent to each other and are more amenable to inductive arguments than the mere Cohen--Macaulayness of $R^+$. This stronger statement is then proven in \S \ref{ss:ProvePn}; as this proof relies on a few different inputs (the Riemann-Hilbert functor, alterations, and log prismatic cohomology), we first give a rough summary of the structure of the argument in \S \ref{ss:OutlineProof}.

\subsection{A preliminary reduction}
\label{MellowPn}

Our goal is to prove that local cohomology in mixed characteristic can be annihilated by finite covers. For this purpose, it is useful to make the following temporary definition capturing the outcome we want to achieve:

\begin{definition}
\label{Def:mellow}
An excellent normal local domain $(R,\mathfrak{m})$ satisfies $(*)_{CM}$ if there exists a finite extension $R \to S$ such that $H^i_{\mathfrak{m}}(R/p) \to H^i_{\mathfrak{m}}(S/p)$ is the $0$ map for $i < \dim(R/p)$.
\end{definition}

Here (and in the sequel) a finite extension $R \to S$ of domains is a finite injective map of domains.  Our main theorem is the following:

\begin{theorem}
\label{KillCohFinite}
For any $p$-henselian  $p$-torsionfree excellent DVR $V$, any flat finite type normal $V$-scheme $X$ and any  point $x \in X_{p=0}$, the local ring $\mathcal{O}_{X,x}$ satisfies $(*)_{CM}$. 
\end{theorem}

In this section, we reduce the assertion in Theorem~\ref{KillCohFinite} to a somewhat smaller class of rings by essentially elementary arguments. First, we observe that $+$-CMness can be detected \'etale locally:

\begin{lemma}
\label{mellowstable}
Fix an excellent normal local domain $(R,\mathfrak{m})$. The following are equivalent:
\begin{enumerate}
\item $R$ satisfies $(*)_{CM}$.
\item The henselization $R^h$ satisfies $(*)_{CM}$.
\item A strict henselization $R^{sh}$ satisfies $(*)_{CM}$.
\end{enumerate}
\end{lemma}
\begin{proof}
$(1) \Rightarrow (2),(3)$: This follows from flatness of $R \to R^h$ as we have a base change isomorphism $R\Gamma_{\mathfrak{m}}(R^h/p) \simeq R^h \otimes_R R\Gamma_{\mathfrak{m}}(R/p)$ (and similarly for $R^{sh}$).

$(2) \Rightarrow (1)$: The henselization $R^h$ can be written as a filtered limit $\colim_i S_i$ where each $S_i$ is the localization of an \'etale $R$-algebra at a point above $\mathfrak{m}$ with trivial residue field extension. As the map $S_i \to R^h$ induces an isomorphism on local cohomology (as above), and because any finite extension of $R^h$ is pulled back from a finite extension of some $S_i$, we learn that $S_i$ satisfies $(*)_{CM}$ for $i \gg 0$. To descend further to $R$, we apply Lemma~\ref{DescendFiniteCov} to a finite extension $S_i \to T_i$ witnessing the $(*)_{CM}$ property of $S_i$ to obtain a finite extension $R \to R'$ such that base change map $S_i \to R' \otimes_R S_i$ factors over $S_i \to T_i$. As base change along $R \to S_i$ induces an isomorphism on $H^j_{\mathfrak{m}}(-/p)$, it follows that $H^j_{\mathfrak{m}}(R/p) \to H^j_{\mathfrak{m}}(R'/p)$ is $0$ for $j < \dim(R/p)$.

$(3) \Rightarrow (2)$: Note that $R^h \to R^{sh}$ is ind-(finite \'etale) extension of normal local domains, and that any finite  extension of $R^{sh}$ is obtained via base change from a finite extension of some finite \'etale $R^h$-subalgebra $S \to R^{sh}$. Now if $S \to T$ is a finite extension of normal domains such that $R^{sh} \to R^{sh} \otimes_S T$ is a finite extension of normal domains witnessing the $(*)_{CM}$ property of $R^{sh}$, then $R^h \to S \to T$ witnesses the $(*)_{CM}$-property of $R^h$, as one checks using base change for local cohomology along the faithfully flat maps $R^h \to R^{sh}$ and $T \to T \otimes_S R^{sh}$. 
\end{proof}

The following lemmas were used above:

\begin{lemma}
\label{DescendFiniteCov}
Let $R \to S$ be an quasi-finite \'etale extension of excellent normal domains. For any finite extension $S \to S'$ of normal domains, there exists a finite extension $R \to R'$ of normal domains such that the $S$-algebra $R' \otimes_R S$ decomposes as a product $\prod_{i=1}^n S_i$ of normal domains such that each of the maps $S \to S_i$ (and hence the map $S \to R' \otimes_R S$) factors over $S \to S'$. 
\end{lemma}
\begin{proof}
Apply Lemma~\ref{GaloisExtRefine} to $K = K(R)$, $L = K(S)$ and $M = K(S')$ to obtain extension $E/K(R)$ such that $E \otimes_{K(R)} K(S) = \prod_i M_i$ with each $K(S) \to M_i$ being a field extension refining $K(S) \to K(S')$. Set $R'$ to be the normalization of $R$ in the field extension $E/K$. The ring $R' \otimes_R S$ is \'etale over $R'$ and hence normal. Moreover, each connected component of this ring is a finite normal extension of $S$. Thus, it follows that $R' \otimes_R S = \prod_i S_i$ with each $S_i$ being the normalization of $S$ in $M_i$. It is then clear from the hypothesis on $K(S) \to M_i$ that the map $S \to S_i$ factors over $S \to S'$.
\end{proof}

\begin{lemma}
\label{GaloisExtRefine}
Say $K \to L$ is a finite extension of characteristic $0$ fields, and $L \to M$ is a further finite extension. Then there is a finite extension $K \to E$ such that the $L$-algebra $E \otimes_K L$ decomposes  as $\prod_{i=1}^n M_i$, with each $L \to M_i$ being a field extension refining $L \to M$. 
\end{lemma}
\begin{proof}
If $\overline{K}$ denotes an algebraic closure of $K$, then $\overline{K}/K$ is an infinite extension and the base change $L \to \overline{K} \otimes_K L$ decomposes as a product $\prod_i N_i$ of finitely many field extensions of $L$ with each $N_i/L$ being an algebraic closure. In particular, each of the maps $L \to N_i$ factors the map $L \to M$. Writing $\overline{K}$ as a union of finite extensions $E/K$ then shows that any sufficiently large finite extension $E/K$ has the desired form.
\end{proof}

The flexibility of allowing a possibly non-closed point $x$ in Theorem~\ref{KillCohFinite} is useful for inductive arguments. Nevertheless, eventually, we wish to restrict attention to closed points to make geometric arguments. To juggle these needs, we make the following temporary definitions:

\begin{definition}
For any integer $n \geq 1$, define the following properties:
\begin{itemize}
\item[$(M_n):$] The conclusion of Theorem~\ref{KillCohFinite} holds true for all choices of $V$, $X$, and $x$ with $\mathcal{O}_{X,x}$ having relative dimension $n$ over $V$ (i.e., $\dim(\mathcal{O}_{X,x}) = n+1$).
\item[$(P_n):$] For any $p$-henselian  $p$-torsionfree excellent DVR $V$,  any maximal ideal $\mathfrak{m} \subset \mathcal{O}_{\mathbf{P}^n_V}$ and any finite extension $\mathcal{O}_{\mathbf{P}^n_V,\mathfrak{m}} \to R$ of normal  domains, there exists a further finite extension $R \to S$ such that  $H^i_{\mathfrak{m}}(R/p) \to H^i_{\mathfrak{m}}(S/p)$ is the $0$ map for $i < \dim(R/p) = n$. 
\end{itemize}
\end{definition}

Thanks to Noether normalizations, the main difference between $(M_n)$ and $(P_n)$ is that the point $x$ appearing in $(M_n)$ can be a non-closed point unlike the point appearing in $(P_n)$. The next lemma explains why the two conditions are nevertheless equivalent, at least at the expense of enlarging the base DVR.

\begin{lemma} 
\label{CMProjective}
Fix an integer $n \geq 1$. The following are equivalent.
\begin{enumerate}
\item $(M_k)$ holds true for all $k \leq n$.
\item $(P_k)$ holds true for all $k \leq n$.
\end{enumerate}
\end{lemma}
\begin{proof}
We shall prove that $(M_n)$ implies $(P_n)$ for each $n$, and that $(P_j)$ for all $j \leq k$ implies $(M_k)$. 

$(1) \Rightarrow (2)$: Assume $(M_n)$ holds true. Fix $\mathfrak{m}$ and $R$ as in $(P_n)$. Then $\mathrm{Spec}(R)$ has finitely many closed points $\{\mathfrak{n}_1,...,\mathfrak{n}_r\}$, and they all lie above $\mathfrak{m}$. But then we have $H^i_{\mathfrak{m}}(R/p) \simeq \prod_{j=1}^r H^i_{\mathfrak{n}_j}(R_{\mathfrak{n}_j}/p)$ by henselization considerations.  Now, by assumption, for each $j$, there exists a finite extension $R_{\mathfrak{n}_j} \to S_j$ such that $H^i_{\mathfrak{n}_j}(R_{\mathfrak{n}_j}/p) \to H^i_{\mathfrak{n}_j}(S_j/p)$ is the $0$ map for $i < \dim(R_{\mathfrak{n}_j}/p) = \dim(R/p)$. We can then find a single finite extension $R \to S$ such that the induced map $R_{\mathfrak{n}_j} \to S \otimes_R R_{\mathfrak{n}_j}$ factors over $R_{\mathfrak{n}_j} \to S_j$. As the local cohomology groups of $S/p$ also split similarly to those of $R/p$, we then conclude that $H^i_{\mathfrak{m}}(R/p) \to H^i_{\mathfrak{m}}(S/p)$ is the $0$ map for $i < \dim(R)$. 

$(2) \Rightarrow (1)$: Fix a $p$-henselian  $p$-torsionfree excellent DVR $V$, a flat finite type normal $V$-scheme $X$ and a point $x \in X_{p=0}$ with $\dim(\mathcal{O}_{X,x}) = k+1$ with $k \leq n$. We must show that $\mathcal{O}_{X,x}$ satisfies $(*)_{CM}$. 

Assume first that $x$ is a closed point of $X_{p=0}$. Then we can find a quasi-finite injective map $\mathcal{O}_{\mathbf{P}^k_V,0} \to \mathcal{O}_{X,x}$ of domains by noether normalization (where $0$ denotes the origin on the special fibre). The normalization of $\mathcal{O}_{\mathbf{P}^k_V,0}$ in the resulting fraction field extension is a finite extension $\mathcal{O}_{\mathbf{P}^k_V,0} \to R$ such that $\mathcal{O}_{X,x}$ is the local ring of $R$ at one of the finitely many closed points of $\mathrm{Spec}(R)$ above $0$. Pick a finite extension $R \to S$ provided by the assumption in (2). Reversing the reasoning used to prove $(1) \Rightarrow (2)$ above then shows that the finite extension $\mathcal{O}_{X,x} \to S \otimes_R \mathcal{O}_{X,x}$ obtained by localizing $S$ at $x$ does the job. 

Assume now that $x$ is not a closed point of $X_{p=0}$. By Lemma~\ref{MakeClosed}, we can find an extension $W/V$ of DVRs that is essentially of finite type such that $\mathcal{O}_{X,x} = \mathcal{O}_{X',x'}$ for a flat finite type normal $W$-scheme $X'$ and a closed point $x' \in X'_{p=0}$. We cannot directly apply the assumption in $(2)$ as $W$ might not be $p$-henselian. Nevertheless, by Lemma~\ref{mellowstable}, it is enough to prove the statement after base changing along $W \to W^h$. Thus, we conclude using the argument in the previous paragraph applied to the flat finite type $W^h$-scheme $Y := X' \otimes_W W^h$ and the closed point $x' \in Y_{p=0}$ (noting that $X'_{p=0} = Y_{p=0}$).
\end{proof}

The following lemma was used above to realize non-closed points of finite type schemes over a DVR as closed points of a different finite type schemes over a different DVR. 

\begin{lemma}
\label{MakeClosed}
Fix a DVR $V$, a finite type $V$-scheme $X$ and a point $x \in X_{\kappa(V)}$. Then there exists a factorization $V \to W \to \mathcal{O}_{X,x}$, where $W/V$ is an essentially finitely presented extension of DVRs that preserves uniformizers, and  $W \to \mathcal{O}_{X,x}$ is a local map inducing a finite residue field extension. 

 In particular, we can regard $\mathcal{O}_{X,x}$ as the local ring of a (flat) finite type $W$-scheme $X'$ at a closed point $x' \in X'_{\kappa(W)}$; if $X$ is normal and $V$-flat, then $X'$ can be taken to be normal and $V$-flat.
\end{lemma}

\begin{proof}
Let $t_1,...,t_r \in \mathcal{O}_{X,x}$ lift a transcendence basis for $\kappa(x)/\kappa(V)$. These functions give a map $\mathrm{Spec}(\mathcal{O}_{X,x}) \to \mathbf{A}^r_V$. On the special fibre, this map has image in the generic point of $\mathbf{A}^r_{\kappa(V)}$: the closed point goes to the generic point by construction, and hence all points must go to the generic point. As $\mathcal{O}_{X,x}$ is local, it follows that the previous map factors as $\mathrm{Spec}(\mathcal{O}_{X,x}) \to \mathrm{Spec}(W) \to \mathbf{A}^r_V$, where $W = \mathcal{O}_{\mathbf{A}^r_V,\eta}$ with $\eta$ denoting the generic point of the special fibre. The ring $W$ is the localization of a noetherian normal domain at a height $1$ prime, and hence is a DVR. The resulting factorization $V \to W \to \mathcal{O}_{X,x}$ solves the problem in the first part of the lemma.

To obtain $(X',x')$ in the second part, we use the following observation: given a finite type scheme $Z$ over a field $k$ and a point $z \in Z$, the point $z$ is closed if and only if $\kappa(z)$ is finite over $k$. One applies this to the special fibre of a finite type $W$-scheme $X'$ equipped with a point $x'$ such that $\mathcal{O}_{X',x} \simeq \mathcal{O}_{X,x}$ as $W$-algebras. The final statement is clear from the construction. 
\end{proof}

The goal of the rest of this section is to prove $(P_n)$ for all $n$.

\subsection{An informal outline of the proof}
\label{ss:OutlineProof}

In this subsection, we give a summary of the main steps of the proof of $(P_n)$. We stress that the summary is informal and contains a few simplifications; the impatient reader is encouraged to skip this subsection and move on to \S \ref{ss:ProvePn}.

Say $V$ is a $p$-henselian and $p$-torsionfree DVR with absolute integral closure $\overline{V}$, so the $p$-completion $\widehat{\overline{V}}$ is a perfectoid ring; write $(A,(d)) = (A_{\inf}(\widehat{\overline{V}}), \ker(\theta))$ for its perfect prism.  Let $X = \mathbf{P}^n_{\overline{V}}$ with fixed absolute integral closure $\pi:X^+ \to X$. We shall regard all sheaves as implicitly pushed forward to $X$. Our task in this subsection is to sketch the strategy used in \S \ref{ss:ProvePn} to prove the following assertion by induction on $n$:

\begin{itemize}
\item[$(\ast):$] For every closed point $x \in X_{p=0}$, we have $H^i_x(\mathcal{O}_{X^+}) = 0$ for $i < n+1$. 
\end{itemize}

Using an approximation argument, one shows that the validity of $(\ast)$ for all $n$ implies $(P_n)$ for all $n$.  We shall prove $(\ast)$ by first tilting to characteristic $p$, and then using the Frobenius. To tilt, observe that $X^+$ is perfectoid on $p$-completion (see Lemma~\ref{AICPerfectoid}), so the Bockstein sequences for $d$ and $p$ reduce $(\ast)$ to:

\begin{itemize}
\item[$(\ast)^\flat:$] For every closed point $x \in X_{p=0}$, we have $H^i_x(\mathcal{O}_{X^+}^\flat) = 0$ for $i < n+1$. 
\end{itemize}

Slight care is needed to define the local cohomology of $\mathcal{O}_{X^+}^\flat$ (see Construction~\ref{LocalCohFormal}). For the rest of this subsection, fix a closed point $x \in X_{p=0}$. Our strategy is to prove $(\ast)^\flat$ by expressing $\mathcal{O}_{X^+}^\flat$ as the colimit of the prismatic cohomology of finite covers of $X$; the importance of this approximation is that it is Frobenius equivariant. Thus, consider the category $\mathcal{P}_X^{fin}$ of all finite normal covers $Y \to X$ of $X$ dominated by $X^+$. For $Y \in \mathcal{P}_X^{fin}$, write $\Prism^n_Y$ for the prismatic complex of $\widehat{Y}$  as in Construction~\ref{ConsPrism} (the superscript $n$ emphasizes that we have no log structures at the moment). We then have
 \[ \colim_{\mathcal{P}_X^{fin}} \Prism^n_Y/p \simeq \colim_{\mathcal{P}_X^{fin}} \Prism^n_{Y,\perf}/p \simeq \mathcal{O}_{X^+}^\flat\]
 in $D_{comp}(A/p)$ by the compatibility of (perfectified) prismatic cohomology with filtered colimits as well as the perfectoidness of the $p$-completion of $X^+$. The following result, deduced from the almost Cohen--Macaulayness result in Theorem~\ref{ACMaic}, then yields $(\ast)^\flat$ in the almost category:

\begin{enumerate}
\item The ind-object $\{H^i_x(\Prism^n_{Y,\perf}/p)\}_{Y \in \mathcal{P}_X^{fin}}$ is almost zero for $i < n+1$.
\end{enumerate}

To proceed further, it is convenient to use a slightly different formulation of (1). Consider the category $\mathcal{P}_X$ of all alterations $Y \to X$ equipped with a lift of the geometric generic point of $X$. Thus, $\mathcal{P}_X^{fin}$ is a full subcategory of $\mathcal{P}_X$. One also has the full subcategory $\mathcal{P}_X^{ss} \subset \mathcal{P}_X$ spanned by  all alterations which are semistable over $\overline{V}$; this category is cofinal by de Jong's theorems \cite{deJongAlterations}.  For $Y \in \mathcal{P}_X^{ss}$, write $\Prism_Y$ for the log prismatic complex of $\widehat{Y}$ as in \cite{CesnavicusKoshikawa}. The methods used to prove (1) adapt to show (see Proposition~\ref{CMFactor}):

\begin{enumerate}[resume]
\item The ind-object $\{H^i_x(\Prism_{Y,\perf}/p)\}_{Y \in \mathcal{P}_X^{ss}}$ is almost zero for $i < n+1$.
\end{enumerate}

Note that replacing $\mathcal{P}_X^{fin}$ with $\mathcal{P}_X^{ss}$ has no effect on the limit: using Theorem~\ref{KillCohMixed} and the Hodge-Tate comparison, one can show that 
\[ \colim_{\mathcal{P}_X^{ss}} \Prism_{Y,\perf}/p \simeq \colim_{\mathcal{P}_X^{fin}} \Prism^n_{Y,\perf}/p \simeq \mathcal{O}_{X^+}^\flat\]
in $D_{comp}(A/p)$ via a natural zigzag (see Theorem~\ref{IndObjectPairs}). To improve (2) to an actual vanishing, we need to connect this discussion to coherent cohomology to access the finiteness coming by induction from Lemma~\ref{IndCMPuncturedNN}. For this connection, it is better to work with prismatic cohomology itself rather than its perfection: the former carries the Hodge-Tate comparison. To this end, we observe that the isogeny theorem for log prismatic cohomology combined with (2) yields the following (see Claim~\ref{CMclaim1}):

\begin{enumerate}[resume]
\item There exists an integer $c = c(n) \geq 1$ such that the ind-object $\{H^i_x(\Prism_Y/p)\}_{Y \in \mathcal{P}_X^{ss}}$ is ind-isomorphic to a $d^c$-torsion object for $i < n+1$.
\end{enumerate}

We can now connect to coherent cohomology using the surjection 
\[ \partial_{d^c}:H^{i-1}_x(\Prism_Y/(p,d^c)) \twoheadrightarrow H^i_x(\Prism_Y/p)[d^c]\] 
coming from the Bockstein sequence for $d^c$. Indeed, $\Prism_Y/(p,d)$ is coherent complex on $Y_{p=0}$ and closely related to $\mathcal{O}_Y/p$: the difference is given by differential forms on $Y$, and these  vanish locally in the tower $\mathcal{P}_X^{ss}$ as one can extract $p$-th roots of local functions in this tower (see Lemma~\ref{AICPerfectoid}). Using an inductive argument to pass from $\Prism_Y/(p,d)$ to $\Prism_Y/(p,d^c)$  (see Lemma~\ref{IncreaseFinite}), one can propagate the finiteness provided by Lemma~\ref{IndCMPuncturedNN} through the surjection $\partial_{d^c}$ mentioned above to obtain the following from (3) (see Claim~\ref{CMclaim2}):

\begin{enumerate}[resume]
\item For $Y \in \mathcal{P}_X^{ss}$, the image of $H^i_x(\Prism_Y/p) \to H^i_x(\mathcal{O}_{X^+}^\flat)$ is contained in a finitely generated $\overline{V}^\flat$-module for $i < n+1$. 
\end{enumerate}

The maps in (4) are Frobenius equivariant by construction, and their colimit over $Y \in \mathcal{P}_X^{ss}$ exhausts the target. Thus, to prove $(\ast)^\flat$, it is enough to show that $H^i_x(\mathcal{O}_{X^+}^\flat)$ has no nonzero Frobenius stable submodules contained within finitely generated $\overline{V}^\flat$-modules. This reduces to a vanishing theorem in local \'etale cohomology by Artin-Schreier theory, which is deduced from Proposition~\ref{EtaleCohAIC} by tilting (see Lemma~\ref{EquationCM}).

\subsection{Proof of $(P_n)$}
\label{ss:ProvePn}

In this section, we prove Theorem~\ref{KillCohFinite} by verifying that $(P_n)$ always holds true. The strategy is to access local cohomology of the structure sheaf via local cohomology of prismatic complexes, and then to use additional structures on the latter (such as a Frobenius action as well as a connection to \'etale sheaf theory via the Riemann-Hilbert functor) to annihilate cohomology after finite covers. To use the prismatic theory, we need to work over a perfectoid base, so we introduce the following notation:

\begin{notation} 
\label{not:geomCM}
We shall use the following notation throughout this section:
\begin{enumerate}
\item (The base) Let $V$ be a $p$-torsionfree $p$-henselian excellent DVR with residue field $k$, and let $\overline{V}$ be a fixed absolute integral closure of $V$, so the residue field $\overline{k}$ of $V$ is an algebraic closure of $k$. Then $\widehat{\overline{V}}$ is a perfectoid rank $1$ valuation ring with algebraically closed fraction field $C$; write $(A_{\inf}, (d))$ for the corresponding perfect prism. All occurrence of almost mathematics are in the usual context (see Notation~\ref{GlobalNotation}).

\item (The geometric object) Fix an integer $n \geq 1$.  Let $X = \mathbf{P}^n_{\overline{V}}$. Let $K$ be the function field of $X$, and choose an algebraic closure $\overline{K}$ of $K$; let $\pi:X^+ \to X$ denote the normalization of $X$ in $\overline{K}$. Write $\Prism_X \in D_{comp,qc}(X)$ for the prismatic complex of the $p$-completion of $X$ in the sense of \cite{BMS1,BhattScholzePrism}. Write $D_{comp,qc}(X,\Prism_X)$ for the full subcategory of $(p,d)$-complete $\Prism_X$-complexes on $X$ which are quasi-coherent modulo $(p,d)$ when regarded as complexes on $X_{p=0}$ (via restriction of scalars along the Hodge-Tate structure map $\mathcal{O}_X/p \to \Prism_X/(p,d)$).
\end{enumerate}
  For proper maps $f:Y \to X$, we shall abusively regard sheaves on $Y$ as sheaves on $X$ via (derived) pushforward, e.g., $\mathcal{O}_Y \in D_{qc}(X)$ denotes $Rf_* \mathcal{O}_Y$, etc. 
\end{notation}

The assertion in $(P_n)$ concerns all finite covers of projective space. We will use prismatic techniques to study these covers. As prismatic cohomology is best behaved in the (log) smooth context, it will be convenient to work with all alterations instead of merely the finite covers.

\begin{definition}[Alterations over $X$]
\label{SSAlt}
We shall use the following categories of alterations over $X$:

\begin{enumerate}
\item Let $\mathcal{P}_X$ be the category of pairs  $(f_Y:Y \to X, \eta_Y:\mathrm{Spec}(\overline{K}) \to Y)$, where $Y$ is a proper integral $\overline{V}$-scheme, $f_Y$ is an alteration, and $\eta_Y$ is an $X$-map (and thus dominant); usually the data of $\eta_Y$ will be suppressed from the notation, and we shall simply write $(f_Y:Y \to X) \in \mathcal{P}_X$ or even $Y \in \mathcal{P}_X$ for such an object.

\item Let $\mathcal{P}_X^{fin} \subset \mathcal{P}_X$ be the full subcategory spanned by those $Y$ which are finite over $X$; the cofinal subcategory of normal schemes in $\mathcal{P}_X^{fin}$ identifies with the category of finite extensions $L/K$ contained in $\overline{K}$ by taking the function field.

\item Let $\mathcal{P}_X^{ss} \subset \mathcal{P}_X$ be the full subcategory spanned by those $Y$'s which are semistable over $\overline{V}$ in the sense of \cite[\S 1.5]{CesnavicusKoshikawa}. For $(f_Y:Y \to X) \in \mathcal{P}_X^{ss}$, write $\Prism_Y$ for the {\em log prismatic complex of $Y$} (i.e., the log $A_{\inf}$-complex denoted $A\Omega$ in \cite[2.2.3]{CesnavicusKoshikawa} and recalled in Remark~\ref{LogPrismaticDef}). Following our conventions from Notation~\ref{not:geomCM}, the object $\Prism_Y$ is regarded as a Frobenius module in $D(X,\Prism_X)$ via pushforward along $f_Y$, and  lies in $D_{comp,qc}(X,\Prism_X)$ by Theorem~\ref{LogPrismatic} (2) below. 
\end{enumerate}

In the sequel, we shall often use that any $Y \in \mathcal{P}_X$ is $\overline{V}$-flat (as $Y$ is integral with function field of characteristic $0$) and thus finitely presented over both $\overline{V}$ and $X$ by \cite[Tag 053E]{StacksProject}. 

\end{definition}

\begin{remark}[Logarithmic invariants]
Any $Y \in \mathcal{P}_X$ has a natural log structure determined by the open subset $Y[1/p] \subset Y$. Unless otherwise specified, we shall always interpret all invariants of $Y$, such as the sheaf $\Omega^i_{Y/\overline{V}}$ of differential $i$-forms, in the logarithmic sense; this convention was already observed in Definition~\ref{SSAlt} (3).
\end{remark}

\begin{remark}[Logarithmic vs usual prismatic cohomology]
\label{LogPrismaticDef}
Let $Y \in \mathcal{P}_X^{ss}$. Recall that the definition in \cite[2.2.3]{CesnavicusKoshikawa} is
\[ \Prism_Y := L\eta_\mu R\nu_* A_{\inf,\widehat{Y}_C^{ad}} \in D(\widehat{Y}, A_{\inf}),\] 
where $\widehat{Y}_C^{ad}$ denotes the adic generic fibre of the formal completion $\widehat{Y}$ and is equipped with the pro-\'etale topology,  the map $\nu:\left(\widehat{Y}_C^{ad}\right)_{proet} \to \widehat{Y}$ is the nearby cycles map, the element $\mu = [\underline{\epsilon}]-1$ is built from a choice of compatible system of $p$-power roots of $1$ as in \cite[\S 3.2]{BMS1}, and the functor $L\eta_\mu$ is the Berthelot-Ogus(-Deligne) decalage functor discussed in \cite[\S 6]{BMS1}. In particular, this definition shows that $\Prism_Y$ is manifestly functorial in $Y$ regarded merely as a $\overline{V}$-scheme (rather than as a log scheme) and makes sense for any $Y \in \mathcal{P}_X$. For any $Y \in \mathcal{P}_X$, write $\Prism_Y^n$ for the non-logarithmic derived prismatic complex of $\widehat{Y}$ as constructed in \cite[\S 7.2]{BhattScholzePrism}. For $Y/\overline{V}$ smooth, we have $\Prism_Y^n \simeq L\eta_\mu R\nu_* A_{\inf,\widehat{Y}_C^{ad}}$ via the comparison with \cite{BMS1} proven in \cite[\S 17]{BhattScholzePrism}, so $\Prism_Y^n \simeq \Prism_Y$ and thus there is no notational clash between Notation~\ref{not:geomCM} (2) and Definition~\ref{SSAlt} (3).  By left Kan extension from the smooth case, it follows that for any $Y \in \mathcal{P}_X^{ss}$, we have a natural map $\Prism_Y^n \to \Prism_Y$ for any $Y \in \mathcal{P}_X$.
\end{remark}

\begin{remark}[The perfection of log prismatic cohomology]
\label{LogPrismPerfRH}
Fix $Y \in \mathcal{P}_X^{ss}$. The construction from \cite{CesnavicusKoshikawa} recalled in Remark~\ref{LogPrismaticDef} has the feature that the $L\eta_\mu$ operation is almost undone by passage to the perfection:

\begin{claim}
\label{LetaPerfClaim}
The map $ \Prism_{Y,\perf} = \left( L\eta_\mu R\nu_* A_{\inf,\widehat{Y}_C^{ad}}\right)_\perf \to R\nu_* A_{\inf,\widehat{Y}_C^{ad}}$ is an almost isomorphism; here the left side denotes the $(p,d)$-completed perfection.
\end{claim}

Generalities on prismatic cohomology also show that the natural map  $\Prism_{Y,\perf}^n \to R\nu_* A_{\inf,\widehat{Y}_C^{ad}}$ is an almost isomorphism for any $Y \in \mathcal{P}_X$. Consequently, the map $\Prism_{Y,\perf}^n \to \Prism_{Y,\perf}$ is an almost isomorphism for $Y \in \mathcal{P}_X^{ss}$.  Using this observation and Theorem~\ref{thm:RH} (1) - (3) (or rather the variants in Remark~\ref{rmk:AinfLift}), we learn that the object $\Prism_{Y,\perf}/p \in D_{comp,qc}(X, \Prism_X/p)$ is naturally almost identified with $\RH_{\Prism}(Rf_* \mathbf{F}_{p,Y})$, where $f:Y[1/p] \to X[1/p]$ is the structure map on generic fibres.

\begin{proof}[Proof of Claim~\ref{LetaPerfClaim}]
This is a general fact for any $(p,d)$-complete perfect $\varphi$-complex $(K,\varphi_K:K \simeq \varphi_* K)$ over $A_{\inf}$ with $K/\mu \in D^{\geq 0}$. Since we could not find a reference, let us sketch a proof. First, by the compatibility of all operations involved with filtered colimits, we may assume $K$ is bounded above; by $\mu$-completeness of $K$, it follows that $K$ is bounded. Given such a $K$, we have (by \cite[Proposition 6.12]{BMS1}) a natural map $c:L\eta_\mu K \to K$ with cone having homology annihilated by $\mu^N$, where $N$ is the cohomological amplitude of $K$. Our task is to show the resulting map 
\[ c_\perf:(L\eta_\mu K)_\perf \to K\] is an almost isomorphism, where the left side denotes the $(p,d)$-completed perfection. Note that $c_\perf$ is the $(p,d)$-completed colimit of the maps 
\[ c_r:\varphi^r_* (L\eta_\mu K) \xrightarrow{\varphi^r_*(c)} \varphi^r_* K \stackrel{\varphi^{-r}_K}{\simeq} K.\] 
Now the cone $c_r$ has homology annihilated by $\phi^{-r}(\mu)^N$ by Frobenius twisting the corresponding statement for $r=0$. Since $\phi^{-(r+1)}(\mu) \mid \phi^{-r}(\mu)$ for all $r$, we learn that the cone of $c_\perf$ has cohomology annihilated by the $(p,d)$-completion of the ideal $\cup_r (\phi^{-r}(\mu)^N)$. But this completion coincides with the ideal $W(\mathfrak{m}^\flat) \subset A_{\inf}$ of almost mathematics by \cite[Lemma 9.2]{BMS2}, so the claim follows.
\end{proof}
\end{remark}

We shall use the following theorem to ensure $\mathcal{P}_X^{ss}$ is big enough:

\begin{theorem}[Existence of semistable alterations (de Jong)]
\label{AlterationsExist}
 $\mathcal{P}_X^{ss}$ is cofinal in $\mathcal{P}_X$.
\end{theorem}
\begin{proof}
It is enough to show that any proper finitely presented integral flat $\overline{V}$-scheme $Y$ admits an alteration $Y' \to Y$ with $Y'$ semistable as in \cite{CesnavicusKoshikawa}. By  approximation arguments, it is enough to show the same for $\overline{V}$ replaced by a $p$-henselian $p$-torsionfree excellent DVR $W$. Such an alteration exists after base change to the completion $\widehat{W}$ by de Jong's theorem \cite{deJongAlterations}. By excellence and henselianness of $W$, we can then descend such an alteration to $W$ (as is also sketched in \cite[\S 1.5]{CesnavicusKoshikawa}).
\end{proof}

Let us describe the limit of the spaces appearing in $\mathcal{P}_X$ in classical terms.

\begin{lemma}[The Riemann-Zariski space of $X^+$]
\label{RZAIC}
Let $\widetilde{X^+} := \lim_{Y \in \mathcal{P}_X} Y$, computed in locally ringed spaces.
\begin{enumerate}
\item $\mathcal{P}_X$, $\mathcal{P}_X^{fin}$ and $\mathcal{P}_X^{ss}$ are all cofiltered posets.

\item The locally ringed space $\lim_{Y \in \mathcal{P}_X^{fin}} Y$ is naturally identified with $X^+$.

\item The map $\widetilde{X^+} \to X^+$ resulting from (2) is naturally identified with the Riemann-Zariski space of $X^+$. In particular, this map is an inverse limit of normalized blowups of $X^+$. 
\end{enumerate}
\end{lemma}
\begin{proof}
(1): To show these categories are posets, it is enough to check $\mathcal{P}_X$ is so. But the $\overline{K}$-point being recorded in each object of $\mathcal{P}_X$ rigidifies the picture as dominant maps between integral schemes are determined by the induced map on function fields.  Moreover, as both $\mathcal{P}_X$ and $\mathcal{P}_X^{fin}$ admit fibre products (given by taking suitable irreducible components of scheme-theoretic fibre products), these posets are cofiltered. The fact that $\mathcal{P}_X^{ss}$ is cofiltered then follows from Theorem~\ref{AlterationsExist}.

(2): this follows because $\mathcal{P}_X^{fin}$ identifies with the category of finite extensions $L/K$ contained in $\overline{K}$ via the functor carrying such an extension $L$ to the normalization of $X$ in $L$. 

(3): The first part is standard, and the second part follows by the Raynaud-Gruson flattening theorem
\end{proof}

The following result summarizes what we need from the log prismatic theory:

\begin{theorem}[Log prismatic cohomology (Cesnavicus-Koshikawa)]
\label{LogPrismatic}
Fix $Y \in \mathcal{P}_X^{ss}$. 
\begin{enumerate}
\item (Isogeny theorem) There exists some $c = c(n)$ such that map $\Prism_Y/p \to \Prism_{Y,\perf}/p$ admits left and right inverses up to multiplication by $d^c$ in $D_{comp}(X_{p=0},A_{\inf}/p)$. 
\item (Hodge-Tate comparison) There is a natural isomorphism $H^*(\Prism_Y/(p,d)) \simeq \Omega^*_{Y/\overline{V}}/p$ of graded algebras over $X_{p=0}$. In particular, $\Prism_Y \in D_{comp,qc}(X, \Prism_X)$. 
\end{enumerate}
\end{theorem}
\begin{proof}
(2) is \cite[Theorem 4.11]{CesnavicusKoshikawa}. The analog of part (1) for $\mu = [\underline{\epsilon}]-1$ (as in \cite[\S 3.2]{BMS1}) instead of $d$ is immediate from the definition $\Prism_Y$ as $L\eta_\mu R\nu_* A_{\inf,\widehat{Y}^{ad}_C}$:  a general property of the $L\eta$ construction from \cite[Proposition 6.12]{BMS1} and the cohomological amplitude bound coming from (2) ensure that the natural map $L\eta_\mu R\nu_* A_{\inf,\widehat{Y}^{ad}_C} \to R\nu_* A_{\inf,\widehat{Y}_C^{ad}}$ (which almost identifies with $\Prism_Y \to \Prism_{Y,\perf}$ by Remark~\ref{LogPrismPerfRH})  admits an inverse up to multiplication by $\mu^n$ on either side. The claim in (1) then follows by reducing modulo $p$ and noting that the ideals $(d)$ and $(\mu)$ agree up to radicals in the rank $1$ valuation ring $A_{\inf}/p =\overline{V}^\flat$.
\end{proof}

\begin{remark}[Why do we need semistable alterations?]
Theorem~\ref{LogPrismatic} is the main reason we pass from finite covers of $X$ to semistable alterations over $X$ in the proof of $(P_n)$. Indeed, the isogeny theorem effectively bounds the difference between prismatic cohomology and its perfection in prismatic cohomology by a fixed power of $d$, which allows us to translate certain questions from $\Prism_Y/p$ to $\Prism_Y/(p,d^c)$ for some fixed $c \geq 1$ thanks to the almost vanishing theorem we have already proven (see Proposition~\ref{CMFactor}). The Hodge-Tate comparison then helps relate $\Prism_Y/(p,d^c)$ to coherent cohomology (see Lemma~\ref{IncreaseFinite}).

It is tempting to  avoid restricting to semistable alterations (and thus avoid  \cite{deJongAlterations} or \cite{CesnavicusKoshikawa}) by simply using the object $L\eta_\mu R\nu_* A_{\inf,\widehat{Y}^{ad}_C}$ for all $Y \in \mathcal{P}_X$ instead of merely $Y \in \mathcal{P}_X^{ss}$: the analog of the isogeny theorem holds true for this object by the proof of Theorem~\ref{LogPrismatic} (1). However, the lack of a description of $\left(L\eta_\mu R\nu_* A_{\inf,\widehat{Y}^{ad}_C}\right)/(p,d)$ via coherent cohomology for general $Y \in \mathcal{P}_X$ prevents us from following this path.
\end{remark}

We now begin proving results that we shall need in order to prove $(P_n)$. First, we explain how to move between coherent cohomology, prismatic cohomology, and log prismatic cohomology in the towers given by the categories in Definition~\ref{SSAlt}.

\begin{theorem}
\label{IndObjectPairs}
\begin{enumerate}
\item (Cofinality of finite maps and vanishing of differential forms) The natural maps
\[ \{ \mathcal{O}_Y/p \}_{Y \in \mathcal{P}_X^{fin}} \xrightarrow{a} \{ \mathcal{O}_Y/p \}_{Y \in \mathcal{P}_X} \xleftarrow{b} \{ \mathcal{O}_Y/p \}_{Y \in \mathcal{P}_X^{ss}} \xrightarrow{c} \{ \Prism_Y/(p,d) \}_{Y \in \mathcal{P}_X^{ss}} \]
are isomorphisms of ind-objects in $D_{qc}(X_{p=0})$, and they all have colimit $\mathcal{O}_{X^+}/p$. 
\item (Asymptotic equality of logarithmic and usual prismatic cohomology)  For $Y \in \mathcal{P}_X$, write $\Prism^n_Y$ for the non-logarithmic prismatic complex of $Y$. Then the map 
\[ \{\Prism^n_Y/(p,d^c)\}_{Y \in \mathcal{P}_X^{ss}} \to \{\Prism_Y/(p,d^c)\}_{Y \in \mathcal{P}_X^{ss}}\]
is an isomorphism of ind-objects $D(X_{p=0}, A_{\inf}/(p,d^c))$ for any $c \geq 1$.
\end{enumerate}
\end{theorem}
\begin{proof}
 (1): Each term of each ind-object appearing here is a coherent complex on $X_{p=0}$ with uniformly bounded cohomological amplitude (via the coherence of higher direct images for the first three, and Theorem~\ref{LogPrismatic} (2) for the last one).  For any $a \leq b \in \mathbf{Z}$, objects in $D^{[a,b]}_{coh}(X_{p=0})$ are compact objects in $D^{[a.b]}_{qc}(X_{p=0})$. Thus, to prove the isomorphy at the level of ind-objects, it suffices to do so after taking colimits. The claim for (b) follows from Theorem~\ref{AlterationsExist}, while that for (a) follows from Theorem~\ref{KillCohMixed} as well as the description in Lemma~\ref{RZAIC}. For (c), using Theorem~\ref{LogPrismatic} (2), it is enough to prove that $\colim_{Y \in \mathcal{P}_X^{ss}} \Omega^1_{Y/V}/p = 0$. As the formation of differential forms is compatible with colimits, it is enough to show the $p$-divisibility of $\Omega^1_{X_i^+/\overline{V}}$ for each normalized blowup $X_i^+ \to X^+$.  Now $X_i^+$ is a normal scheme with algebraically closed fraction field, so the claim follows from Lemma~\ref{AICPerfectoid}.
 
(2): As isomorphy of ind-objects of the derived category can be detected after base change along maps with nilpotent kernels, it is enough to prove the result after base change along $A_{\inf}/(p,d^c) \to A_{\inf}/(p,\varphi^{-1}(d))$, i.e., to show that  $\{\Prism^n_Y/(p,\varphi^{-1}(d))\}_{Y \in \mathcal{P}_X^{ss}} \to \{\Prism_Y/(p,\varphi^{-1}(d))\}_{Y \in \mathcal{P}_X^{ss}}$ is an isomorphism of ind-objects in $D(X_{p=0}, A_{\inf}/(p,\varphi^{-1}(d)))$. Using (1) as well as the de Rham comparison isomorphism for usual prismatic cohomology, we are then reduced to checking the following: the natural  map induces an ind-isomorphism $\{(\mathcal{O}_Y/p)^{(1)}\}_{Y \in \mathcal{P}_X} \to \{ \mathrm{dR}^n_{Y/\overline{V}}/p \}_{Y \in \mathcal{P}_X}$, where $(-)^{(1)}$ denotes the Frobenius twist relative to $\overline{V}/p$ and $\mathrm{dR}^n$ denotes the  non-logarithmic derived de Rham complex functor from \cite{IllusieCC2} (see also \cite{BhattpadicddR}). In other words, we must show that for all $Y \in \mathcal{P}_X$, there exists a map $Y' \to Y$ and a map $\mathrm{dR}^n_{Y/\overline{V}}/p \to (\mathcal{O}_{Y'}/p)^{(1)}$ factoring the canonical transition maps $(\mathcal{O}_Y/p)^{(1)} \to (\mathcal{O}_{Y'}/p)^{(1)}$ and $\mathrm{dR}^n_{Y/\overline{V}}/p \to \mathrm{dR}^n_{Y'/\overline{V}}/p$. In fact, we shall show that we can even take $Y' \to Y$ to be finite. Indeed, consider the pro-object of maps $Y' \to Y$ in $\mathcal{P}_X$ which are finite. The inverse limit of this pro-object is an absolute integral closure of $Y$ and is thus perfectoid after $p$-completion by Lemma~\ref{AICPerfectoid}. The claim now follows by applying Lemma~\ref{FactorizeFrobAIC} to this pro-object and using the following observation (proven by left Kan extension from the smooth case): if $A \to B$ is a flat map of $\mathbf{F}_p$-algebras with relative Frobenius $B^{(1)} \to B$, then there is a natural factorization 
\[ \mathrm{dR}^n_{B^{(1)}/A} \xrightarrow{a} B^{(1)} \xrightarrow{b} \mathrm{dR}^n_{B/A} \xrightarrow{c} B,\] 
where $a$ and $c$ are given by $\mathrm{gr}^0$ of the Hodge filtration, $b$ comes from $\mathrm{gr}_0$ of the conjugate filtration, and $cb$ is the relative Frobenius, and $ba$ is the map on de Rham complexes induced by the relative Frobenius.
 \end{proof}

The following lemmas were used above.

\begin{lemma}[Perfectness properties for absolute integral closures]
\label{AICPerfectoid}
Let $R$ be a normal domain with algebraically closed fraction field. 
\begin{enumerate}
\item The multiplicative monoids $R$,  $R[1/p]^*$ and $R[1/p]^* \cap R$ are $p$-divisible.
\item The $p$-adic completion $\widehat{R}$ is a perfectoid ring.
\item Both the the $R$-module $\Omega^{1,n}_{R/\mathbf{Z}}$ of non-logarithmic Kahler differentials as well as the $R$-module $\Omega^1_{R/\mathbf{Z}}$ of logarithmic Kahler differentials (with log structure given by $\mathrm{Spec}(R[1/p]) \subset \mathrm{Spec}(R)$) is $p$-divisible.
\end{enumerate}
\end{lemma}
\begin{proof}
(1): The assumption on $R$ implies that eacg monic polynomial over $R$ has a root in $R$, so each element of $R$ has a $p$-th root. As the hypothesis on $R$ passes to localizations, and because $p$-th roots of units are units, we obtain the claim for $R[1/p]^*$. The stability of $R[1/p]^* \cap R \subset R[1/p]^*$ under $p$-th roots (coming from integral closedness of $R$) then implies that $R[1/p]^* \cap R$ is $p$-divisible.

(2): If $p=0$, then it is clear from (1) and reducedness that $R$ is perfect whence perfectoid. Assume now that $p \neq 0$ in $R$, whence $p$ is a nonzerodivisor. By (1), we can choose a $p$-th root $p^{1/p} \in R$ of $p \in R$. It is enough to show that the Frobenius induces a bijection $R/p^{1/p} \simeq R/p$. The surjectivity is  clear from (1).
For injectivity, if $x \in R$ with $x^p \in pR$, then $y = \frac{x}{p^{1/p}} \in R[1/p]$ satisfies $y^p \in R$, whence $y \in R$ by normality, so $x \in p^{1/p} R$ as wanted.

(3): The claim for $\Omega^{1,n}_{R/\mathbf{Z}}$ follows from the $p$-divisibility of $R$, while that for $\Omega^1_{R/\mathbf{Z}}$ follows from that of $\Omega^{1,n}_{R/\mathbf{Z}}$ and $S[1/p]^* \cap S$ for any \'etale $R$-algebra\footnote{Any \'etale $R$-algebra is a finite product of Zariski localizations of $R$ by \cite[Lemma 3]{GabberAffineAnalog}.} by definition of logarithmic differential forms.
\end{proof}

\begin{lemma}
\label{FactorizeFrobAIC}
Let $\{Y_i\}$ be a cofiltered system of finitely presented flat $\overline{V}$-schemes along affine transition maps. Assume that $Y_\infty := \lim_i Y_i$ is perfectoid after $p$-completion. For any $Y_i$ in the pro-system, there exists a sufficiently large map $Y_j \to Y_i$ in the pro-system and an $\overline{V}/p$-map $h:Y_{j,p=0}^{(1)} \to Y_{i,p=0}$ such that the following diagram of $\overline{V}/p$-schemes is commutative:
\[ \xymatrix{ 	Y_{j,p=0} \ar[r] \ar[d] & Y_{i,p=0} \ar[d] \\
(Y_{j,p=0})^{(1)} \ar[r] \ar[ru]^-{h}  & (Y_{i,p=0})^{(1)}, }\]
 where the vertical maps are the relative Frobenii over $\overline{V}/p$ and the horizontal maps are natural ones.
%
%
\end{lemma}
\begin{proof}
As the $p$-completion of $Y_\infty$ is perfectoid, the relative Frobenius $Y_{\infty,p=0} \to (Y_{\infty,p=0})^{(1)}$ is an isomorphism, so the lemma is trivially true without the finiteness requirement on $Y_j \to Y_i$. As both $Y_{i,p=0}$ and $(Y_{i,p=0})^{(1)}$ are finitely presented over $\overline{V}$, the lemma itself follows by approximation.
\end{proof}

Eventually, we shall annihilate local prismatic cohomology by passing up to alterations. We have essentially already seen how to do this with perfectified prismatic cohomology in the proof of the almost Cohen--Macaulayness of $X^+$. The next proposition will help us deduce a weaker statement for prismatic cohomology itself, using crucially the isogeny theorem for prismatic cohomology of (log) smooth schemes.

\begin{proposition}
\label{CMFactor}
There exists a constant $c = c(n)$ such that for any $Y \in \mathcal{P}_X^{ss}$, there is a map $f:Y' \to Y$ in $\mathcal{P}_X^{ss}$ and $K \in D_{comp,qc}(X,\Prism_X/p)$ such that the following hold true:
\begin{enumerate}
\item  Write $f^*:\Prism_Y/p \to \Prism_{Y'}/p$ for the pullback. Then $d^c f^*$ factors over $K$ in $D_{comp}(X_{p=0}, A_{\inf}/p)$. 
\item $K/d \in D_{qc}(X_{p=0})$ is cohomologically CM (so $R\Gamma_x((K/d)_x) \in D^{\geq n-\dim(\overline{\{x\}})}$ for all  $x \in X_{p=0}$).
\end{enumerate}
\end{proposition}
\begin{proof}
Take $c_0$ to be the constant in Theorem~\ref{LogPrismatic} (1).  We shall construct $K$ in the almost category satisfying (2) and prove the analog of (1) for for the map $d^{c_0} f^*$ also in the almost category; this implies the proposition as one can increase $c_0$ by $1$ to pass back to the real world. For the rest of the proof, we work in the almost category.

As $Y \in \mathcal{P}_X^{ss}$, Remark~\ref{LogPrismPerfRH} gives  $\Prism_{Y,\perf}/p \stackrel{a}{\simeq} \RH_{\Prism}(Rf_* \mathbf{F}_p)$, where $f:Y[1/p] \to X[1/p]$ is the structure map. As $f$ is a dominant quasi-finite map between smooth varieties,  we can find a non-empty affine open $U \subset X[1/p]$ such that $f$ is finite \'etale over $U$. If $j:f^{-1}(U) \to X[1/p]$ denotes the resulting quasi-finite affine map, then $Rj_* \mathbf{F}_p[n]$ is perverse on $X[1/p]$: pushforward along quasi-finite affine maps preserve perversity by \cite[Corollary 4.1.3]{BBDG}. By (the proof of) Theorem~\ref{ACMaic} (1), we can find a finite cover $Y' \to Y$ in $\mathcal{P}_X^{fin}$ such that, if $g:Y'[1/p] \to X[1/p]$ denotes the resulting map, then the pullback $Rf_* \mathbf{F}_{p,Y} \to Rg_* \mathbf{F}_{p,Y'}$ factors over $Rf_* \mathbf{F}_{p,Y} \to Rj_* \mathbf{F}_{p,U}$. By replacing $Y'$ if necessary and picking an embedding of its function field in $\overline{K}$, we may also assume $Y' \in \mathcal{P}_X^{ss}$. Applying Theorem~\ref{thm:RH} (or rather the variant in Remark~\ref{rmk:AinfLift}), we have a commutative diagram
\[ \xymatrix{ \Prism_Y/p \ar[r] \ar[dd] & \Prism_{Y,\perf}/p \simeq \RH_{\Prism}(Rf_* \mathbf{F}_{p,Y}) \ar[d] \\
  & K :=\RH_{\Prism}(Rj_* \mathbf{F}_{p,U}) \ar[d] \\
  \Prism_{Y'}/p \ar[r] & \Prism_{Y',\perf}/p \simeq \RH_{\Prism}(Rg_* \mathbf{F}_{p,Y'}) }\]
in the almost category. By our choice of $j$, the object $Rj_* \mathbf{F}_{p,U}[n]$ is perverse, so $K/d \simeq \RH_{\overline{\Prism}}(Rj_* \mathbf{F}_p) \in D_{qc}(X_{p=0})^a$ has the desired property in (2) by Corollary~\ref{cor:PervACM}. As $Y' \in \mathcal{P}_X^{ss}$, the bottom horizontal map admits an inverse up to multiplication by $d^c$ as an $A_{\inf}/p$-complex on $X_{p=0}$ by Theorem~\ref{LogPrismatic} (1). The diagram above then yields the desired statement.
\end{proof}

Let us explain how to compute local cohomology of prismatic cohomology; 

\begin{construction}[Local cohomology of $\Prism_X$-complexes]
\label{LocalCohFormal}
 Fix $K \in D_{comp,qc}(X,\Prism_X/p)$ as well as a constructible closed subset $Z \subset X_{p=0}$. We shall explain how to define $R\Gamma_Z(K) \in D(X,\Prism_X/p)$.

Let $D_{nilp,qc}(X,\Prism_X/p) \subset D(X,\Prism_X/p)$ be the full subcategory spanned by objects that are quasi-coherent modulo $d$ and have $d^\infty$-torsion homology sheaves. The complete-torsion equivalence of Dwyer-Greenlees in form presented in \cite[0A6X]{StacksProject} shows that the functors local cohomology along $d$ and derived $d$-completion give mutually inverse equivalences $D_{comp,qc}(\Prism_X/p) \simeq D_{nilp,qc}(X,\Prism_X/p)$. 
 
Define $R\Gamma_Z(K) \in D_{qc}(X,\Prism_X/p) \in D_{nilp,qc}(X,\Prism_X/p)$ as the unique object whose $d$-completion in $D_{comp,qc}(X,\Prism_X/p)$ identifies with $R\Gamma_Z(K)^{\wedge} := R\lim_n R\Gamma_Z(K/d^n)$, where $R\Gamma_Z(K/d^n)$ is defined in the usual way as cohomology with supports in a closed set \cite[Tag 0A39]{StacksProject}; Explicitly, this recipe amounts to the following:
 \[ R\Gamma_Z(K) = R\Gamma_Z(K)^{\wedge} \otimes^L_{A_{\inf}} \left(A_{\inf}[1/d]/A_{\inf}\right)[-1] \simeq \colim_m R\Gamma_Z(d^{-m}K/K)[-1].\]
 We shall regard this construction as an exact functor $R\Gamma_Z(-):D_{comp,qc}(X,\Prism_X/p) \to D_{nilp,qc}(X,\Prism_X/p)$. 
 \end{construction}

\begin{remark}
Some further comments on Construction~\ref{LocalCohFormal} to help demystify the construction. 
\begin{enumerate}
\item Say $K \in D_{qc}(X, \Prism_X/(p,d))$; this gives $K' \in D_{qc}(X_{p=0})$ by restriction of scalars along $\mathcal{O}_X/p \to \Prism_X/(p,d)$. Then $R\Gamma_Z(K)$ as constructed in Construction~\ref{LocalCohFormal} agrees with the complex $R\Gamma_Z(K')$ as defined via local cohomology as the closed subset $Z \subset X_{p=0}$ is constructible (see \cite[Tag 0A6T]{StacksProject}).

\item The functor $R\Gamma_Z(-):D_{comp,qc}(X,\Prism_X/p) \to D_{nilp,qc}(X,\Prism_X/p)$  commutes with all colimits: this is well-known for the usual local cohomology functors, and follows in our case as the operation of taking local cohomology along a finitely generated ideal $I$ is insensitive to completing the input along a finitely generated subideal $J \subset I$.

\item The construction of $R\Gamma_Z(-)$ given Construction~\ref{LocalCohFormal} makes sense and is functorial in all $K \in D_{comp}(X_{p=0}, A_{\inf}/p)$, i.e., we do not need the additional linearity over $\Prism_X/p$ in Construction~\ref{LocalCohFormal}. This follows from the formula for $R\Gamma_Z(K)$ given in Construction~\ref{LocalCohFormal}. In particular, $R\Gamma_Z(-)$ is functorial with respect to the implicit inverse maps appearing in Theorem~\ref{LogPrismatic} (1) as well as Proposition~\ref{CMFactor} (1).

\end{enumerate}
\end{remark}

The following lemma provides an important technical ingredient in our proof of $(P_n)$ as it allows us to propagate some finiteness properties from the structure sheaf $\mathcal{O} = \mathrm{gr}^{HT}_0(\Prism/d)$ to $\Prism/d^c$.

\begin{lemma}
\label{IncreaseFinite}
 Fix a maximal ideal $\mathfrak{m} \subset \mathcal{O}_X$.  Assume that for any $Y \in \mathcal{P}_X^{ss}$ there exist some $Y' \to Y$ in $\mathcal{P}_X^{ss}$ such that 
 \[ H^i_{\mathfrak{m}}(\mathcal{O}_Y/p) \to H^i_{\mathfrak{m}}(\mathcal{O}_{Y'}/p)\]
 factors over a finitely presented over $\overline{V}^\flat/d$-module for $i < n$. Then for any $Y \in \mathcal{P}_X^{ss}$ and any $c \geq 1$, there is some $Y' \to Y$ in $\mathcal{P}_X^{ss}$ such that 
\[ H^i_{\mathfrak{m}}(\Prism_{Y}/(p,d^c)) \to H^i_{\mathfrak{m}}(\Prism_{Y'}/(p,d^c))\]
factors over a finitely presented $\overline{V}^\flat$-module for $i < n$.
\end{lemma}

\begin{proof}
Let us first informally explain the (simple) strategy: we filter $\Prism_Y/(p,d^c)$ by $c$ copies of $\Prism_Y/(p,d)$, use Theorem~\ref{IndObjectPairs} (1) to pass from the assumption on $\mathcal{O}_Y/p$ to one on $\Prism_Y/(p,d)$, and iterate the construction finitely many times to obtain the desired conclusion. However, since we need to restrict attention to local cohomology outside the top degree, one must confront the non-noetherianness of the situation, which adds technical complexity to the argument below.

By Theorem~\ref{IndObjectPairs} (2), it is enough to prove the analogous statement for the non-logarithmic derived prismatic complexes $\Prism_Y^n$ instead of the logarithmic ones\footnote{The reason we switch to non-logarithmic derived prismatic complexes is that one step of the proof requires us  to descend from $A_{\inf}$ to an imperfect (Breuil-Kisin) prism to access  noetherianness, but the logarithmic theory in \cite{CesnavicusKoshikawa} only works over (some) perfect prisms. The paper \cite{KoshikawaLogPrisms}, which appeared while the current paper was being prepared, gives a definition of logarithmic prismatic cohomology over a Breuil-Kisin prism with suitable base change properties. It is likely that by using the theory in \cite{KoshikawaLogPrisms} one could run the entire argument using log prismatic cohomology, thus avoiding Theorem~\ref{IndObjectPairs} (2).}. As the map $\{\mathcal{O}_Y/p\}_{Y \in \mathcal{P}_X^{ss}} \to \{\Prism^n_Y/(d,p)\}_{Y \in \mathcal{P}_X^{ss}}$ is an isomorphism of ind-objects (Theorem~\ref{IndObjectPairs} (1) and (2)), we learn that for any $Y \in \mathcal{P}_X^{ss}$ there exist some $Y' \to Y$ in $\mathcal{P}_X^{ss}$ such that 
\[ H^i_{\mathfrak{m}}(\Prism^n_{Y}/(p,d)) \to H^i_{\mathfrak{m}}(\Prism^n_{Y'}/(p,d))\]
factors over a finitely presented $\overline{V}^\flat$-module for $i < n$. This handles the $c=1$ case, and we shall deduce the rest by a filtering argument. Assume from now that $c \geq 2$.

For any $Y \in \mathcal{P}_X^{ss}$, regard $\tau^{\leq n-1} R\Gamma_{\mathfrak{m}}(\Prism^n_Y/(p,d^c))$ as an object of $D(\overline{V}^\flat)$ endowed with the finite decreasing filtration given by powers of $d$, i.e., set
\[ \mathrm{Fil}^i_Y  = \tau^{\leq n-1} R\Gamma_{\mathfrak{m}}(d^i \Prism^n_Y/(p,d^c)) \simeq \tau^{\leq n-1} R\Gamma_{\mathfrak{m}}(\Prism^n_Y/(p,d^{c-i})),\]
so $\mathrm{Fil}^i_Y = 0$ for $i \leq c$ and $\mathrm{Fil}^0_Y = \tau^{\leq n-1} R\Gamma_{\mathfrak{m}}(\Prism^n_Y/(p,d^c))$. The associated graded pieces $\mathrm{gr}^i_Y$ are zero unless $i \in [0,c-1]$. Moreover, using long exact sequence associated to the exact triangle,
\[ R\Gamma_{\mathfrak{m}}( d^{i+1} \Prism^n_Y/(p,d^c) ) \to R\Gamma_{\mathfrak{m}} (d^i \Prism^n_Y/(p,d^c)) \to R\Gamma_{\mathfrak{m}} (\Prism^n_Y/(p,d))\]
we learn that $H^j \mathrm{gr}^i$ is zero unless $j \in [0,...,n-1]$ and is explicitly given by
\begin{equation}
\label{SS1}
H^j \mathrm{gr}^i_Y \simeq H^j_{\mathfrak{m}}(\Prism^n_Y/(p,d)) \quad \text{ if } i \in [0,c-1], j \in [0,n-2]  \text{ or if }   i=c-1,  j \in [0,n-1]
\end{equation}
and
\begin{equation}
\label{SS2}
 H^{n-1}  \mathrm{gr}^i_Y  = \mathrm{ker} \left( H^{n-1}_{\mathfrak{m}}(\Prism^n_Y/(p,d)) \xrightarrow{\delta} H^n_{\mathfrak{m}}(\Prism^n_Y/(p,d^{c-(i+1)})) \right) \quad \text{for } i \in [0,c-2].
 \end{equation}
In particular, the resulting spectral sequence for a filtered complex (normalized as in \cite[Tag 012M]{StacksProject}) has $E_1^{p,q} = 0$ unless $p \in [0,n-1]$ and $q \in [-(c-1),n-1]$. Thus, the spectral sequence is located in a rectangle whose long side has length $N = n-1 + c-1 = c+n-2$. We shall now deduce the lemma by applying \cite[Lemma 10.5.6]{ScholzeWeinsteinBerkeley} using this particular value of $N$.

Given $Y = Y_0$, for $0 \leq k < (N+1)3^{N+1}$, choose maps $Y_{k+1} \to Y_k$ such that 
\[ H^j_{\mathfrak{m}}(\Prism^n_{Y_k}/(p,d)) \to H^j_{\mathfrak{m}}(\Prism^n_{Y_{k+1}}/(p,d))\]
factors over a finitely presented $\overline{V}^\flat$-module for $j < n$; this is possible thanks to the $c=1$ case already proven above. We claim that setting $Y' = Y_{(N+1)3^{N+1}}$ does the job by  \cite[Lemma 10.5.6]{ScholzeWeinsteinBerkeley}. To apply this lemma, we need to show the following: for each map $Y_{k+1} \to Y_k$, the induced map
\begin{equation}
\label{SS3}
H^{j} \mathrm{gr}^i_{Y_k} \to H^j \mathrm{gr}^i_{Y_{k+1}}
\end{equation}
factors over a finitely presented $\overline{V}^\flat$-module for all $i \in [0,c-1]$, $j \in [0,n-1]$. For the indices treated in \eqref{SS1}, this is clear from our hypothesis on the maps $Y_{k+1} \to Y_k$. For the remaining indices (i.e., $j=n-1$ and $i \in [0,c-2]$, described  in \eqref{SS2}), we cannot use the same argument as the property of ``factoring over a finitely presented module'' for a map of $\overline{V}^\flat$-modules is not inherited submodules preserved under the map (such as the right side of \eqref{SS2})  since $\overline{V}^\flat$ is not noetherian. To circumvent this, we descend to a noetherian situation using a Breuil-Kisin prism, and then argue by base change; what follows is the only part of this paper that uses the theory from \cite{BhattScholzePrism} over an imperfect prism (which is not covered by \cite{BMS1}).

After possibly enlarging $V$, we can assume that $Y_{k+1} \to Y_k \to X$ is the base change of a map $Z_{k+1} \to Z_k \to X_V := \mathbf{P}^n_V$. Moreover, as $X_{p=0} \to X_{V,p=0}$ is a universal homeomorphism, we can also assume $\mathfrak{m} \subset \mathcal{O}_X$ is the radical of a maximal ideal $\mathfrak{n} \subset \mathcal{O}_{X_V}$.  Choose a Breuil-Kisin prism $(A,I)$ with $A/I = \widehat{V}$ as well as a map $(A,I) \to (A_{\inf}, (d))$ lifting the map $\widehat{V} \to \widehat{\overline{V}}$ (see \cite[Example 1.3]{BhattScholzePrism}); explicitly, if $W = W( (V/p)_{red})$, then $A = W\llbracket u \rrbracket$ with $\phi$ determined by $\phi(u) = u^p$, the map $A \to \widehat{V}$ given by $u \mapsto \pi$ for a uniformizer $\pi \in V$, and the map $A \to A_{\inf}$ given by $u \mapsto [\underline{\pi}]$ for a compatible system $\underline{\pi} \in \overline{V}^\flat$ of $p$-power roots of $\pi$.  Recall that $A/(p,I)$ is a field, the map $A \to A_{\inf}$ is faithfully flat (by \cite[Lemma 4.30]{BMS1} or Lemma~\ref{DropCompFlat} below), and $IA_{\inf} = (d)$.  The $(p,I)$-completed base change of the non-logarithmic derived prismatic complexes $\Prism^n_{Z_k/A}$ along $A \to A_{\inf}$ gives $\Prism^n_{Y_k}$ for all $k$, and this identification is compatible with all naturally defined maps (such as ones coming from reduction modulo $d$, or passing to local cohomology, or pulling back along $Z_{k+1} \to Z_k$). In particular, we learn that the map 
\[ H^i_{\mathfrak{n}}(\Prism^n_{Z_k/A}/(p,I)) \to H^i_{\mathfrak{n}}(\Prism^n_{Z_{k+1}/A}/(p,I))\]
factors over a finitely presented $A/(p,I)$-module for $i < n$: the ring $A/(p,I)$ is a field, so the property of factoring over a finitely presented $A/(p,I)$-module is equivalent to the image  being finite dimensional, and the latter can be checked after the faithfully flat exension $A/(p,I) \to A_{\inf}/(p,d) \to \overline{k}$, noting that after base change to $A_{\inf}/(p,d)$, the map factors over a finitely presented $A_{\inf}/(p,d)$-module and thus has finitely generated image. Since $A/(p,I)$ is a field, the property of having finite dimensional image passes to functorially defined submodules; in particular, the map
\[ \mathrm{ker} \left( H^{n-1}_{\mathfrak{n}}(\Prism^n_{Z_k/A}/(p,I)) \xrightarrow{\delta} H^n_{\mathfrak{n}}(\Prism^n_{Z_k/A}/(p,I^{c-i+1})) \right) \to  \mathrm{ker} \left( H^{n-1}_{\mathfrak{n}}(\Prism^n_{Z_{k+1}/A}/(p,I)) \xrightarrow{\delta} H^n_{\mathfrak{n}}(\Prism^n_{Z_{k+1}/A}/(p,I^{c-i+1})) \right) \]
also has finite dimensional image. By base change and \eqref{SS2}, this yields that the map in \eqref{SS3} factors over a finitely presented $\overline{V}^\flat$-module for $j=n-1$ and $i \in [0,c-2]$ as well, finishing the proof. 
\end{proof}

\begin{remark}
\label{NoetherianPropertyAlmostZero}
By Theorem~\ref{ACMaic}, the $\overline{V}^\flat$-modules $M := H^i_{\mathfrak{m}}(\mathcal{O}_{X^+}^\flat)$ for $i < n+1$ and $N := H^j_{\mathfrak{m}}(\mathcal{O}_{X^+}^\flat/d^c)$ for $j < n$ and any $c \geq 1$ are almost zero, and can thus be regarded as modules over $\overline{k} = \overline{V}^\flat/d^{1/p^\infty}$, which is a field. In particular, these $\overline{V}^\flat$-modules are locally noetherian, i.e., $\overline{V}^\flat$-submodules of finitely generated $\overline{V}^\flat$-submodules of $M$ or $N$ are automatically finitely generated. 
\end{remark}

We can now prove the main theorem. 

\begin{theorem}
\label{AICCMMain}
The equivalent assertions in Lemma~\ref{CMProjective} hold true for all $n \geq 1$. 
\end{theorem}
\begin{proof}
We shall prove by induction on $n$ that $(P_k)$ holds true for all $k \leq n$ (or equivalently that $(M_k)$ holds true for all $k \leq n$). When $n=1$, the statement holds true by Auslander-Buchsbaum. Assume $n \geq 2$ and that $(P_k)$ (and hence $(M_k)$) hold true for $k < n$. We shall verify $(P_n)$. Fix a maximal ideal $\mathfrak{m} \subset \mathcal{O}_{\mathbf{P}^n_V}$ corresponding to a closed point of the special fibre. Let us first make some reductions to translate $(P_n)$ to a statement about prismatic cohomology. 

\begin{enumerate}
\item {\em Reduction to a statement over $\overline{V}$.} We claim it suffices to show the following: for any $Y \in \mathcal{P}_X^{ss}$, there exists a further map $Y' \to Y$ in $\mathcal{P}_X$ such that
\[ H^i_{\mathfrak{m}}(\mathcal{O}_Y/p) \to H^i_{\mathfrak{m}}(\mathcal{O}_{Y'}/p)\]
is the $0$ map for $i < n$. Indeed, this implies a similar statement where $Y,Y' \in \mathcal{P}_X^{fin}$ by Theorem~\ref{IndObjectPairs} (1). Now any map $Y' \to Y$ in $\mathcal{P}_X^{fin}$ is the base change of a map between finite covers of $\mathbf{P}^n$ defined over a finite extension $V'$ of $V$. As local cohomology commutes with flat extension of scalars (such as $V' \to \overline{V}$), we obtain the assertion in $(P_n)$ by descent. 

\item {\em Reduction to prismatic cohomology modulo $(p,d)$.} We claim it is enough to show the following: for any $Y \in \mathcal{P}_X^{ss}$, there exists a further map $Y' \to Y$ in $\mathcal{P}_X^{ss}$ such that
\[ H^i_{\mathfrak{m}}(\Prism_Y/(p,d)) \to H^i_{\mathfrak{m}}(\Prism_{Y'}/(p,d))\]
is $0$ for $i < n$. This follows from Theorem~\ref{IndObjectPairs} (1).

\item {\em Reduction to prismatic cohomology modulo $p$.} We claim it is enough to show the following: for any $Y \in \mathcal{P}_X^{ss}$, there exists a further map $Y' \to Y$ in $\mathcal{P}_X^{ss}$ such that
\[ H^i_{\mathfrak{m}}(\Prism_Y/p) \to H^i_{\mathfrak{m}}(\Prism_{Y'}/p)\]
is $0$ for $i < n+1$. Indeed, iterating such a construction twice yields the statement in (2) thanks to the Bockstein sequence with respect to $d \in H^0(\Prism_Y/p)$.
\end{enumerate}

We now begin with the proof of (3). First, we show that sufficiently high transition maps in $\mathcal{P}_X^{ss}$ induce maps on local prismatic cohomology whose image contained inside a submodule with bounded $d$-torsion:
\begin{claim}
\label{CMclaim1}
For any $Y \in \mathcal{P}_X^{ss}$ and $i < n+1$,  there exists a map $Y' \to Y$ in $\mathcal{P}_X^{ss}$ such that the induced map 
\[ H^i_{\mathfrak{m}}(\Prism_Y/p) \to H^i_{\mathfrak{m}}(\Prism_{Y'}/p)\]
has image annihilated by $d^c$ for some $c \geq 1$ depending only on $X$.
\end{claim}
\begin{proof}
This follows from the factorization in Proposition~\ref{CMFactor}. 
\end{proof}

Next, using our induction assumption and a duality argument, we show that maps on local prismatic cohomology induced from sufficiently high transition maps in $\mathcal{P}_X^{ss}$ carry submodules with bounded $d$-torsion into finitely generated $\overline{V}^\flat$-modules:

\begin{claim}
\label{CMclaim2}
For any $c \geq 1$ and any $Y' \in \mathcal{P}_X^{ss}$, there exists a map $Y'' \to Y'$ in $\mathcal{P}_X^{ss}$ such that the   the image of 
\[ H^i_{\mathfrak{m}}(\Prism_{Y'}/p)[d^{c}] \to H^i_{\mathfrak{m}}( \Prism_{Y''}/p)[d^{c}] \]
 is contained in a finitely generated $V^\flat$-submodule of the target for $i < n+1$.
\end{claim}
\begin{proof}
  By the Bockstein sequence, it is enough to ensure that the image of 
 \[ H^i_{\mathfrak{m}}(\Prism_{Y'}/(p,d^{c})) \to H^i_{\mathfrak{m}}(\Prism_{Y''}/(p,d^{c}))\]
is contained in a finitely generated $\overline{V}^\flat$-submodule of the target for $i < n$. In fact, we shall we prove the stronger statement the this map itself factors over a finitely presented $\overline{V}^\flat$-module.  This follows from  Lemma~\ref{IncreaseFinite} once we show the following, verifying the hypothesis Lemma~\ref{IncreaseFinite}:
\begin{itemize}
\item[$(\ast)$]
For all $Z \in \mathcal{P}_X^{ss}$, there exists a map $Z' \to Z$ in $\mathcal{P}_X^{ss}$ such that 
\[ H^i_{\mathfrak{m}}(\mathcal{O}_{Z}/p) \to H^i_{\mathfrak{m}}(\mathcal{O}_{Z'}/p)\]
factors over a finitely presented $\overline{V}^\flat$-module for $i < n$. 
\end{itemize}
By approximation arguments similar to reduction (1) above, it is enough to prove such a statement for finite covers of $\mathbf{P}^n_V$ instead. The claim then follows from the assumption that $(M_k)$ holds true for $k < n$ as well as Lemma~\ref{IndCMPunctured} applied to $\mathbf{P}^n_{V/p}$. Alternately, we may use Lemma~\ref{IndCMPuncturedNN} directly.
\end{proof}

Combining Claims~\ref{CMclaim1} and \ref{CMclaim2} with Remark~\ref{NoetherianPropertyAlmostZero} shows that for any $Y \in \mathcal{P}_X^{ss}$, the image of 
\[ H^i_{\mathfrak{m}}(\Prism_Y/p) \to H^i_{\mathfrak{m}}(\mathcal{O}_{X^+}^\flat),\]
is a finitely generated $\overline{V}^\flat$-submodule for $i < n+1$. Note that this map is Frobenius equivariant, and hence its image is Frobenius stable. Applying Lemma~\ref{EquationCM} to the $\mathcal{O}_{X^+}^\flat$-linearization of the image shows that the image, and hence the map, is $0$. Taking the colimit over all $Y$'s, it follows that $H^i_{\mathfrak{m}}(\mathcal{O}_{X^+}^\flat)$ is $0$ for $i <n+1$. By Bockstein sequences for $d$ and $p$, this implies that $H^i_{\mathfrak{m}}(\mathcal{O}_{X^+}) = H^{i-1}_{\mathfrak{m}}(\mathcal{O}_{X^+}/p) = 0$ for $i < n+1$. 
Combining this with $(\ast)$, it follows that for any $Y \in \mathcal{P}_X^{fin}$, there is some map $Y' \to Y$ in $\mathcal{P}_X^{fin}$ such that the image of 
\[ H^i_{\mathfrak{m}}(\mathcal{O}_{Y}/p) \to H^i_{\mathfrak{m}}(\mathcal{O}_{Y'}/p)\]
is $0$ for $i < n$: the image is finitely generated for some finite cover and is $0$ when we take the colimit over all finite covers, so it must already be $0$ on a sufficiently large finite cover. By an approximation argument as in reduction (1) above, this proves that $(P_n)$ holds true. 
\end{proof}

\begin{remark}[Reformulation via ind-CM and cohomologically CM objects]
\label{rmk:CCMIndCMuCM}
Consider the ind-object $\{\mathcal{O}_Y/p\}_{\mathcal{P}_X^{fin}}$ in $\mathrm{Coh}(X_{p=0}) \subset D^b_{coh}(X_{p=0})$ with colimit $\mathcal{O}_{X^+}/p \in D_{qc}(X_{p=0})$. By Lemma~\ref{IndCMcCMNN}, the colimit $\mathcal{O}_{X^+}/p$ is cohomologically CM (in the sense of Definition~\ref{CMDerived}) if and only if the ind-object $\{\mathcal{O}_Y/p\}_{\mathcal{P}_X^{fin}}$ is ind-CM (in the sense of Definition~\ref{IndCMDef}). Theorem~\ref{AICCMMain} implies that these equivalent properties  hold true. In conjunction with Lemma~\ref{RegSeqCrit}, this implies that $\mathcal{O}_{X^+}/p$ is a Cohen--Macaulay quasi-coherent sheaf on $\mathbf{P}^n_{V/p}$ in the usual sense. 
\end{remark}

\begin{remark}[Dualizing complexes are pro-discrete in $\mathcal{P}_X$]
\label{ProCM}
Theorem~\ref{AICCMMain} and  Lemma~\ref{IndCMPuncturedNN} yield the following statement: with notations as above, the pro-systems $\{H^i(\omega^\bullet_{Y/\overline{V}}/p)\}_{Y \in \mathcal{P}_X}$ and $\{H^i(\omega^\bullet_{Y/\overline{V}} \otimes_{\overline{V}}^L k)\}_{Y \in \mathcal{P}_X}$ are $0$ for $i > - \dim(X_{p=0}) = \dim(X_{\overline{k}}) = n$. In fact, the first statement implies the second follows by applying the following observation to the map $\overline{V}/p \to \overline{k}$ with the module $M = \omega^\bullet_{Y/\overline{V}}/p[-n]$ and iterating twice: given a map $A \to B$ of commutative rings and $M \in D(A)$, we have $\tau^{> 0}(M \otimes_A^L B) \simeq \tau^{> 0} \left(  (\tau^{> 0} M) \otimes_A^L B \right)$ via the natural map. For the first statement, using Lemma~\ref{IndCMPuncturedNN}, it is enough to show that the ind-object $\{ \mathcal{O}_Y/p\}_{\mathcal{P}_X}$ in $D_{qc}(X_{p=0})$ is ind-CM, which is exactly what we showed in Theorem~\ref{AICCMMain}.
\end{remark}

The following result was used above:

\begin{lemma}[Killing local \'etale cohomology, aka the ``equational lemma'']
\label{EquationCM}
Let $\mathfrak{m} \subset \mathcal{O}_X$ be a maximal ideal. Fix an integer $i$. The $\mathcal{O}_{X^+}^\flat$-module $H^i_{\mathfrak{m}}(\mathcal{O}_{X^+}^\flat)$ contains no nonzero Frobenius stable finitely generated $\mathcal{O}_{X^+}^\flat$-submodule.
\end{lemma}

\begin{proof}
Let $\mathrm{Spec}(R) \subset X$ be a standard affine open containing the closed point defined by $\mathfrak{m}$. Its preimage in $X^+$ has the form $\mathrm{Spec}(R^+)$ for an absolute integral closure $R^+$ of $R$. The local cohomology $H^i_{\mathfrak{m}}(\mathcal{O}_{X^+}^\flat)$ identifies with $H^i_{\mathfrak{m}}(R^{+,\flat})$ (using, e.g., the formula in Lemma~\ref{LocalCohFormal} to avoid  completion problems). Assume there exists a  nonzero Frobenius stable finitely generated $R^{+,\flat}$-submodule $M \subset H^i_{\mathfrak{m}}(R^{+,\flat})$. By perfectness of the ambient module and finite generation, any such $M$ must be scheme-theoretically supported on $(R^{+,\flat}/(d,\mathfrak{m}))_{red} = R^+/\sqrt{\mathfrak{m}R^+}$ and thus corresponds to a holonomic Frobenius module on $R^+/\sqrt{\mathfrak{m}R^+}$ in the sense of \cite{BhattLurieModpRH}. But $\mathrm{Spec}(R^+/\sqrt{\mathfrak{m}R^+})$ is a profinite set (as a topological space) with each stalk being an algebraically closed residue field, so its \'etale topology identifies with the Zariski topology, and hence every nonzero holonomic Frobenius module over $R^+/\sqrt{\mathfrak{m}R^+}$ has nonzero Frobenius fixed points\footnote{A nonzero sheaf $F$ of $\mathbf{F}_p$-modules on a profinite set $S$ must have $H^0(S,F) \neq 0$. Indeed, by left exactness of $H^0(S,-)$, it is enough to show the same for $F$ constructible. Since constructible sets (and hence quasi-compact opens) in $S$ are clopen, each constructible sheaf $F$ on $S$ is a direct sum of constant sheaves supported on clopen subsets, so the claim is clear.}. The Artin-Schreier sequence then shows that $H^i_{\mathfrak{m}}(\mathrm{Spec}(R^{+,\flat}), \mathbf{F}_p) \neq 0$. Thanks to  tilting for local \'etale cohomology \cite[Theorem 2.2.7]{CesnavicusScholzePurity}, this group is $H^i_{\mathfrak{m}}(\mathrm{Spec}(R^{+,h}), \mathbf{F}_p)$, where $R^{+,h}$ denotes the $p$-henselization of $R^+$. By Nisnevich excision, this group is also simply  $H^i_{\mathfrak{m}}(\mathrm{Spec}(R^+), \mathbf{F}_p)$: the map $R^+ \to R^{+,h}$ is an ind-\'etale neighbourhood of $V(\mathfrak{m})$ since $p \in \mathfrak{m}$. Since $\dim(R^+) \geq 1$,  Proposition~\ref{EtaleCohAIC} shows that this group vanishes for all $i$, which is a contradiction.
\end{proof}

\begin{remark}
The term ``equational lemma'' used to describe Lemma~\ref{EquationCM} goes back to (the method of proof of) the similar statement in \cite[Theorem 2.2]{HHCM}. The strategy of proving this statement by reduction to relatively simple facts in \'etale cohomology was already used in \cite{BhattAnnihilatingGroupSch} in the characteristic $p$ case.
\end{remark}

\newpage \section{Extending to excellent rings}
\label{sec:ExcCM}

In this section, we extend the geometric results  \S \ref{CMgeomcase} to arbitrary excellent noetherian local domains $R$ by relatively standard approximation arguments and record some consequences. More precisely, we first prove the Cohen--Macaulayness of $R^+$ when formulated in terms of local cohomology (Theorem~\ref{AICVanishLocalCoh}) and deduce that excellent splinters are Cohen--Macaulay (Corollary~\ref{splintersCM}). An inductive characterization of Cohen--Macaulayness that we already encountered earlier then allows us to prove that $R^+/p$ is Cohen--Macaulay over $R/p$ in the sense of regular sequences (Corollary~\ref{CMRegSeqAIC}), proving the result promised in Theorem~\ref{RPlusCMIntro}. Finally, these results are reinterpreted as the $p$-complete flatness of $R^+$ over $R$ when $R$ is regular (Theorem~\ref{RPlusFlatReg}).

The local cohomological variant of our main result for arbitrary excellent noetherian local rings is the following theorem, whose proof uses Popescu's approximation theorem to reduce to the geometric case treated in \S \ref{CMgeomcase}.

\begin{theorem}[Vanishing of local cohomology of absolute integral closures]
\label{AICVanishLocalCoh}
Let $(R,\mathfrak{m},k)$ be an excellent noetherian local domain with $p \in \mathfrak{m}$ and let $R^+$ be an absolute integral closure of $R$. 
\begin{enumerate}
\item $H^i_{\mathfrak{m}}(R^+/pR) = 0$ for $i < \dim(R/pR)$ and $H^i_{\mathfrak{m}}(R^+) = 0$ for $i < \dim(R)$.
\item Any system of parameters on $R$ gives a Koszul regular sequence on $R^+$.
\item The derived $\mathfrak{m}$-adic completion $B = \widehat{R^+}$ is concentrated in degree $0$, coincides with the classical $\mathfrak{m}$-adic completion, and is Cohen--Macaulay over $R$. 
\item If $R$ admits a dualizing complex, then there exists a finite extension $R \to S$ with $H^i_{\mathfrak{m}}(R/pR) \to H^i_{\mathfrak{m}}(S/pR)$ being the $0$ map for all $i < \dim(R/pR)$.
\end{enumerate}
\end{theorem}

\begin{proof}
We only give the arguments assuming $R$ is $p$-torsionfree, though they all extend to the case where $p=0$ in $R$ (where the statements are already known) with purely notational changes.

(1): The assertion for $R^+/p$ implies that for $R^+$ in view of Bockstein sequences as well as the fact that $H^i_{\mathfrak{m}}(R^+)$ is $p^\infty$-torsion, so it suffices to prove the vanishing for $R^+/p$. We prove this in a series of steps using Popescu's approximation theorem. By excellence, we may assume $R$ is normal, say of dimension $n$.

Let us first show the claim when $R$ is $\mathfrak{m}$-complete. By the Cohen structure theorem, we can choose a finite injective map $P = W \llbracket x_2,...,x_n \rrbracket \to R$ with $W$ a Cohen ring for $k$. As $P \to R$ is a finite injective map of complete noetherian local domains, we have $P^+ = R^+$ and $\sqrt{\mathfrak{m}_P R} = \mathfrak{m}$ (where $\mathfrak{m}_P$ is the maximal ideal of $P$), so we may assume $P = R$. Thanks to Popescu's theorem \cite[Tag 07GC]{StacksProject}, we can write $P = \colim_i P_i$ as a filtered colimit of smooth algebras over $P_0 = W[x_2,...,x_n]$. Let $Q_i$ be the localization of $P_i$ at the image of $\mathfrak{m} \in \mathrm{Spec}(R)$, so we also have $P = \colim_i Q_i$ with each $Q_i$ being the local ring of a smooth $W$-scheme at a closed point of residue characteristic $p$. In particular, each $Q_i$ is a regular local ring. Let $\mathfrak{n} = (p,x_2,..,x_n) \subset Q_0$, so $\mathfrak{n}P = \mathfrak{m}$.  By the compatibility of local cohomology with filtered colimits and a finite presentation argument to descend a finite extension of $P$ to that of $Q_j$ for $j \gg 0$, it is enough to show that $H^i_{\mathfrak{n}}(Q_j^+/p) = 0$ for $i < n-1 = \dim(P/p)$. Since $Q_j$ is a flat over $Q_0$ (as $P_j$ was flat over $P_0$), we have $H^i_{\mathfrak{n}}(Q_j/p) = 0$ for $i < n-1$ and all $j$ as the same holds true for $j=0$ by direct calculation. Since $Q_j$ is regular and essentially finitely presented over $W$, Theorem~\ref{KillCohFinite} and Lemma~\ref{FlatCritComp} imply that $Q_j \to Q_j^+$ is $p$-completely flat, and thus $H^i_{\mathfrak{n}}(Q_j^+/p) = 0$ for $i < n-1$, as wanted.

Next, we prove (1) when $R$ is $\mathfrak{m}$-henselian. By Popescu's theorem again,  the $\mathfrak{m}$-completion $\widehat{R}$ of $R$ is a normal domain and the map $R \to \widehat{R}$ to the $\mathfrak{m}$-completion of $R$ can be written as a filtered colimit of smooth maps $R \to R_i$. Given a class $\alpha \in H^i_{\mathfrak{m}}(R/p)$ for $i < \dim(R/p) = \dim(\widehat{R}/p)$,  the complete case treated in the previous paragraph gives a finite cover $\widehat{R} \to T$ such that $\alpha$ maps to $0$ in $H^i_{\mathfrak{m}}(T/p)$. We can descend $T$ to a finite cover $R_j \to T_j$ for some $j \geq 0$ by \cite[Tags 01ZL, 01ZO, 07RR]{StacksProject}. By enlarging $j$, we can also assume that $\alpha$ maps to $0$ in $H^i_{\mathfrak{m}}(T_j/p)$. But the map $R \to R_j$ admits a section: this is a smooth map over a henselian base with a distinguished section over the residue field (or in fact any artinian quotient of $R$). Base changing $R_j \to T_j$ along a section $R_j \to R$ then gives a finite cover $R \to T$ such that $\alpha$ maps to $0$ in $H^i_{\mathfrak{m}}(T/p)$, as wanted.

Finally, we prove $(1)$ for an arbitrary excellent noetherian normal local domain $(R,\mathfrak{m})$. Let $R \to R^h$ be the henselization of $R$.  By Lemma~\ref{AICConnComp} (1) - (3), the tensor product $R^{+,h} := R^+ \otimes_R R^h$ is the henselization of $R^+$ at $\mathfrak{m}$ and each connected component of $R^{+,h}$ is an absolute integral closure of $R^h$. Using Lemma~\ref{VanishConnComp} and Lemma~\ref{AICConnComp} (4), it is now enough to observe that $H^i_{\mathfrak{m}}( (R^h)^+/p) = 0$ for $i < \dim(R/p)$, which follows from the henselian case treated above.

(2): This is a formal consequence of (1). Indeed, given a system of parameters, say $\underline{x} := x_1,...,x_d$, on $R$,  we must show that $\mathrm{Kos}(R^+;\underline{x})$ is coconnective. But this complex is quasi-isomorphic to $\mathrm{Kos}(R\Gamma_{\mathfrak{m}}(R^+); \underline{x})$, which is coconnective as $R\Gamma_{\mathfrak{m}}(R^+)$ is concentrated in degree $d = \dim(R)$ by (1).

(3): This follows from (2) and Lemma~\ref{KoszulRegularCrit} applied to $M = R^+$. 

(4): Combine (1) applied to all localizations of $R$ at points of characteristic $p$ with Corollary~\ref{IndCMcCM} applied to the ind-object $\{S/p\}$ of $D(R/p)$ indexed by all finite $R$-subalgebras $S \subset R^+$. 
\end{proof}

\begin{example}[Excellence cannot be dropped entirely]
\label{ExcNecc}
Let $k$ be a field. Let $(R,\mathfrak{m})$ be a noetherian local UFD of dimension $3$ with $\mathfrak{m}$-adic completion $\widehat{R} = k\llbracket x,y,w,z \rrbracket / (wx,wy)$; such a ring exists by \cite{Heitmann-completion-UFD}. We claim that the conclusion of Theorem~\ref{AICVanishLocalCoh} (1) fails for $R$. Indeed, assume towards contradiction that  each class in $H^2_{\mathfrak{m}}(R)$ can be annihilated by passage to a finite cover of $R$. Then the same holds true for $H^2_{\mathfrak{m}}(\widehat{R})$. Now $\widehat{R}$ has two minimal primes $\mathfrak{p} = (w)$ and $\mathfrak{q} = (x,y)$, with $A = R/\mathfrak{p} = k \llbracket x,y,z \rrbracket$ and $B = R/\mathfrak{q} = k \llbracket w,z \rrbracket$. Moreover, there is a Mayer-Vietoris exact sequence
\[ 0 \to \widehat{R} \to A \times B \to C := \widehat{R}/(x,y,w) = k\llbracket z \rrbracket \to 0\]
of finitely generated $\widehat{R}$-modules. Taking the long exact sequence of local cohomology shows that $H^2_{\mathfrak{m}}(\widehat{R}) \to H^2_{\mathfrak{m}}(B)$ is surjective. Thus, our assumption implies that each class in $H^2_{\mathfrak{m}}(B)$ can be annihilated by a finite cover of $B$. But $B$ is a splinter (it is complete and regular), so $H^2_{\mathfrak{m}}(B) \to H^2_{\mathfrak{m}}(B')$ is split injective for any finite cover $B \to B'$. Since $H^2_{\mathfrak{m}}(B) \neq 0$ as $\dim(B) = 2$, we get a contradiction.
\end{example}

\begin{remark}[Relaxing excellence]
\label{MellowDualizingDrop}
The analog of Theorem~\ref{AICVanishLocalCoh} (4) in characteristic $p$, under slightly different assumptions, was proven in \cite{HunekeLyubeznik}, and gives a refinement of \cite{HHCM}. The main result of \cite{QuyRPlus}  extends this to all noetherian local domains over $\mathbf{F}_p$ that are quotients of CM rings. Thus, one might wonder if the entirety of Theorem~\ref{AICVanishLocalCoh} holds true for any noetherian local domain that is the quotient of a CM ring.
\end{remark}

\begin{remark}[Weakly functorial CM algebras]
Theorem~\ref{AICVanishLocalCoh} part (3) gives a new and explicit construction of ``big Cohen--Macaulay algebras'' that are weakly functorial in local maps of local domains of mixed characteristic. Such algebras were previously constructed by \cite{AndreWeakFun}  (via Hochster's partial algebra modification process as refined by \cite{DietzSeeds} and extended to mixed characteristic in \cite{ShimomotoIntegerPerfAndre}; see also \cite[Appendix A]{MSTWWAdjoint}) and Gabber \cite{GabberMSRINotesBCM} (via an ultraproduct construction); see also Remark~\ref{rmk:WeakFunNonLoc} below.
\end{remark}

\begin{remark}[Top degree local cohomology of $R^+$]
For an excellent noetherian local domain $R$, Theorem~\ref{AICVanishLocalCoh} yields a  mixed characteristic analog of \cite{HHCM}, which concerns the local cohomology of $R^+$ outside the top degree. It would be quite interesting to extend these ideas to the top degree and prove a mixed characteristic analog of \cite{SmithTightParam}.
\end{remark}

\begin{remark}
Theorem~\ref{AICVanishLocalCoh} implies that several results in commutative algebra that were contingent on \cite{HeitmannDSC3} and thus had a dimension $\leq 3$ constraint now hold in arbitrary dimension; one such example is \cite[Theorem 4.13]{BIMKunz}. We thank Wenliang Zhang for suggesting this remark.
\end{remark}

The following regularity criterion (which is borrowed from \cite[Theorem 8.5.1]{BrunsHerzog}) was used in Theorem~\ref{AICVanishLocalCoh}.

\begin{lemma}
\label{KoszulRegularCrit}
Let $R$ be a commutative ring equipped with a sequence $\underline{x} := x_1,...,x_d$ of elements. Let $M \in D^{\leq 0}(R)$ such that $\mathrm{Kos}(M; \underline{x})$ is concentrated in degree $0$. Then the derived $\underline{x}$-completion $N$ of $M$ is concentrated in degree $0$ and $\underline{x}$ gives a regular sequence on $N$. Moreover, $N$ coincides with the classical $\underline{x}$-completion of $H^0(M)$. 
\end{lemma}

\begin{proof}
We shall prove that $N$ is concentrated in degree $0$ and that $x_1$ is a nonzerodivisor on $N$; one can then deduce the first assertion by induction on $d$ applied to $M' = M/x_1 \in D^{\leq 0}(R)$ with the sequence $\underline{x}' = x_2,...,x_d$, noting that $N' = N/x_1N$ is the derived $\underline{x}'$-completion of $M'$. 

By assumption, we know that $\mathrm{Kos}(M; \underline{x}^n) \simeq H^0(M)/\underline{x}^nH^0(M)$. This implies that $N \simeq R\lim_n \mathrm{Kos}(M;\underline{x}^n) \simeq \lim_n H^0(M)/\underline{x}^nH^0(M)$, whence  $N$ is discrete and is the classical $\underline{x}$-adic completion of $H^0(M)$. In particular, $N$ is $\underline{x}$-adically separated. As Koszul regular sequences are quasi-regular, we also know that the associated graded $\mathrm{gr}_{\underline{x}}(N)$ of the $\underline{x}$-adic filtration on $N$ is a polynomial algebra $N/\underline{x}N[X_1,...,X_d]$. This implies that if $f \in N$ satisfies $x_1 f \in \underline{x}^n N$, then $f \in \underline{x}^{n-1} N$. In particular, since $N$ is $\underline{x}$-adically separated, we learn that $x_1$ is a nonzerodivisor on $N$, as wanted.

The fact that $N$ coincides with the classical $\underline{x}$-completion of $H^0(M)$ was already proven above. 
\end{proof}

The following lemmas on the structure on the absolute integral closure and its henselization were also used Theorem~\ref{AICVanishLocalCoh}.

\begin{lemma}
\label{AICConnComp}

Let $R$ be a normal domain with a fixed absolute integral closure $R \to R^+$. Let $I \subset R$ be an ideal, and write $R^h$ for the henselization of $R$ along $I$. Consider $R^{+,h} := R^+ \otimes_R R^h$.

\begin{enumerate}
\item The map $R^+ \to R^{+,h}$ is the henselization of $R^+$ along $IR^+$.
\item The map $\mathrm{Spec}(R^+/I) \to \mathrm{Spec}(R^{+,h})$ induces a bijection on connected components. 
\item Each connected component $Z \subset \mathrm{Spec}(R^{+,h})$ is the spectrum of an absolute integral closure of $R^h$ via the natural map $R^h \to R^{+,h} \to \mathcal{O}(Z)$. 
\item For any $M \in D(R)$ and any integer $i$, the map $H^i_I(M) \to H^i_I(M \otimes_R R^h)$ is bijective.
\end{enumerate}
\end{lemma}
\begin{proof}
(1): The formation of henselizations commutes with integral base change, see \cite[Tag 0DYE]{StacksProject}.

(2): This follows from (1) by henselianness since $R^{+,h}/I = R^+/I$.

(3): The map $R^+ \to R^{+,h}$ is a filtered colimit of \'etale $R^+$-algebras. As each \'etale $R^+$-algebra is a finite product of absolutely integrally closed domains, it follows that each connected component $\mathcal{O}(Z)$ of $R^{+,h}$ is also an absolutely integrally closed domain. Each of these connected components is also integral over $R^h$ (as $R^h \to R^{+,h}$ is integral and $R^{+,h} \to \mathcal{O}(Z)$ is surjective), so each component must be an absolute integral closure of $R^h$.

(4): By general nonsense, we have $R\Gamma_I(M) \otimes_R^L R^h \simeq R\Gamma_I(M \otimes_R R^h)$. As $R^h$ is $R$-flat, this gives $H^i_I(M) \otimes_R R^h \simeq H^i_I(M \otimes_R R^h)$. It remains to observe that $N \simeq N \otimes_R R^h$ for any $I$-power torsion complex as $R \to R^h$ is an isomorphism after derived $I$-completion.
\end{proof}

\begin{lemma}
\label{VanishConnComp}
Let $R$ be a commutative ring. Let $F$ be a set-valued functor on $R$ algebras that commutes with filtered colimits and  finite products. If $F(S) = 0$ for each connected component $\mathrm{Spec}(S) \subset \mathrm{Spec}(R)$, then $F(R) = 0$.
\end{lemma}
\begin{proof}
Pick an element $x \in F(R)$. As each connected component has the form $\colim_i R_i$ where $R_i$ is a direct factor of $R$, we learn from the compatibility of $F$ with filtered colimits there exists a clopen cover $\{U_i\}$ of $\mathrm{Spec}(R)$ such that $x = 0$ in $F(\mathcal{O}(U_i))$ for all $i$. By quasi-compactness, we may assume this cover is finite. As a finite clopen cover may be refined by a finite clopen cover by sets that are pairwise disjoint,   we may assume that the $U_i$'s are pairwise disjoint. But then $R = \prod_i \mathcal{O}(U_i)$, whence $x = 0$ in $F(R)$ by the finite product compatibility. 
\end{proof}

As a first application, we deduce that excellent splinters are Cohen--Macaulay.

\begin{corollary}[Splinters are Cohen--Macaulay]
\label{splintersCM}
Let $R$ be an excellent domain with $p \in \mathrm{Rad}(R)$. If $R$ is a splinter, then $R$ is Cohen--Macaulay.
\end{corollary}
\begin{proof}
As $p \in \mathrm{Rad}(R)$, it is enough to show that $R$ is Cohen--Macaulay at points of characteristic $p$, i.e., that $H^i_{\mathfrak{p}}(R/pR) = 0$ for $i < \dim(R_{\mathfrak{p}}/pR_{\mathfrak{p}})$ for every prime $\mathfrak{p}$ of $R$ containing $p$; here we implicitly use that the Cohen--Macaulay locus is open by excellence, see \cite[\S 2.8]{CesnavicusCM}. If $R$ is a splinter, the same holds true for any localization of $R$ (as finite covers spread out along open immersions). Thus, we may replace $R$ with its localization at a prime to assume $(R,\mathfrak{m})$ is an excellent noetherian local domain with $p \in \mathfrak{m}$. Choose an absolute integral closure $R \to R^+$.  The map $R \to R^+$ is a direct limit of split maps by the splinter property, so the same holds true for the map $H^i_{\mathfrak{m}}(R/pR) \to H^i_{\mathfrak{m}}(R^+/pR^+)$ as well. In particular, this map is injective for all $i$. The target vanishes for $i < \dim(R/pR)$ by Theorem~\ref{AICVanishLocalCoh} and hence so does the source. In particular, $R$ is Cohen--Macaulay. 
\end{proof}

Next, we strengthen Theorem~\ref{AICVanishLocalCoh} to a statement that involves the entirety of $\mathrm{Spec}(R^+/p)$ instead of its formal completion at the preimage of a closed point.

\begin{corollary}[Cohen--Macaulayness of absolute integral closures]
\label{CMRegSeqAIC}
Let $R$ be an excellent noetherian domain with an absolute integral closure $R^+$. Then the $R/p^nR$-module $R^+/p^nR^+$ is a Cohen--Macaulay for all $n \geq 1$.
\end{corollary}

\begin{proof}
We may assume $n=1$. As $\mathrm{Spec}(R^+/pR) \to \mathrm{Spec}(R/pR)$ is surjective, it is clear that $R^+_{\mathfrak{p}}/\mathfrak{p}R^+_{\mathfrak{p}} \neq 0$ for any prime $\mathfrak{p} \subset R$ containing $p$. For the regular sequence property, note that the formation of absolute integral closures commutes with localization. Theorem~\ref{AICVanishLocalCoh} thus implies that for each $\mathfrak{p} \in \mathrm{Spec}(R/pR) \subset \mathrm{Spec}(R)$, the complex $R\Gamma_{\mathfrak{p}}(R^+_{\mathfrak{p}}/pR_{\mathfrak{p}}^+)$ is concentrated in degree $\dim(R_{\mathfrak{p}}/pR_{\mathfrak{p}})$.  As excellent rings are catenary, the claim now follows from  Lemma~\ref{RegSeqCrit} with Theorem~\ref{AICVanishLocalCoh}.
 \end{proof}

The following consequence was conjectured in \cite{LyubeznikRPlus}.

\begin{corollary}[Cohen--Macaulayness of $\overline{R^+}$]
\label{LyubeznikConj}
Let $(R,\mathfrak{m})$ be a excellent noetherian local domain with an absolute integral closure $R^+$. Then the $R/pR$-module $(R^+/pR^+)_{red}$ is Cohen--Macaulay.
\end{corollary}
\begin{proof}
It is easy to see that $\sqrt{pR^+} = p^{1/p^\infty} R^+$. Consequently, $(R^+/p)_{red} = \colim_n R^+/p^{1/p^n} R^+$. The claim now follows readily from Corollary~\ref{CMRegSeqAIC}.
\end{proof}

\begin{remark}
One can in fact deduce Corollary~\ref{LyubeznikConj} directly from the almost Cohen--Macaulayness result in Theorem~\ref{ACMaic} using an argument with \'etale cohomology and induction on dimension. This argument will appear in \cite{BhattLuriepadicRH1}; this implication was also independently noticed by Heitmann-Ma (unpublished) who used Frobenius modules in lieu of \'etale cohomology. In turn, Corollary~\ref{LyubeznikConj} and Theorem~\ref{ACMaic} already imply that $R^+/p^n$ has depth $\geq d-2$ where $d = \dim(R)$; however, improving this estimate to $d-1$ (as follows from Theorem~\ref{AICVanishLocalCoh}) seems quite difficult to do directly. The proof of Theorem~\ref{AICVanishLocalCoh} essentially jumps this ``off-by-1'' hurdle by introducing an extra arithmetic parameter by deforming from $\mathcal{O}/p$ to $\Prism/p$. A similar ``off-by-1'' problem is encountered in \cite{CesnavicusScholzePurity}, where it is was solved by deforming from $\mathcal{O}$ to $\Prism_{\perf}$.
\end{remark}

We briefly indicate how one might use $R^+$ instead of arbitrary big Cohen--Macaulay algebras to prove some results on singularities as in \cite{MaSchwedeSingMixed}. 

\begin{remark}[$F$-rationality through $R^+$]
Let $(R,\mathfrak{m})$ be a noetherian local domain of residue characteristic $p$ with dimension $d$ with a fixed absolute integral closure $R^+$. Let us define $R$ to be {\em $F$-rational} if the maps $H^i_{\mathfrak{m}}(R) \to H^i_{\mathfrak{m}}(R^+)$ are injective for all $i$; equivalently (by Theorem~\ref{AICVanishLocalCoh}), we require $R$ is Cohen--Macaulay and demand the preceding injectivity for $i=d$. Then one has the following properties:

\begin{enumerate}
\item $R$ has pseudo-rational singularities. In fact, one has the following stronger assertion: if $f:X \to \mathrm{Spec}(R)$ is any proper surjective map, then $R \to R\Gamma(X, \mathcal{O}_X)$ induces an injective map on $H^d_{\mathfrak{m}}(-)$. To see this, we may assume $X$ is integral. In fact, it suffices to show the same injectivity after replacing $X$ with an absolute integral closure $X^+$ compatible with the given one for $R$. The map $R \to R\Gamma(X, \mathcal{O}_X) \to R\Gamma(X^+, \mathcal{O}_{X^+})$ then factors as $R \to R^+ \to R\Gamma(X^+,\mathcal{O}_X^+)$, so the claim follows from the fact that $R^+ \to R\Gamma(X^+,\mathcal{O}_X^+)$ gives an isomorphism on $p$-completion (and hence also on applying $R\Gamma_{\mathfrak{m}}(-)$) by Theorem~\ref{KillCohMixed}.

\item When $R$ has characteristic $p$, this notion coincides with $F$-rationality as classically defined (see \cite[Proposition 3.5]{MaSchwedeSingMixed}).

\item Given $f \in R$, if $R/f$ is an $F$-rational domain, then so is $R$.  To see this, we may assume $f \neq 0$, so $\dim(R/f) = d-1$. As $R/f$ is Cohen--Macaulay, the same holds true for $R$. We must show that $H^d_{\mathfrak{m}}(R) \to H^d_{\mathfrak{m}}(R^+)$ is injective. It suffices to show injectivity on the $f$-torsion: any non-trivial submodule of $H^d_{\mathfrak{m}}(R)$ intersects non-trivially with the $f$-torsion. But the map on $f$-torsion identifies with the map $H^{d-1}_{\mathfrak{m}}(R/f) \to H^{d-1}_{\mathfrak{m}}(R^+/f)$ by the Bockstein sequence, so the claim follows as we can factor $R/f \to (R/f)^+$ over $R/f \to R^+/f$.
\end{enumerate}
\end{remark}

Next, we want to reformulate the main theorem as a flatness assertion over regular rings. 
The following lemma, essentially borrowed from \cite[Lemma 4.31]{BMS1}, allows us to upgraded $I$-complete flatness to genuine flatness in noetherian situations without finiteness hypotheses.

\begin{lemma}[$I$-complete flatness = flatness over a noetherian ring]
\label{DropCompFlat}
Let $R$ be a noetherian ring and let $I \subset R$ be an ideal. If $M \in D(R)$ is $I$-completely flat, then the derived $I$-completion $\widehat{M}$ is $R$-flat (i.e., a flat $R$-module in the usual sense).
\end{lemma}
\begin{proof}
By replacing $M$ with $\widehat{M}$, we may assume $M$ is derived $I$-complete. By flatness modulo $I$ and the formula $M \simeq R\lim_n(M \otimes_R^L R/I^n)$ coming from derived $I$-completeness, we already know that $\widehat{M}$ is discrete. It remains to show that $N \otimes_R^L M \in D^{\geq 0}$ for all discrete $R$-modules $N$. By compatibility with filtered colimits, we may assume $N$ is finitely presented. But then $N$ is a pseudocoherent $R$-complex as $R$ is noetherian. It is then easy to see (e.g., by tensoring with a resolution of $N$ by finite free $R$-modules) that $N \otimes_R^L M$ is derived $I$-complete, so $N \otimes_R^L M \simeq N \widehat{\otimes}_R^L M \simeq R\lim_n(N \otimes_R^L M \otimes_R^L R/I^n)$. By the Artin-Rees lemma, the pro-systems $\{N/I^nN\}$ and $\{N \otimes_R^L R/I^n\}$ are pro-isomorphic, so $N \widehat{\otimes}_R^L M \simeq \lim_n (N/I^nN\otimes_R^L M)$. As $M$ is $I$-completely flat, this last object is in $D^{\geq 0}$, which proves the lemma.
\end{proof}

Applying the previous lemma, we obtain:

\begin{theorem}[Flatness of absolute integral closures of regular rings]
\label{RPlusFlatReg}
Let $R$ be an excellent regular ring with $p \in \mathrm{Rad}(R)$. If $R^+$ is an absolute integral closure, then the $p$-completion $\widehat{R^+}$ is faithfully flat over $R$.
\end{theorem}
\begin{proof}
Theorem~\ref{AICVanishLocalCoh} and Lemma~\ref{FlatCritComp} imply that $\widehat{R^+}$ is $p$-completely $R$-flat, and Lemma~\ref{DropCompFlat} improves that to genuine flatness. For faithful flatness, it suffices to show that $\mathrm{Spec}(\widehat{R^+}) \to \mathrm{Spec}(R^+)$ is surjective. But the image is stable under generalization by flatness and contains $\mathrm{Spec}(R/pR)$ (as $\mathrm{Spec}(R^+/pR) \to \mathrm{Spec}(R/pR)$ is surjective), so it must be everything as $p \in \mathrm{Rad}(R)$.
\end{proof}

Unlike Corollary~\ref{CMRegSeqAIC}, Theorem~\ref{RPlusFlatReg} includes an assertion about characteristic $0$ points. We can use this to improve Corollary~\ref{CMRegSeqAIC} itself in a similar vein:

\begin{corollary}[Cohen--Macaulayness of $p$-complete absolute integral closures]
\label{CMPlusComp}
Let $S$ be a biequidimensional excellent noetherian domain with $p \in \mathrm{Rad}(S)$. Assume that $S$ admits a Noether normalization, i.e., there exists a finite injective map $R \to S$ of excellent biequidimensional domains with $R$ regular. If $S^+$ denotes an absolute integral closure of $S$, then the $p$-completion $\widehat{S^+}$ is a Cohen--Macaulay $S$-module.
\end{corollary}

\begin{proof}
Note that the map $\mathrm{Spec}(\widehat{S^+}) \to \mathrm{Spec}(S)$ is surjective as $p \in \mathrm{Rad}(S)$. Using Corollary~\ref{CatCohCM}, it is enough to prove that $R\Gamma_x( (\widehat{S^+})_x) \in D^{\geq \dim(S_x)}$ for all $x \in \mathrm{Spec}(S)$. For any $x \in \mathrm{Spec}(S)$ with image $y \in \mathrm{Spec}(R)$ and any $M \in D(S)$, we have $\dim(R_y) = \dim(S_x)$ (as $R \to S$ is a finite injective map of biequidimensional domains) and the object $R\Gamma_x(M_x)$ is a summand of $R\Gamma_y(M_y)$ as an object in $D(R_y)$ (as the fibres of $\mathrm{Spec}(S) \to \mathrm{Spec}(R)$ are finite discrete sets). Consequently, it is enough to prove $R\Gamma_y((\widehat{S^+})_y) \in D^{\geq \dim(R_y)}$ for all $y \in \mathrm{Spec}(R)$. But $S^+ = R^+$ as $R \to S$ is a finite injective map of domains, and thus $\widehat{S^+} = \widehat{R^+}$. The claim now follows from  Theorem~\ref{RPlusFlatReg} (and the observation that $p \in \mathrm{Rad}(R)$ since $p \in \mathrm{Rad}(S)$ and $\mathrm{Spec}(S) \to \mathrm{Spec}(R)$ is finite surjective).
\end{proof}

\begin{remark}[Weakly functorial CM algebras, redux]
\label{rmk:WeakFunNonLoc}
Corollary~\ref{CMPlusComp} applies to any complete noetherian local domain $S$ with $p \in \mathrm{Rad}(S)$, thus giving an explicit construction of weakly functorial CM algebras on the category of such local domains (with possibly non-local maps). Moreover, one might wonder if the conclusion of Corollary~\ref{CMPlusComp} holds true without assuming the existence of a Noether normalization\footnote{This question was answered after this paper first appeared: \cite[Theorem 2.4]{BMPSTWW} extended Corollary~\ref{CMPlusComp} to all excellent local domains by essentially elementary arguments. Thus, one now has a very simple construction of weakly functorial Cohen--Macaulay algebras over excellent local domains $R$ of residue characteristic $p$: we may simply use the $p$-adic completion $\widehat{R^+}$.}.
\end{remark}

\newpage \section{The global theorem}

In this section, we extend the results from \S \ref{CMgeomcase} to the graded case. Using the equivalent geometric language, we first formulate our main theorem in \S \ref{ss:MainThmGraded} as a variant of Kodaira vanishing ``up to finite covers'' for proper scheme $X$ over a $p$-adic DVR equipped with an ample line bundle $L$ (Theorem~\ref{KodairaFiniteCover}). In \S \ref{ss:GAIC}, we introduce the graded absolute integral closure $R^{+,GR}$ as a suitable section ring, and reformulate the main theorem in terms of local cohomology (Theorem~\ref{GAICCM}). In \S \ref{ss:GAICACM}, using the Riemann-Hilbert functor from Theorem~\ref{thm:RH}, we prove the target theorem in the almost category (Proposition~\ref{WeakCMGAIC}); this is perhaps the most novel material in the section, relying crucially on the following small miracles of geometry in an infinitely ramified context:
\begin{itemize}
\item An interpretation of $R^{+,GR}$ (which is an inverse limit of affine cones over projective schemes) as the derived global functions on $T_\infty$ (which is an inverse limit of the $\mathbf{G}_a$-bundles over projective schemes that are obtained by blowing up the vertex in the affine cones that approximate  $R^{+,GR}$) that is only available  at infinite level (see Proposition~\ref{AICAmple}).
\item A variant of ``semiperversity of nearby cycles'' from \cite[\S 4.4]{BBDG} that also becomes available only at infinite level (allowing us to replace $\mathbf{G}_a$ with $\mathbf{G}_m$ while studying the constant sheaf on the scheme $T_\infty$ mentioned above, see first half of proof of Lemma~\ref{IndPervConstGAIC}).
\end{itemize}
 We pass from almost mathematics to honest mathematics in \S \ref{ss:GACMtoCM} by broadly following the contours of the argument in  \S \ref{CMgeomcase}. In \S \ref{rmk:RelativeGradedBCM}, we extend our results to a relative context where $X$ is merely assumed to be proper over an affine scheme that is finitely presented over a $p$-adic DVR. Finally, in \S \ref{GlobalSABig}, the assumption that $L$ is ample is relaxed to merely requiring $L$ to be semiample and big.

\subsection{The main geometric theorem}
\label{ss:MainThmGraded}

\begin{notation}
Let $V$ be an excellent henselian $p$-torsionfree DVR with residue field $k$ of characteristic $p$. Let $\overline{V}$ be an absolute integral closure of $V$, so $\overline{V}$ is a rank $1$ valuation ring over $V$ with residue field $\overline{k}$ being an algebraic closure of $k$; let $C = \widehat{V}[1/p]$ be the completed fraction field of $\widehat{V}$. Let $X/\overline{V}$ be a flat integral proper scheme of relative dimension $d$ equipped with an ample line bundle $L$.

\end{notation}

Our goal is to prove the following theorem:

\begin{theorem}[Kodaira vanishing up to finite covers]
\label{KodairaFiniteCover}
There is a finite surjective map $\pi:Y \to X$ such that the following pullback maps are $0$:
\begin{enumerate}
\item $H^{> 0}(X_{p=0}, L^a) \to H^{> 0}(Y_{p=0}, \pi^* L^a)$ for $a \geq 0$.
\item $H^{< d}(X_{p=0}, L^b) \to H^{< d}(Y_{p=0}, \pi^* L^b)$ for $b < 0$.
\end{enumerate}
In particular, there exist such a $\pi$ with $\pi^*:H^*(X, L^{-b})_{tors} \to H^*(Y, \pi^* L^{-b})_{tors}$ being $0$ for any fixed $b < 0$.
\end{theorem}

\begin{remark}
Remark~\ref{rmk:ProperDVR} extends Theorem~\ref{KodairaFiniteCover} by relaxing the positivity assumption on $L$. 
\end{remark}

\subsection{Reformulation via graded absolute integral closures}
\label{ss:GAIC}

We shall prove Theorem~\ref{KodairaFiniteCover} by reinterpreting it as a result about the local cohomology of the affine cone over $X$, to which we can apply techniques analogous to those in \S \ref{CMgeomcase}.

\begin{notation}[Affine cones]
For any scheme $Y/\overline{V}$ and line bundle $L$, write 
\[ U(Y,L) := \mathrm{Spec}_Y(\bigoplus_{n \in \mathbf{Z}} L^n) \to Y \quad \text{and} \quad T(Y,L) := \mathrm{Spec}_Y(\bigoplus_{n \in \mathbf{Z}_{\geq 0}} L^n) \to Y\]
for the $\mathbf{G}_m$-torsor and the total space of the line bundle $L^{-1}$ Thus, we have an open immersion $U(Y,L) \subset T(Y,L)$ with complement given by the $0$ section of $T(Y,L) \to Y$. Write $R(Y,L) = \bigoplus_{n \in \mathbf{Z}_{\geq 0}} H^0(Y,L^n)$ for the homogeneous co-ordinate ring of $L$ and write $C(Y,L) = \mathrm{Spec}(R(Y,L))$. We have a canonical affinization map $T(Y,L) \to C(Y,L)$. Moreover, if $Y/\overline{V}$ is a flat projective scheme and $L$ is ample, then $C(Y,L)$ is a finitely presented flat affine $\overline{V}$-scheme. Thus, we obtain a  diagram
\[ U(Y,L) \subset T(Y,L) \to C(Y,L)\]
where the first map and the composite are open immersions, and the second map is a proper birational map that is an isomorphism over $U(Y,L) \subset C(Y,L)$ (which is also the complement of the origin of $C(Y,L)$, corresponding to the ideal $R(Y,L)_{\geq 1} \subset R(Y,L)$ of elements of degree $\geq 1$). 
\end{notation}

\begin{lemma}
\label{FlatZarAIC}
If $R$ is a normal domain with algebraically closed fraction field, then any fppf map $R \to S$ admits a section Zariski locally on $R$. 
\end{lemma}
\begin{proof}
This is a variant of \cite[Lemma 3]{GabberAffineAnalog}. By finite presentation considerations, we may assume $R$ is local, whence $R$ is henselian by absolute integral closedness. Any fppf map map $R \to S$ admits a refinement $R \to S \to S'$ with $R \to S'$ being fppf and quasi-finite. As $R$ is henselian, we have $S' \simeq S_1 \times S_2$, where $R \to S_1$ is finite flat and $\mathrm{Spec}(S_2) \to \mathrm{Spec}(R)$ misses the closed point. It is now enough to observe that $R \to S_1$ admits a section: indeed, any finite cover $R \to S$ admits a section, as $R$ maps isomorphically to $S/\mathfrak{p}$ for any prime $\mathfrak{p}$ of $S$ above the generic point of $\mathrm{Spec}(R)$.
\end{proof}

\begin{lemma}
\label{AICFlatCoh}
Let $Y$ be a integral normal scheme with algebraically closed fraction field. Then $H^i(Y,\widehat{\mathbf{Z}}(1)) = 0$ for $i > 0$. In particular, $\mathrm{Pic}(Y)$ is a $\mathbf{Q}$-vector space. (Here $\widehat{\mathbf{Z}}(1) = \lim_n \mu_n$ is an affine group scheme, and we define $R\Gamma(X, \widehat{\mathbf{Z}}(1))$ as the fpqc cohomology or equivalently as $R\lim R\Gamma_{fppf}(X, \mu_n)$.)
\end{lemma}
\begin{proof}
It is enough to prove that $H^0(Y,\mu_n) \to R\Gamma(Y,\mu_n)$ is an isomorphism for all $n$, where the cohomology is in the fppf topology. By Lemma~\ref{FlatZarAIC}, we may compute these objects in the Zariski topology instead of the fppf topology. But $\mu_n$ is constant in the Zariski topology of $Y$, so the claim follows from Grothendieck's theorem  \cite[Tag 02UW]{StacksProject} that constant sheaves have trivial cohomology on irreducible spaces. 

Alternately, one could avoid Lemma~\ref{FlatZarAIC} by observing that the restriction of the fppf Kummer sequence to the Zariski topology of $Y$ is still a short exact sequence of sheaves as $\mathbf{G}_m(U) = \mathcal{O}(U)^*$ is a divisible abelian group for all affine opens $U \subset Y$ by the assumption on $Y$. The associated long exact sequence then again reduces ust to Grothendieck's theorem mentioned above.
\end{proof}

For the rest of the section, we fix the following notation:

\begin{notation}[The graded absolute integral closure]
Fix an absolute integral closure $X^+ \to X$ and notation as in Definition~\ref{SSAlt}. Let $\widetilde{\mathbf{G}_m} := \lim_n \mathbf{G}_m \to \mathbf{G}_m$ be the ``profinite universal cover'' and write $\widehat{\mathbf{Z}}(1) = \lim_n \mu_n$ for its kernel, regarded as a profinite group scheme. Fix a compatible system $\{ L^{1/n} \}_{n \in \mathbf{Z}_{\geq 0}}$ of $n$-th roots of $L$ over $X^+$, regarded as a lift the $\mathbf{G}_m$-torsor determined by $L$ along the surjection $\widetilde{\mathbf{G}_m} \to \mathbf{G}_m$; such a choice exists and is unique up to isomorphism as $H^i(X^+, \widehat{\mathbf{Z}}(1))  = 0$ for $i=2,1$ by Lemma~\ref{AICFlatCoh}. Thanks to this choice, we obtain a compatible system of diagrams 
\[ U(X^+,L^{1/n}) \subset T(X^+,L^{1/n}) \to C(X^+, L^{1/n})\] 
indexed by $n \in \mathbf{Z}_{\geq 1}$ under divisibility. The transition maps are integral (and thus affine), so we can take a limit to obtain a diagram
\[ U_\infty \subset T_\infty \to C_\infty\] 
where the first map and the composite map are both open immersions and the second map is a pro-proper map that is an an affinization and isomorphism over $U_\infty \subset C_\infty$. We shall regard all these schemes as living over $C(X,L)$ and all sheaves are implicitly pushed forward to $C(X,L)$ unless otherwise specified. 

We give the following names to the following rings/ideals resulting from this construction:
\begin{enumerate}
\item $R := R(X,L) = \bigoplus_{n \in \mathbf{Z}_{\geq 0}} H^0(X,L^n)$ viewed as a graded $\overline{V}$-algebra.
\item The homogeneous ideals $\widetilde{\mathfrak{m}} = R_{\geq 1}$ and $\mathfrak{m} := (p, R_{\geq 1}) \subset R$, so $R/\widetilde{\mathfrak{m}} \to R/\mathfrak{m}$ identifies with $\overline{V} \to \overline{V}/p$. 
\item $R^{+,gr} := \Gamma(C(X^+,L), \mathcal{O}) = R(X^+,L) =  \colim_{Y \in \mathcal{P}_X^{fin}} R(Y, L) = \bigoplus_{n \in \mathbf{Z}_{\geq 0}} H^0(X^+,L^n)$.
\item $R^{+,GR} := \Gamma(C_\infty, \mathcal{O}) = \colim_n R(X^+,L^{1/n}) \simeq \bigoplus_{n \in \mathbf{Q}_{\geq 0}} H^0(X^+, L^n)$.
\end{enumerate}
Thus, we have integral maps $R \to R^{+,gr} \to R^{+,GR}$ of rings with the second map being a direct summand as graded $R^{+,gr}$-modules. Note that all these rings are flat over $\overline{V}$ with relative dimension $d+1$, and hence have Krull dimension $d+2$. 
\end{notation}

In the above notation, the graded analog of Theorem~\ref{AICVanishLocalCoh} is  the following result; this implies Theorem~\ref{KodairaFiniteCover} and its proof shall occupy most  of this section.

\begin{theorem}
\label{GAICCM}
We have $H^i_{\mathfrak{m}}(R^{+,gr}) = H^i_{\mathfrak{m}}(R^{+,GR}) = 0$ for $i < d+2$.
\end{theorem}

\begin{remark}[$R^{+,gr}$ vs $R^{+,GR}$]
In applications of Theorem~\ref{GAICCM}, the case of $R^{+,gr}$ is the most relevant one (e.g., see Theorem~\ref{VanishingFiniteCov}). However, for our method of proof (which relies on perfectoid methods), it is crucial to pass to the larger ring $R^{+,GR}$. Roughly speaking, the reason is that $R^{+,gr}$ is not ramified enough: its punctured spectrum is a $\mathbf{G}_m$-torsor over $X^+$ (which is perfectoid after $p$-completion), so the $p$-completion cannot be perfectoid as the fibres of $\mathrm{Spec}(R^{+,gr}/p) - V(\mathfrak{m}) \to X^+_{p=0}$ are not semiperfect.  The passage to $R^{+,GR}$ has the effect of adding infinite ramification in the fibres by turning this $\mathbf{G}_m$-torsor into a $\widetilde{\mathbf{G}_m}$-torsor, resulting in a scheme which is perfectoid on $p$-completion. We refer the reader to Proposition~\ref{AICAmple} for a concrete payoff of passing to this larger ring. For completeness, we also remark that one could also work with slightly smaller (but still highly ramified) subrings of $R^{+,GR}$ -- such as the subring corresponding to degrees in $\mathbf{Z}[1/p]_{\geq 0} \subset \mathbf{Q}_{\geq 0}$ -- without changing the argument.
\end{remark}

\subsection{Almost Cohen--Macaulayness of $R^{+,GR}$}
\label{ss:GAICACM}

In this subsection, we prove $R^{+,GR}$ is almost Cohen--Macaulay by interpreting it through perfectoid geometry (twice) and using the Riemann-Hilbert functor. First, we begin with the analog of Lemma~\ref{AICPerfectoid} (2). 

\begin{lemma}[Perfectoidness of $R^{+,GR}$]
\label{PerfdGAIC}
The ring $R^{+,GR}$ is a normal domain and each homogenous element of $R^{+,GR}$ admits a $p$-th root. In particular, the $p$-completion $(R^{+,GR})^{\wedge}$ is perfectoid.
\end{lemma}
\begin{proof}
For each integer $n \geq 1$, the map $T(X^+,L^{1/n}) \to X^+$ is smooth, so $T(X^+,L^{1/n})$ is a normal integral scheme. The affinization $\overline{U(X^+,L^{1/n})}$ is then a normal integral affine scheme, so $R(X^+,L^{1/n})$ is a normal domain. Taking a limit as $n \to \infty$ then shows that $R^{+,GR}$ is a normal domain. 

To find $p$-th roots, fix a homogeneous element $f \in H^0(X^+, L^n)$ for some $n \in \mathbf{Q}$; we may assume $n \geq 0$ since $f=0$ otherwise. By approximation, there is a finite cover $Y \to X$ factoring $X^+ \to X$ such that $L^n$ descends to a line bundle $M$ on $Y$ and such that $f$ comes from a unique $f' \in H^0(Y, M)$. By enlarging $Y$, we may assume $M = N^p$ is a $p$-th power. The cyclic covering trick then gives a $p$-th root of $f$ in a finite cover of $Y$, and thus also after pullback to $X^+$. 
\end{proof}

In our applications, it will be more convenient to work with $T_\infty$ due to its direction connection to geometry via the map $T_\infty \to X^+$.

\begin{lemma}[Perfectoidness of $T_\infty$]
\label{TotLBPerfd}
The $p$-adic completion of $T_\infty$ is perfectoid.
\end{lemma}
\begin{proof}
This can be checked Zariski locally on $X^+$. Pick $V := \mathrm{Spec}(S) \subset X^+$ a Zariski open such that $L|_V$ is trivial. The $\widetilde{\mathbf{G}_m}$-torsor $\{L^{1/n}\}_{n \in \mathbf{Z}_{\geq 0}}|_V$ is also trivial as $H^1(V, \widehat{\mathbf{Z}}(1)) = 0$ by Lemma~\ref{AICFlatCoh}. Consequently, the restriction $T_\infty|_V \to V$ can be identified as the affine $V$-scheme attached to the $S$-algebra $S[t^{\mathbf{Q}_{\geq 0}}]$. We must show this algebra is perfectoid after $p$-completion, but this is immediate: $S$ is perfectoid after $p$-completion by Lemma~\ref{AICPerfectoid}, and $S \to S[t^{\mathbf{Q}_{\geq 0}}]$ is flat and relatively perfect modulo $p$. 
\end{proof}

\begin{proposition}[Killing cohomology for ample line bundles by finite covers]
\label{AICAmple}
If $f:T_\infty \to C_\infty$ denotes the affinization map, then $\mathbf{F}_{p,C_\infty} \simeq Rf_* \mathbf{F}_{p,T_\infty}$ in $D(C_\infty, \mathbf{F}_p)$ via the natural map.  Consequently, we have $\mathcal{O}_{C_\infty}/p \simeq Rf_* \mathcal{O}_{T_\infty}/p$ and thus the map
\[ R^{+,GR}/p := \bigoplus_{n \in \mathbf{Q}_{\geq 0}} H^0(X^+,L^n)/p \to \bigoplus_{n \in \mathbf{Q}_{\geq 0}} R\Gamma(X^+_{p=0}, L^n)\] 
is an isomorphism.
\end{proposition}
\begin{proof}
Note that $f$ is an isomorphism outside the $0$ section and gives $g:X^+ \to \mathrm{Spec}(\overline{V})$ when pulled back along the $0$ section $\mathrm{Spec}(\overline{V}) \subset C_\infty$. Lemma~\ref{EtaleCohAIC} implies that $\mathbf{F}_p \simeq Rg_* \mathbf{F}_p$ as sheaves on $\mathrm{Spec}(\overline{V})$. By proper base change, it follows that $\mathbf{F}_p \simeq Rf_* \mathbf{F}_p$ on $C_\infty$ as well. The second part then follows from Theorem~\ref{thm:RH} (2) thanks to Lemma~\ref{TotLBPerfd}.
\end{proof}

\begin{remark}
Proposition~\ref{AICAmple} proves the half of Theorem~\ref{KodairaFiniteCover} that corresponds to non-negative powers of $L$. The proof still relies on a non-trivial input from $p$-adic Hodge theory: the proper pushforward compatibility of $\RH_{\overline{\Prism}}$ that ultimately comes from the primitive comparison theorem. It would be interesting to find a proof that avoids this input. 
\end{remark}

The next assertion is the analog of Theorem~\ref{ACMaic} and crucially relies on the fact that we work with $C_\infty = \lim_n C(X^+,L^{1/n})$ rather than $C(X^+,L)$.

\begin{lemma}[Perversity of the constant sheaf on $C_\infty$]
\label{IndPervConstGAIC}
The sheaf $\mathbf{F}_{p,C_\infty[1/p]}[d+1]$ is ind-perverse when regarded as a sheaf on $C(X,L)[1/p]$ via pushforward.
\end{lemma}
\begin{proof}
Write $\eta \in X$ for the generic point; this point has a unique lift $X^+$ that we also call $\eta$. 

We first show that the constant sheaf $\mathbf{F}_{p,T_\infty[1/p]}$ on $T_\infty[1/p]$ is $*$-extended from the pro-open subset given by the generic fibre $U_{\infty,\eta}$ of $U_\infty[1/p] \to X^+[1/p]$. In fact, since the map $T_\infty \to X^+$ is pro-smooth, it follows immediately Theorem~\ref{EtaleCohAIC} and smooth base change that $\mathbf{F}_{p,T_\infty[1/p]}$ on $T_\infty[1/p]$ is $*$-extended from the pro-open subset given by the generic fibre $T_{\infty,\eta}$ of $T_\infty[1/p] \to X^+[1/p]$. It remains to show that the constant sheaf on $T_{\infty,\eta}$ is $*$-extended from $U_{\infty,\eta}$. Unwinding definitions, this amounts to the following observation (that we leave to the reader): given an algebraically closed field $E$ of characteristic $0$, if $Z = \lim_n \mathbf{G}_a = \mathrm{Spec}(E[t^{\mathbf{Q}_{\geq 0}}])$ denotes the inverse limit of copies of the multiplicative monoid $\mathbf{G}_a$ under the maps $x \mapsto x^n$ indexed by $n \in \mathbf{Z}_{\geq 1}$ under divisibility, then the constant sheaf $\mathbf{F}_{p,Z}$ is $*$-extended from the open subset $Z^\circ = \lim_n \mathbf{G}_m = \mathrm{Spec}(E[t^{\mathbf{Q}}]) \subset Z$ given by the same construction for $\mathbf{G}_m$. 

Via Proposition~\ref{AICAmple}, the previous paragraph implies that $\mathbf{F}_{p,C_\infty[1/p]}$ (viewed as a sheaf on $C(X,L)$ via pushforward) is the $*$-extension of the constant sheaf along $U_{\infty,\eta} \subset C_\infty[1/p] \to C(X,L)[1/p]$. Now $U_{\infty,\eta}$ is projective limit of smooth affine curves over $\eta$, and hence a projective limit of smooth affine schemes of dimension $d+1$ over $\overline{V}[1/p]$. Moreover, the map $U_{\infty,\eta} \to C(X,L)[1/p]$ is a projective limit of quasi-finite maps between affine schemes. As pushforward along quasi-finite affine maps preserves perversity by \cite[Corollary 4.1.3]{BBDG}, we conclude by noting that the shifted constant sheaf $\mathbf{F}_{p,Z}[\dim(Z)]$ is perverse for a smooth $\overline{V}[1/p]$-variety $Z$.
\end{proof}

The next lemma is a consequence of the previous one, and shall be used later in the proof of a vanishing result for the local \'etale cohomology of $R^{+,GR}/p$.

\begin{lemma}[Killing cohomology of (anti-)ample line bundles on the special fibre by finite covers]
\label{SpecialFibreGAIC}
Fix $a \in \mathbf{Q}$. If $a \geq 0$, then $H^{>0}(X^+_k, L^a) = 0$; if $a < 0$, then $H^{<d}(X^+_k, L^a) = 0$. 
\end{lemma}
\begin{proof}
Relabelling $X$ if necessary, there are $3$ cases to consider: $a \in \{0,1,-1\}$. 

The cases $a \in \{0,1\}$ follow from Proposition~\ref{AICAmple} by base change, but we give a direct proof below when $a = 1$ for use in the case $a=-1$.

For $a =1$, let us prove the stronger statement that the ind-object $\{H^i(Y_{\overline{k}}, L)\}_{Y \in \mathcal{P}_X^{fin}}$ is $0$ for $i > 0$. Indeed, for any finite $Y \in \mathcal{P}_X^{fin}$, the map $X^+_{\overline{k}} \to Y_{\overline{k}}$ factors over $Y_{\overline{k}}^{\perf} \to Y_{\overline{k}}$ since $X^+_{\overline{k}}$ is perfect. In particular, the map $H^i(Y_{\overline{k}},L) \to H^i(X^+_{\overline{k}},L)$ is $0$ for $i > 0$ since it factors over $H^i(Y_{\overline{k}}^{\perf}, L)$ which vanishes by Serre vanishing; as $H^i(Y_{\overline{k}},L)$ is a finite dimensional $k$-vector space, the claim follows. 

For $a=-1$, via duality, it suffices to show that the pro-system $\{H^i(Y_{\overline{k}}, \omega_{Y_{\overline{k}}}^\bullet \otimes L)\}_{Y \in \mathcal{P}_X^{fin}}$ is pro-zero for $i > -d$ (where $ \omega_{Y_{\overline{k}}}^\bullet$ is the dualizing complex normalized to start in homological degree $d$). Our proof of the Cohen--Macaulayness of $X^+_k$ shows that  $\{\omega_{Y_{\overline{k}}}^\bullet\}_{Y \in \mathcal{P}_X^{fin}} \simeq \{ \omega_Y[d] \}_{Y \in \mathcal{P}_X^{fin}}$ as pro-objects (Remark~\ref{ProCM}). It is thus enough to show that $\{H^i(Y_{\overline{k}}, \omega_{Y_{\overline{k}}} \otimes L)\}_{Y \in \mathcal{P}_X^{fin}}$ is pro-zero for $i > 0$. As in the previous paragraph, by the perfectness of $X^+_{\overline{k}}$, for any $Y \in \mathcal{P}_X$ and any $N \geq 0$, there exists a map $Y' \to Y$ in $\mathcal{P}_X^{fin}$ such that $Y'_{\overline{k}} \to Y_{\overline{k}}$ factors over the $N$-fold relative\footnote{As $\overline{k}$ is perfect, the Frobenius twist functor $Z \mapsto Z^{(1)}$ on $\overline{k}$-schemes is invertible. We write $Z \mapsto Z^{(-1)}$ for the inverse, and $Z \mapsto Z^{(-N)}$ for its $N$-fold composition. The $N$-fold relative Frobenius $Z \to Z^{(N)}$ then yields a map $Z^{(-N)} \to Z$ that we also call the $N$-fold relative Frobenius.} Frobenius $Y^{(-N)}_{\overline{k}} \to Y_{\overline{k}}$. But then the trace map $H^i(Y'_{\overline{k}}, \omega_{Y'_{\overline{k}}} \otimes L) \to  H^i(Y_{\overline{k}}, \omega_{Y_{\overline{k}}} \otimes L)$ factors over $H^i(Y^{(-N)}_{\overline{k}}, \omega_{Y^{(-N)}_{\overline{k}}} \otimes L) \to  H^i(Y_{\overline{k}}, \omega_{Y_{\overline{k}}} \otimes L)$.  Now $H^i(Y^{(-N)}_{\overline{k}}, \omega_{Y^{(-N)}_{\overline{k}}} \otimes L)$ is abstractly (i.e., not $\overline{k}$-linearly) identified with $H^i(Y_k, \omega_{Y_k} \otimes L^{p^N})$ and thus vanishes (by Serre vanishing) for $i > 0$ provided $N \gg 0$. But then the map $H^i(Y'_{\overline{k}}, \omega_{Y'_{\overline{k}}} \otimes L) \to  H^i(Y_{\overline{k}}, \omega_{Y_{\overline{k}}} \otimes L)$ is $0$ for $i > 0$, as wanted. 
 \end{proof}

\begin{proposition}[Almost Cohen--Macaulayness of $R^{+,GR}$]
\label{WeakCMGAIC}
Recall that we have $\widetilde{\mathfrak{m}} = R_{\geq 1} \subset R$. 
\begin{enumerate}
\item $R\Gamma_{\widetilde{\mathfrak{m}}}(R^{+,GR}/p)$ is almost concentrated in degrees $d+1$.
\item $R\Gamma_{\widetilde{\mathfrak{m}}}(R^{+,GR})^{\wedge}$ is concentrated in degree $d+1$ and almost $p$-torsionfree in degree $d+1$.
\item $R\Gamma_{\widetilde{\mathfrak{m}}}(R^{+,GR} \otimes_{\overline{V}}^L \overline{k})$ is concentrated in degree $d+1$. 
\end{enumerate}
\end{proposition}
\begin{proof}
Lemma~\ref{IndPervConstGAIC} gives almost Cohen--Macaulayness of $R^{+,GR}/p$, which gives (1) as well as (2) in the almost category. Lemma~\ref{SpecialFibreGAIC} implies that $R\Gamma_{\widetilde{\mathfrak{m}}}(R^{+,GR} \otimes_V^L \overline{k})$ is concentrated in degree $d+1$, giving (3). It remains to prove that the $p$-completion $M := R\Gamma_{\widetilde{\mathfrak{m}}}(R^{+,GR})^{\wedge}$ is concentrated in degrees $d+1$. But both $M_{\overline{k}}$ and $M_*$ lie in $D^{\geq d+1}$, so the usual fibre sequence 
\[ M \to M_{\overline{k}} \times M_* \to M_* \otimes^L_V {\overline{k}}\]
then implies that $M \in D^{\geq d+1}$. As $I$ is generated up to radicals by $d+1$ elements, we know that $R\Gamma_{\widetilde{\mathfrak{m}}}(R^{+,GR})$ lives in $D^{\leq d+1}$, and hence so does its $p$-completion $M$.
\end{proof}

\begin{remark}
Since $\mathfrak{m} R/p = \widetilde{\mathfrak{m}} R/p$ as ideals of $R/p$, Proposition~\ref{WeakCMGAIC} (2)  implies (via reduction mod $p$) that $H^i_{\mathfrak{m}}(R^{+,GR}/p) = 0$ unless $i=d,d+1$ and that $H^d_{\mathfrak{m}}(R^{+,GR}/p)$ is almost zero. This proves most of Theorem~\ref{GAICCM} except the crucial assertion that $H^d_{\mathfrak{m}}(R^{+,GR}/p)$ is actually $0$. We shall prove this in the next section using the Frobenius in a more serious way.
\end{remark}

\subsection{From almost to honest Cohen--Macaulayness}
\label{ss:GACMtoCM}

We shall prove Theorem~\ref{GAICCM} by following the outline of Theorem~\ref{AICCMMain} with the following two changes. First, instead of approximating $\mathcal{O}_{C_\infty/p}$ via mod $(p,d)$ reductions of prismatic complexes of schemes approximating $C_\infty$, we shall  approximate $\mathcal{O}_{T_\infty}/p$ via mod $(p,d)$ reductions of prismatic complexes of schemes approximating ${T_\infty}$, and then push this description down along $T_\infty \to C_\infty$; in view of Proposition~\ref{AICAmple}, this still gives an approximation of $\mathcal{O}_{C_\infty}/p$. Secondly, as we already know Cohen--Macaulayness of $U_\infty$ (as it is pro-smooth over $X^+$), we can avoid the inductive argument from the proof of Theorem~\ref{AICCMMain}.

 To carry out this argument, we first set up some notation that helps describe $T_\infty$ as a projective limit of varieties obtained as total spaces of line bundles on suitable covers of $X$. 

\begin{construction}[Approximating $T_\infty$ via line bundles on finite covers of $X$]
Define $\mathcal{P}_X$, $\mathcal{P}_X^{fin}$, and $\mathcal{P}_X^{ss}$ exactly as in Definition~\ref{SSAlt}. Regard $\{L^{1/n}\}$ as a compatible system of $n$-th roots of $L$ on the Riemann-Zariski space $\widetilde{X^+} = \lim_{Y \in \mathcal{P}_X} Y$ of $X^+$. For $Y \in \mathcal{P}_X$, write $\pi_Y:\widetilde{X^+} \to Y$ and $f_Y:Y \to X$ for the natural maps. 

For a fixed positive integer $n$, let $\mathcal{Q}_n$ denote the category of those $Y \in \mathcal{P}_X$ equipped with an $n$-th root $L_n$ of $f_Y^* L$ that descends the chosen $n$-th root $L^{1/n}$ over $\widetilde{X^+}$. More precisely, the objects $\mathcal{Q}_n$ are tuples $(Y \in \mathcal{P}_X, L_n \in \mathrm{Pic}(Y),\alpha,\beta)$ where $\alpha$ and $\beta$ are isomorphisms $\alpha:L_n^{\otimes n} \simeq f_Y^* L $ on $Y$ and $\beta:\pi_Y^* L_n \simeq L^{1/n}$ on $\widetilde{X^+}$ such that $\pi_Y^*(\alpha)$  agrees with the chosen isomorphism $(L^{1/n})^{\otimes n} \simeq \pi_X^* L$ under $\beta$; morphisms in this category are given in the obvious way. Thanks to the choice of $\beta$, the situation is very rigid: the map $\mathcal{Q}_n \to \mathcal{P}_X$ is faithful, so $\mathcal{Q}_n$ is a poset. Moreover, for any $(Y,L_n,\alpha,\beta) \in \mathcal{Q}_n$, the slice category $\mathcal{Q}_{n,/(Y,L_n,\alpha,\beta)}$ identifies with the slice category $\mathcal{P}_{X,/Y}$ under the obvious map. In particular, $\mathcal{Q}_{n}$ is cofiltered. In the future, unless there is potential for confusion, we shall simply write $(Y,L_n)$ for an object of $\mathcal{Q}_n$.


For $m \mid n$, there is a natural map $\mathcal{Q}_n \to \mathcal{Q}_m$ given by $(Y,L_n) \mapsto (Y,L_n^{\otimes \frac{n}{m}})$. Thus, we can regard the collection $\{\mathcal{Q}_n\}$ as a family of categories parametrized by the poset $\mathbf{Z}_{\geq 1}$ of positive integers divisibility (i.e., there is a map $n \to m$ exactly when $m \mid n$). Let $\mathcal{Q}$ be the Grothendieck construction of the resulting functor $\mathbf{Z}_{\geq 1} \to \mathrm{Cat}$. Explicitly, an object of $\mathcal{Q}$ is given by a pair $(n \in \mathbf{Z}_{\geq 1}, (Y,L_n) \in \mathcal{Q}_n)$; a map $(n,(Y,L_n)) \to (m,(Z,L_m))$ exists only if $m \mid n$ in which case it is given by a map $(Y,L_n^{\otimes n/m}) \to (Z,L_m)$ in $\mathcal{Q}_m$. One checks that  $\mathcal{Q}$ is cofiltered.

There is a natural functor $\mathcal{Q}_n \to \mathrm{Sch}_{/T(X,L)}$ given by $(Y,L_n) \mapsto T(Y,L_n)$. In fact, using the natural maps $T(Y,L_n) \to T(Y, L_n^{\otimes n/m})$ for $m \mid n$, the preceding assignment naturally extends to a functor $\mathcal{Q} \to \mathrm{Sch}_{/T(X,L)}$. Regard this construction as yielding an pro-object $\{T(Y,L_n)\}_{\mathcal{Q}}$.

Finally, there are evident subcategories $\mathcal{Q}_n^{ss}, \mathcal{Q}_n^{fin} \subset \mathcal{Q}_n$ and $\mathcal{Q}^{ss}, \mathcal{Q}^{fin} \subset \mathcal{Q}$ spanned by those $(Y,L_n)$ with $Y$ lies in $\mathcal{P}_X^{ss}$ or $\mathcal{P}_X^{fin}$ respectively. 
\end{construction}

\begin{remark}[The inverse limit of the spaces parametrized by $\mathcal{Q}$]
\label{QfinQssBC}
The primary reason we introduce the category $\mathcal{Q}$ (and its variants) is that it indexes a collection of understandable finite type $\overline{V}$-schemes approximating $T_\infty$. More precisely, the locally ringed space $\lim_{(Y,L_n) \in \mathcal{Q}^{fin}} T(Y,L_n)$ identifies with $T_\infty$, while  $\lim_{(Y,L_n) \in \mathcal{Q}} T(Y,L_n)$ identifies with the base change of $T_\infty \to X$ along $\widetilde{X^+} \to X^+$. In particular, the map $\lim_{(Y,L_n) \in \mathcal{Q}} T(Y,L_n) \to \lim_{(Y,L_n) \in \mathcal{Q}^{fin}} T(Y,L_n)$ is the base change of $\widetilde{X^+} \to X^+$ along the pro-smooth map $T_\infty \to X^+$. We shall use this description to understand the difference in the cohomologies (both prismatic and coherent) of $\lim_{(Y,L_n) \in \mathcal{Q}} T(Y,L_n)$ and  $\lim_{(Y,L_n) \in \mathcal{Q}^{fin}} T(Y,L_n)$
\end{remark}

\begin{remark}[Log prismatic complex for the spaces parametrized by $\mathcal{Q}$]
Given $(Y,L_n) \in \mathcal{Q}_n$, the map $T(Y,L_n) \to Y$ is a smooth map. In particular, if $Y \in \mathcal{P}_X^{ss}$ is semistable, then $T(Y,L_n)$ is also semistable. Write $\Prism_{T(Y,L_n)}$ for the log prismatic complex of its $p$-completion. This complex satisfies the isogeny theorem as well as the Hodge-Tate comparison as in Theorem~\ref{LogPrismatic}. We shall write $\Prism^n_{T(Y,L_n)}$ for the non-logarithmic derived prismatic complex of $T(Y,L_n)$.
\end{remark}

The next lemma is analogous to Theorem~\ref{IndObjectPairs}.

\begin{lemma}[A collection of ind-isomorphisms]
\label{GAICIndObj}
Let $\mathcal{Q}^{fin} \subset \mathcal{Q} \supset \mathcal{Q}^{ss}$ be the categories above. 
\begin{enumerate}
\item All arrows in the following diagram of ind-objects are isomorphisms
\[ \{ \mathcal{O}_{T(Y,L_n)}/p  \}_{\mathcal{Q}^{fin}} \xrightarrow{a}    \{ \mathcal{O}_{T(Y,L_n)}/p \}_{\mathcal{Q}}  \xleftarrow{b}  \{ H^0(\mathcal{O}_{T(Y,L_n)}/p) \}_{\mathcal{Q}} \xrightarrow{c}  \{S/p\}_{S \subset R^{+,GR}},   \]
where the last object is indexed by finite $R$-subalgebras $S \subset R^{+,GR}$
\item All arrows in the following diagram of ind-objects are isomorphisms
\[ 
\{ \mathcal{O}_{T(Y,L_n)}/p \}_{\mathcal{Q}^{ss}} \xrightarrow{d}  \{ \mathcal{O}_{T(Y,L_n)}/p \}_{\mathcal{Q}} \xrightarrow{e}  \{\Prism^n_{T(Y,L_n)}/(p,d) \}_{\mathcal{Q}} \xrightarrow{f} \{\Prism_{T(Y,L_n)}/(p,d) \}_{\mathcal{Q}} \xleftarrow{g}  \{\Prism_{T(Y,L_n)}/(p,d)\}_{\mathcal{Q}^{ss}} \]
\end{enumerate}
In both cases, the colimit of the corresponding ind-object is $R^{+,GR}/p$.
\end{lemma}
\begin{proof}

(1): All the ind-objects appearing here have coherent terms over $R^{+,GR}/p$ and uniformly bounded cohomological amplitude (as the map $T(Y,L_n) \to C(X,L)$ is proper with fibres of dimension $\leq d$). It is therefore enough to show that all the maps give isomorphisms after taking colimits. Using the base change description in Remark~\ref{QfinQssBC}, the claim for $a$ follows from the fact that $R\Gamma(X^+, \mathcal{O}_{X^+}/p) \simeq R\Gamma(\widetilde{X^+}, \mathcal{O}_{\widetilde{X^+}}/p)$, which follows from Theorem~\ref{KillCohMixed} as $\widetilde{X^+} \to X^+$ is a cofiltered limit of pro-(proper surjective) maps between integral schemes with algebraically closed function fields.  In fact, this reasoning also shows that this complex is concentrated in degree $0$, which gives the isomorphy of $b$ after taking the colimit. Finally, the isomorphy of $c$ after taking the colimit is the content of Proposition~\ref{AICAmple}.

(2): The isomorphy of $d$ and $g$ follows from the cofinality of $\mathcal{Q}^{ss} \subset \mathcal{Q}$, which is proven using de Jong's theorem as in Theorem~\ref{AlterationsExist}. Moreover, there is a natural map $\{ \mathcal{O}_{T(Y,L_n)}/p \}_{\mathcal{Q}^{ss}} \to \{\Prism_{T(Y,L_n)}/(p,d)\}_{\mathcal{Q}^{ss}}$ making the diagram commute by the Hodge-Tate comparison, and this map is also an isomorphism: both sides have coherent terms with uniformly bounded cohomological amplitude (by the Hodge-Tate comparison), and one has an isomorphism after taking colimits by Theorem~\ref{IndObjectPairs} (1) as well as the base change description given in the previous paragraph. It remains to check the claim for $f$, which follows from the argument for the analogous statement Theorem~\ref{IndObjectPairs} (2) which uses Lemma~\ref{FactorizeFrobAIC} (that applies thanks to Lemma~\ref{TotLBPerfd}).
\end{proof}

The following is an analog of Lemma~\ref{CMFactor}.

\begin{lemma}[$T_\infty^\flat$ is ind-CM up to bounded $d$-torsion]
\label{CMFactorGAIC}
There exists a constant $c = c(d)$ such that for any $(n,(Y,L_n)) \in \mathcal{Q}^{ss}$, there is a map there is a map $f:(m,(Z,L_m)) \to (n,(Y,L_n))$ in $\mathcal{Q}^{ss}$ and $K \in D_{comp,qc}(\Prism_R/p)$ with the following properties:
\begin{enumerate}
\item Writing $f^*:\Prism_{T(Y,L_n)/p} \to \Prism_{T(Z,L_m)}/p$ for the pullback, the map $d^c f^*$ factors over $K$ in $D(\Prism_R/p)$.
\item $K/d \in D_{qc}(R/p)$ is cohomologically CM (so $R\Gamma_x( (K/d)_x) \in D^{\geq d+1 - \dim(\overline{\{x\}})}$ for all $x \in \mathrm{Spec}(R/p)$). 
\end{enumerate}
\end{lemma}
\begin{proof}
This is proven as in Lemma~\ref{CMFactor}. Specifically, using Lemma~\ref{IndPervConstGAIC}, we learn the existence of an $f$ as above such that $\Prism_{T(Y,L_n),\perf}/p \to \Prism_{T(Z,L_m),\perf}/p$ factors in the almost category over $K := \RH_{\Prism}(F[-d-1])$ for a perverse $\mathbf{F}_p$-sheaf $F$ on $C(X,L)[1/p]$. One then argues as in Lemma~\ref{CMFactor}, using the isogeny theorem to pass from the perfect prismatic complexes to their imperfect variants.
\end{proof}

Finally, we obtain the analog of Lemma~\ref{EquationCM}.

\begin{proposition}[Killing local \'etale cohomology, aka the ``graded equational lemma'']
\label{LocalEtaleCohGAIC}
 $R\Gamma_{\mathfrak{m}}(R^{+,GR,\flat}, \mathbf{F}_p)$ is concentrated in degree $d+2$. 
\end{proposition}
\begin{proof}
Consider the complex $K := R\Gamma_{\mathfrak{m}}(R^{+,GR,\flat}) \in D(\overline{V}^\flat)$. Formal glueing along the $d$-completion map gives an exact triangle
\[ K \to \widehat{K} \times K[1/d] \to \widehat{K}[1/d].\] 
Now $K[1/d] = 0$ as $K$ is $d$-nilpotent. Applying $(-)^{\phi=1}$ and using Artin-Schreier sequences then gives
\[ R\Gamma_{\mathfrak{m}}(R^{+,GR,\flat}, \mathbf{F}_p) \to R\Gamma_{\mathfrak{m}}(R^{+,GR}_k, \mathbf{F}_p) \to (\widehat{K}[1/d])^{\phi=1},\]
where we use that $\widehat{K}^{\phi=1} \simeq (K/d)^{\phi=1} \simeq (K/d^{1/p^\infty})^{\phi=1}$; here the first isomorphism comes from the topological nilpotence of $\phi$ on $d\widehat{K}$ (see also \cite[Lemma 9.2]{BhattScholzePrism}), while the second isomorphism comes from the insensitivity of $(-)^{\phi=1}$ to nilpotents. Thanks to this triangle and the fact that inverting $d$ kills almost zero objects, it is enough to show that $R\Gamma_{\mathfrak{m}}(R^{+,GR}_k, \mathbf{F}_p)$ is concentrated in degree $d+2$ and that $R\Gamma_{\mathfrak{m}}(R^{+,GR,\flat})^{\wedge}$ is almost concentrated in degree $d+1$. 

For the special fibre, we have a $\phi$-equivariant identification
\[ R\Gamma_{\mathfrak{m}}(R^{+,GR}_k) = \mathrm{fib}( \bigoplus_{n \in \mathbf{Q}_{\geq 0}} H^0(X^+_k, L^n) \to  \bigoplus_{n \in \mathbf{Q}} R\Gamma(X^+_k, L^n) ).\]
Using Lemma~\ref{SpecialFibreGAIC}, this simplifies to give a $\phi$-equivariant isomorphism
\[ R\Gamma_{\mathfrak{m}}(R^{+,GR}_k) = (\bigoplus_{n \in \mathbf{Q}_{< 0}} H^d(X^+_k, L^n))[-(d+1)]  \]
Now $\phi$ is bijective on $X^+_k$ and hence also on the objects above. But then, for grading reasons, there are no $\phi$-invariant elements in $\oplus_{n \in \mathbf{Q}_{< 0}} H^d(X^+_k, L^n)$: the set of degrees that occur in such an element is a finite subset of $\mathbf{Q}_{<0}$ that is stable under multiplying by $p$, but there are no such non-empty subsets. This implies that $R\Gamma_{\mathfrak{m}}(R^{+,GR}_k)^{\phi=1}$ is concentrated in degree $d+2$.

To show $R\Gamma_{\mathfrak{m}}(R^{+,GR,\flat})^{\wedge}$ is almost concentrated in degree $d+1$, it is enough (by completeness) to show  the same for $R\Gamma_{\mathfrak{m}}(R^{+,GR,\flat})^{\wedge}/d \simeq R\Gamma_{\mathfrak{m}}(R^{+,GR}/p)$, but this follows from Proposition~\ref{WeakCMGAIC}. 
\end{proof}

The following is our main theorem in this section, proving Theorem~\ref{KodairaFiniteCover} in particular.

\begin{theorem}[$R^{+,GR}/p$ is CM]
\label{GrAICCM}
Any of the  (isomorphic) ind-objects in $D_{qc}(C(X,L)_{p=0})$ appearing in Lemma~\ref{GAICIndObj} are ind-CM in the sense of Definition~\ref{IndCMDef}. In particular, the colimit $R^{+,GR}/p$ is a cohomologically CM $R/p$-module in the sense of Definition~\ref{CMDerived}.
\end{theorem}

\begin{proof}
The structure of the argument is analogous to the proof of Theorem~\ref{AICCMMain}: using the almost Cohen--Macaulayness result, one shows a finiteness property of the image of a large enough transition map for the ind-object obtained by applying local cohomology functors to $\{\Prism_{T(Y,L_n)}/(p,d)\}_{\mathcal{Q}}$, and then concludes using the graded equational lemma. There is one simplification in the current setting: we do not need to carry out an induction on dimension as the statement outside the cone point already follows from Theorem~\ref{AICCMMain}.

First, we claim that $\{\Prism_{T(Y,L_n)}/(p,d)\}_{\mathcal{Q}}$ is ind-CM over $C(X,L)_{p=0} - V(\mathfrak{m}) = U(X,L)_{p=0}$. This statement only depends on the restriction of $\{\Prism_{T(Y,L_n)}/(p,d)\}_{\mathcal{Q}}$ to $U(X,L)_{p=0} \subset C(X,L)_{p=0}$. Moreover, the inclusion $U(X,L)_{p=0} \subset C(X,L)_{p=0}$ factors as $U(X,L)_{p=0} \subset T(X,L){p=0} \to C(X,L)_{p=0}$ with the second map being an isomorphism of $U(X,L)_{p=0}$. Thus, it is enough to prove the stronger statement that the ind-object $\{f_{(Y,L_n),*} \Prism_{T(Y,L_n)}/(p,d)\}_{\mathcal{Q}}$ in $D_{qc}(T(X,L)_{p=0})$ is ind-CM; here $f_{(Y,L_n)}:T(Y,L_n) \to T(X,L)$ is the natural map. This follows by using Lemma~\ref{GAICIndObj} to switch to the structure sheaf, then using Theorem~\ref{AICCMMain} to conclude that that $\{f_{Y,*} \mathcal{O}_{T(Y,L)}/p\}_{\mathcal{P}_X}\}$ is ind-CM over $T(X,L)_{p=0}$, and finally observing that the natural maps $T(Y,L_n) \to T(Y,L_n^{\otimes n}) = T(Y,L)$ are finite flat for $(Y,L_n) \in \mathcal{Q}$ and $m \mid n$.

%
%
%
%

Thanks to the previous paragraph, it remains to prove that the ind-($\overline{V}^\flat/d$-module) $\{H^i_{\mathfrak{m}}(\Prism_{T(Y,L_n)}/(p,d))\}_{\mathcal{Q}}$ is $0$ for $i < d+1$. By Lemma~\ref{IndCMPuncturedNN} as well as ind-CMness over $U(X,L)_{p=0}$ shown in the previous paragraph, this ind-object is an ind-(finitely presented $\overline{V}^\flat/d$-module) for $i < d+1$, i.e., for each $(n, (Y,L_n)) \in \mathcal{Q}^{ss}$, there is a map $(m,(Z,L_m)) \to (n,(Y,L_n))$ in $\mathcal{Q}^{ss}$ such that the induced map $H^i_{\mathfrak{m}}(\Prism_{T(Y,L_n)}/(p,d)) \to H^i_{\mathfrak{m}}(\Prism_{T(Z,L_m)}/(p,d))$ factors over a finitely presented image $\overline{V}^\flat/d$-module for $i < d+1$. Arguing as in Lemma~\ref{IncreaseFinite}, the same holds true with $d$ replaced by $d^c$ for any fixed $c \geq 1$. 

On the other hand, using Lemma~\ref{CMFactorGAIC}, for any $(n,(Y,L_n)) \in \mathcal{Q}^{ss}$,  there is a map $(m,(Z,L_m)) \to (n,(Y,L_n))$ in $\mathcal{Q}^{ss}$ such that the induced map $H^i_{\mathfrak{m}}(\Prism_{T(Y,L_n)}/p) \to H^i_{\mathfrak{m}}(\Prism_{T(Z,L_m)}/p)$ has image annihilated by $d^c$ for some fixed $c = c(X)$ and $i < d+2$. 

Combining the previous two paragraphs and taking the colimit shows that for all $(n,(Y,L_n)) \in \mathcal{Q}^{ss}$, the image of  $H^i_{\mathfrak{m}}(\Prism_{T(Y,L_n)}/p) \to H^i_{\mathfrak{m}}(R^{+,GR,\flat})$ is contained in a finitely generated $\overline{V}^\flat$-submodule for $i < d+2$. As $H^i_{\mathfrak{m}}(R^{+,GR,\flat})$ is almost zero for $i < d+2$ (Proposition~\ref{WeakCMGAIC}), it is naturally an $\overline{k}$-vector space, and hence  locally noetherian as a $\overline{V}^\flat$-module (as in Remark~\ref{NoetherianPropertyAlmostZero}). In particular, the image of $H^i_{\mathfrak{m}}(\Prism_{T(Y,L_n)}/p) \to H^i_{\mathfrak{m}}(R^{+,GR,\flat})$ is a finite dimensional $\overline{k}$-vector space for $i < d+2$. As this image is also Frobenius stable, one argues as in Lemma~\ref{EquationCM}: this image must be generated by its Frobenius fixed points by finiteness and perfectness, and hence vanishes by Proposition~\ref{LocalEtaleCohGAIC} and Artin-Schreier theory as $i < d+2$. Consequently, the preceding map is $0$. Taking the colimit then shows that  $H^i_{\mathfrak{m}}(R^{+,GR,\flat}) = 0$ for $i < d+2$ as well. By Bockstein sequences, this implies $H^i_{\mathfrak{m}}(R^{+,GR}/p) = 0$ for $i < d+1$ and $H^i_{\mathfrak{m}}(R^{+,GR}) = 0$ for $i < d+2$. Thus, for $i < d+1$, the ind-object $\{H^i_{\mathfrak{m}}(\Prism_{T(Y,L_n)}/(p,d))\}_{\mathcal{Q}}$ has $0$ colimit. As we already saw that this ind-object is an ind-(finitely presented $\overline{V}^\flat/d$-module), a formal argument now shows that this ind-object must be $0$: any ind-(finitely generated module) over a commutative ring with $0$ colimit must be ind-isomorphic to $0$.
\end{proof}

\begin{corollary}
Assume $(X,L)/\overline{V}$ comes via base change from some $(Y,M)/V$. Let $S = \oplus_{n \geq 0} H^0(Y, M^n)$ be the homogenous co-ordinate ring. Then the $S/p$-module $R^{+,GR}/p$ is Cohen--Macaulay.
\end{corollary}
\begin{proof}
The map $\mathrm{Spec}(R^{+,GR}/p) \to \mathrm{Spec}(S/p)$ is surjective, so it remains to check that systems of parameters in the local rings of $S/p$ give regular sequences in the corresponding localization of $R^{+,GR}/p$. This follows from Theorem~\ref{GrAICCM} and Lemma~\ref{RegSeqCrit}.
\end{proof}

\subsection{The relative variant}
\label{rmk:RelativeGradedBCM}

Say $T/\overline{V}$ is a flat finitely presented domain and $X/T$ is proper integral $T$-scheme equipped with a relatively ample line bundle $L$. Assume $X$ is $V$-flat. By choosing absolute integral closures of $T$ and $X$ as well as a compatible system of roots of $L$, one can define the $T$-algebras $R := R(X,L) \subset R^{+,gr} \subset R^{+,GR}$ as in \S \ref{ss:GAIC}.  Let $I = R_{\geq 1} \subset R$ be the irrelevant ideal, so $V(I) \subset \mathrm{Spec}(R) =: C(X,L)$ identifies with $\mathrm{Spec}(T)$ since $R(X,L)/I = H^0(X,\mathcal{O}_X) = T$; write $Z = \mathrm{Spec}(R/I) \subset \mathrm{Spec}(R)$ for the corresponding closed subscheme.  Following the notation from \S \ref{ss:GAIC}, we obtain the following diagram of $T$-schemes:
\[ \xymatrix{ & T_\infty \ar[r] \ar[d] \ar[ld] & C_\infty = \mathrm{Spec}(R^{+,GR}) \ar[d] & \\
X^+ \ar[d] & T(X,L) \ar[r] \ar[ld] & C(X,L) = \mathrm{Spec}(R) \ar[d] & \mathrm{Spec}(R/I) =: Z \ar[l] \ar[ld]^-{\simeq} \\
X \ar[rr] & & \mathrm{Spec}(T). & }\]
Our main result is the following analog of Theorem~\ref{GrAICCM}:

\begin{theorem}
\label{RGrAICCM}
The $R/p$-module $R^{+,GR}/p$ is cohomologically CM.
\end{theorem}

The proof largely follows the outline of Theorem~\ref{GrAICCM}, so we only indicate the necessary modifications.

\begin{proof}
As the statement of the lemma only depends on the pair $(X,L)$, we may replace $T$ with the Stein factorization of $X \to \mathrm{Spec}(T)$ to assume $T = H^0(X,\mathcal{O}_X)$.  Our goal is to prove the following statement for each $x \in \mathrm{Spec}(R/p) \subset \mathrm{Spec}(R) = C(X,L)$:

\begin{itemize}
\item[$(\ast)_x$] We have $H^i_x(R^{+,GR}_x/p) = 0$ for $i < \dim(R_x/p)$
\end{itemize}

We begin by proving $(\ast)_x$ if $x \in U(X,L)_{p=0} = C(X,L)_{p=0} - V(I)_{p=0}$. Indeed, the restriction of $T_\infty \to C(X,L)$ away from $V(I)$ identifies with the map $U_\infty \to U(X,L)$. This map factors as $U_\infty \to U(X,L) \times_X X^+ \to U(X,L)$, where the first map is pro-(finite flat) while the second map is a smooth base change of $\pi:X^+ \to X$. As we have already seen that $\mathcal{O}_{X^+}/p$ is cohomologically CM over $X_{p=0}$, it follows that the same holds true for $\mathcal{O}_{U_\infty}/p$ over $U(X,L)_{p=0}$, which implies $(\ast)_x$ for $x \in U(X,L)_{p=0}$.

It remains to prove $(\ast)_x$ for all $x \in Z_{p=0} \subset \mathrm{Spec}(R/p) = C(X,L)_{p=0}$. We shall prove this claim by induction on $d = \dim(T/p)$. If $d=0$ (whence $T = \overline{V}$), the claim follows from Theorem~\ref{GrAICCM}.  By induction on dimension  (i.e., by a variant of Lemma~\ref{CMProjective}), we may assume that the statement is known to hold true after localization over all non-closed points of $Z_{p=0} = \mathrm{Spec}(T/p)$.

Pick a closed point $x \in Z_{p=0} \subset C(X,L)_{p=0}$, so $\dim(T_x/p) = d$ while $\dim(R_x/p) = \dim(X_{p=0}) + 1$ (as taking the affine cone increases dimension by $1$). We now show $(\ast)_x$ by mimicing the argument in Theorem~\ref{GrAICCM}. First, one checks that Lemma~\ref{IndPervConstGAIC} holds true in our current setting with the same proof, and thus $(\ast)_x$ holds true in the almost category by Corollary~\ref{cor:PervACM}. To prove the statement on the nose, following the argument in Theorem~\ref{GrAICCM}, we reduce to checking that 
\[ \left(R\Gamma_x(R^{+,GR}_{\overline{k}})\right)^{\phi=1} \in D^{\geq \dim(X_{p=0})+2}.\] 
Since $x \in Z = V(I)$, we have $R\Gamma_x(R^{+,GR}_{\overline{k}}) = R\Gamma_x R\Gamma_I(R^{+,GR}_{\overline{k}})$. Moreover, since  $\mathrm{Spec}(R/p) \to \mathrm{Spec}(T/p)$ carries $Z_{p=0}$ homeomorphically onto $\mathrm{Spec}(T/p)$, we can (and do) regard $R\Gamma_x R\Gamma_I(R^{+,GR}_{\overline{k}})$ as the local cohomology on $\mathrm{Spec}(T/p)$ at the point $x$ of the pushforward of $R\Gamma_I(R^{+,GR}_{\overline{k}})$ along $\mathrm{Spec}(R/p) \to \mathrm{Spec}(T/p)$. Now observe that we have a $\phi$-equivariant identification
\[ R\Gamma_I(R^{+,GR}_{\overline{k}}) \simeq \bigoplus_{n < 0} R\Gamma(X^+_{\overline{k}}, L^n)[-1]\]
in $D(T_{\overline{k}})$ by a variant of Proposition~\ref{AICAmple}. Applying $R\Gamma_x(-)^{\phi=1}$ gives
\[ \left(R\Gamma_x(R^{+,GR}_{\overline{k}})\right)^{\phi=1} = \left( R\Gamma_x R\Gamma_I(R^{+,GR}_{\overline{k}}) \right)^{\phi=1} = \left( \bigoplus_{n < 0} R\Gamma_x R\Gamma(X^+, L^n)\right)^{\phi=1}  [-1].\]
Lemma~\ref{RelativeGlobalDepth} implies that $R\Gamma_x R\Gamma(X^+,L^n) \in D^{\geq \dim(X_{p=0})}$ for all $n < 0$. Moreover, for grading reasons, there are $\phi$-invariant elements in the lowest cohomology degree as in Lemma~\ref{LocalEtaleCohGAIC}, so we learn that 
\[ \left( \bigoplus_{n < 0} R\Gamma_x R\Gamma(X^+, L^n)\right)^{\phi=1} \in D^{\geq \dim(X_{p=0})+1}\]
whence
\[R\Gamma_x(R^{+,GR}_{\overline{k}}) \in D^{\geq \dim(X_{p=0}) + 2},\] 
as wanted.
\end{proof}

The following statement was used above:

\begin{lemma}
\label{RelativeGlobalDepth}
With notation as above, the $T_{\overline{k}}$-complex $R\Gamma(X^+_{\overline{k}},L^{b})$ has depth $\geq \dim(X_{p=0})$ at all closed points for $b < 0$. 
\end{lemma}
\begin{proof}
We may assume $b=-1$ by renaming $L$. Let $\mathbf{D}_{T_{\overline{k}}}(-)$ be the autoduality of $D^b_{coh}(T_{\overline{k}})$ induced by the relative dualizing complex for $T_{\overline{k}}/\overline{k}$.  By Example~\ref{CMDualConn},  it is enough to show that the pro-object $\{\mathbf{D}_{T_{\overline{k}}}(R\Gamma(Y_{\overline{k}}, L^{-1}))\}$ lies in $D^{\leq -\dim(X_{p=0})}$. But the pushforward compatibility of Grothendieck duality gives an identification
\[ \mathbf{D}_{T_{\overline{k}}}(R\Gamma(Y_{\overline{k}}, L^{-1})) \simeq R\Gamma(Y_{\overline{k}}, L \otimes \omega_{Y_{\overline{k}}}^\bullet).\]
One can now proceed as in the proof of Lemma~\ref{SpecialFibreGAIC}: the pro-system $\{R\Gamma(Y_{\overline{k}}, L \otimes \omega_{Y_{\overline{k}}}^\bullet)\}$ identifies with $\{R\Gamma(Y_{\overline{k}}, L \otimes \omega_{Y_{\overline{k}}}[\dim(X_{\overline{k}})])\}$ by the ind-Cohen--Macaulayness of $X^+_{\overline{k}}$ (see Remark~\ref{EquationCM}), so one then concludes by Serre vanishing and the perfectness of $X^+_{\overline{k}}$ as in Lemma~\ref{SpecialFibreGAIC}.
\end{proof}

\subsection{Consequences for semiample and big line bundles}
\label{GlobalSABig}

Using Theorem~\ref{KillCohMixed}, we can relax the ampleness assumption in Theorem~\ref{RGrAICCM}.

\begin{theorem}[Killing cohomology of semiample (and big) line bundles]
\label{VanishingFiniteCov}
Let $V$ be a $p$-henselian excellent DVR. Let $X \to \mathrm{Spec}(T)$ be a proper surjective map of integral flat finitely presented $V$-schemes. Fix a closed point $x \in \mathrm{Spec}(T)_{p=0}$ and a semiample line bundle $L \in \mathrm{Pic}(X)$.
\begin{enumerate}
\item There exists a finite surjective map $Y \to X$ such that 
\[ \tau^{> 0} R\Gamma(X_{p=0},L^a) \to \tau^{>0} R\Gamma(Y_{p=0}, L^a)\] 
is $0$ all $a \geq 0$.
\item If $L$ is also big, then there exists a finite surjective map $Y \to X$  such that 
\[ R\Gamma_x( R\Gamma(X_{p=0}, L^b)  ) \to R\Gamma_x(R\Gamma(Y_{p=0}, L^b))\] 
is $0$ on $H^i$ for $i < \dim(X_{p=0})$ and $b < 0$. 
\end{enumerate}
\end{theorem}
\begin{proof}
Fix an absolute integral closure $\overline{V}/V$. By passing to suitable irreducible components of base changes to a sufficiently large finite extension of $V$, we may assume that $X_{\overline{V}}$ and $T_{\overline{V}}$ are integral. 

Assume first that $L$ is ample. In this case, we're in the setup of \S \ref{rmk:RelativeGradedBCM}.  Let $\mathcal{P}$ be the category of all finite surjective maps $Y \to X$ of integral schemes equipped with a lift of a fixed geometric generic point of $X$, so $X^+ = \lim_{\mathcal{P}} Y$ is an absolute integral closure of $X$. The proof of Theorem~\ref{RGrAICCM} (more precisely, the analog of Proposition~\ref{AICAmple}) immediately yields (1); here we observe that the statement in (1) is non-trivial only for finitely many values of $a$ by Serre vanishing, so we may use the coherence of each $R\Gamma(X_{p=0}, L^a)$ to reduce checking (1) to the analogous statement with $Y$ with replaced by $X^+$.  For (2), note that Theorem~\ref{RGrAICCM} and Lemma~\ref{IndCMcCM} (as well as a standard filtration argument to pass from a statement on cohomology groups to a statement at the level of complexes, see \cite[Lemma 3.2]{ddscposchar}) show that the ind-object 
\[ \{R\Gamma_x (R(Y_{p=0},L))\}_{\mathcal{P}} = \{R\Gamma_x\left(\left(\bigoplus_{a \geq 0} \tau^{>0} R\Gamma(Y_{p=0}, L^a)\right) \oplus \left(\bigoplus_{b < 0} R\Gamma(Y_{p=0}, L^b) \right)  \right) [-1] \ \}_{\mathcal{P}} \]
lies in $D^{\geq \dim(X_{p=0})+1}$; here the left side denotes the local cohomology at $x \in \mathrm{Spec}(T)_{p=0} \subset \mathrm{Spec}(R(X,L))_{p=0}$ of the finite $R(X,L)/p$-module $R(Y_{p=0},L)$, while the right side denotes the local cohomology on $\mathrm{Spec}(T)_{p=0}$ of the displayed complex (with equality arising as in the proof of Theorem~\ref{RGrAICCM}).  Shifting to the left by $1$ and extracting the summands corresponding to $b< 0$ then gives the desired claim in (2).

We now deduce the general case of the theorem from the ample case using Theorem~\ref{KillCohMixed}; this argument is analogous to one in \cite[Proposition 7.3]{ddscposchar}. As $L$ is semiample, there exists an integer $n \geq 1$ such that $L^{\otimes n}$ is globally generated. We can then find a proper surjective map $f:X \to Z$ such that $Z$ is an integral proper integral $T$-scheme which is $V$-flat and such that $L^{\otimes n} = f^* M$ for an ample line bundle $M$ on $Z$. Choose  absolute integral closures $X^+ \to X$ and $Z^+ \to Z$ as well as a map $f^+:X^+ \to Z^+$ lifting $f$. Then $L^{\otimes n}|_{X^+} = f^{+,*}(M|_{Z^+})$. As $\mathrm{Pic}(X^+)$ and $\mathrm{Pic}(Z^+)$ are both uniquely divisible (Lemma~\ref{AICFlatCoh}), there is a unique $n$-th root $N$ of $M|_{Z^+}$ such that $L = f^{+,*} N$. Note that $N$ is the pullback of an ample line bundle on a finite cover of $Z$ factoring $Z^+ \to Z$. Replacing $X$ and $Z$ with suitable finite covers if necessary, we may then assume $L = f^* N$. Applying Theorem~\ref{KillCohMixed} to $X \to Z$ yields a diagram 
\[ \xymatrix{ X' \ar[r] \ar[d] & X \ar[d]^-f \\
			Z' \ar[r]^h & Z }\]
where the horizontal maps are finite covers by $V$-flat integral schemes, and such that the resulting square on pushforwards
\[ \xymatrix{ \mathcal{O}_Z/p \ar[r] \ar[d] & \mathcal{O}_{Z'}/p \ar[d] \\
		Rf_* \mathcal{O}_X/p \ar[r] \ar@{..>}[ru]^s & Rg_* \mathcal{O}_{X'}/p }\]
admits a dotted arrow labelled $s$ making the diagram commute (here $g:X' \to Z$ is the structure map for the first square above). By the projection formula, for any $E \in D_{qc}(Z)$, this gives a commutative diagram
\[ \xymatrix{ R\Gamma(Z_{p=0},E) \ar[r] \ar[d] & R\Gamma(Z'_{p=0},E) \ar[d] \\
		R\Gamma(X_{p=0},E) \ar[r] \ar@{..>}[ru]^s & R\Gamma(X'_{p=0},E). }\]
It is now easy to see that (1) for the pair $(X,L)$ follows from (1) for the pair $(Z',N)$ by taking $E = \oplus_{a \geq 0} N^a$. Similarly, (2) also follows by taking $E = \oplus_{b < 0} N^b$ and observing that $\dim(X_{p=0}) = \dim(Z_{p=0})$ if $L$ is big.
\end{proof}

\begin{remark}[The case of proper $V$-schemes]
\label{rmk:ProperDVR}
An important special case of Theorem~\ref{VanishingFiniteCov} concerns the case $T=V$. In this case, $\mathrm{Spec}(T/p)_{p=0}$ has a single point. Consequently, Theorem~\ref{VanishingFiniteCov} yields to the following assertion: for any proper integral $V$-scheme $X$ and any semiample and big line bundle $L$ on $X$, there exists a finite surjective map $Y \to X$ such that the maps 
\begin{enumerate}
\item $\tau^{> 0} R\Gamma(X_{p=0}, L^a) \to \tau^{>0} R\Gamma(Y_{p=0}, L^a)$ for all $a \geq 0$, and
\item $\tau^{< d} R\Gamma(X_{p=0}, L^b) \to \tau^{>d} R\Gamma(Y_{p=0}, L^b)$ for all $b < 0$, where $d = \dim(X_{p=0})$
\end{enumerate}
are $0$. As each $R\Gamma(X,L^n)$ has bounded $p$-power torsion, Bockstein sequences then show that we can also arrange that
\begin{enumerate}
\item[$(2')$] $H^*(X, L^b)_{tors} \to H^*(Y, L^b)_{tors}$ for fixed $b < 0$
\end{enumerate}
is the $0$ map; here we  use that $X$ and $X_{p=0}$ have the same coherent cohomological amplitude to relate the $p$-torsion in $H^d(X,L^b)$ to $H^{d-1}(X_{p=0}, L^b)$.  In particular,  this extends Theorem~\ref{KodairaFiniteCover} to the case of semiample and big line bundles, and thus proves Theorem~\ref{XPlusCMIntro}.
\end{remark}

\newpage
\bibliographystyle{amsalpha}
\bibliography{mybib}

\end{document}